\documentclass[11pt]{article}%
\usepackage{amsfonts}
\usepackage{amsmath}
\usepackage{amssymb}
\usepackage{graphicx}
\usepackage{epsfig}
\usepackage{enumerate}
\usepackage{color}
\usepackage{hyperref}
\providecommand{\U}[1]{\protect\rule{.1in}{.1in}}
\newtheorem{theorem}{Theorem}

\newtheorem{definition}[theorem]{Definition}

\newtheorem{lemma}[theorem]{Lemma}

\newtheorem{proposition}[theorem]{Proposition}
\newtheorem{remark}[theorem]{Remark}

\newenvironment{proof}[1][Proof]{\noindent\textbf{#1.} }{\ \rule{0.5em}{0.5em}}

\renewcommand{\thefootnote}{\fnsymbol{footnote}}
\begin{document}

\author{Tianyang Nie$\;^{\ast}$\bigskip\\{\small Laboratoire de Math\'{e}matiques, Universit\'{e} de Bretagne Occidentale,}\\{\small 29285 Brest C\'{e}dex 3, France}\\{\small ~School of Mathematics, Shandong University, Jinan, Shandong
250100, China}}
\title{A stochastic approach to a new type of parabolic variational inequalities}
\maketitle
\date{}

\begin{abstract}
We study the following quasilinear partial differential equation with two subdifferential operators:
\[
\left\{
\begin{array}
[c]{l}
{\displaystyle\frac{\partial u}{\partial s}(s,x)}+(\mathcal{L}u)(s,x,u(s,x),(\nabla u(s,x))^{\ast}\sigma(s,x,u(s,x)))\medskip\\
\qquad+f(s,x,u(s,x),(\nabla u(s,x))^{\ast}\sigma(s,x,u(s,x)))\in \partial\varphi(u(s,x))+\langle \partial\psi(x),\nabla u(s,x) \rangle, \medskip\\
\qquad\qquad\qquad\qquad\qquad\qquad\qquad\qquad\qquad\qquad\qquad\qquad (s,x) \in[0,T]\times Dom\psi,\medskip\\
u(T,x)=g(x),\quad x\in Dom\psi,
\end{array}
\right.
\]
where for $u\in C^{1,2}\big([0,T]\times Dom\psi\big)$ and $(s,x,y,z)\in [0,T]\times Dom\psi\times Dom\varphi\times\mathbb{R}^{1\times d}$,
\[(\mathcal{L}u)(s,x,y,z):=\frac{1}{2}\sum_{i,j=1}^{n}(\sigma\sigma^{\ast})_{i,j}(s,x,y)\frac{\partial^{2}u}{\partial x_{i}\partial x_{j}}(s,x)
+\sum_{i=1}^{n}b_{i}(s,x,y,z)\frac{\partial u}{\partial x_{i}}(s,x).
\]
The operator $\partial\psi$ (resp. $\partial\varphi$) is the subdifferential of the
convex lower semicontinuous function $\psi:\mathbb{R}^{n}\rightarrow
(-\infty,+\infty]$ (resp. $\varphi:\mathbb{R}\rightarrow(-\infty,+\infty]$).

We define the viscosity solution for such kind of partial differential equations and prove the uniqueness of the viscosity solutions when $\sigma$ does not depend on $y$. To prove the existence of a viscosity solution, a stochastic representation formula of Feymann-Kac type will be developed. For this end, we investigate a fully coupled forward-backward stochastic variational inequality.
\end{abstract}

\footnotetext[1]{{\scriptsize The work of this author is supported
by the Marie Curie ITN Project, \textquotedblleft Deterministic and
Stochastic Controlled Systems and Application\textquotedblright,
FP7-PEOPLE-2007-1-1-ITN, No. 213841-2.}}
\renewcommand{\thefootnote}{\fnsymbol{footnote}}
\footnotetext{\textit{{\scriptsize E-mail addresses:}}
{\scriptsize nietianyang@163.com}}

\medskip


\section{Introduction}
Crandall and Lions introduced the notion of viscosity solution in \cite{CL-83}, and in the later work of Crandall, Ishii and Lions \cite{CIL-92}, they gave a systematically investigation of the viscosity solution for second order partial differential equations (PDEs), which provides a powerful tool to study PDEs and related problems. Pardoux and Peng were the first to give a stochastic interpretation
for the viscosity solutions of semilinear PDEs (see \cite{PP-92} and \cite{P-91}) via their original work on nonlinear backward stochastic differential equations (BSDEs), \cite{PP-90}. This relation between BSDEs and PDEs was investigated by different authors. Let us emphasize that Pardoux and Tang \cite{PT-99} studied the link between the solution of fully coupled forward-backward stochastic differential equations (FBSDEs) and the associated quasilinear parabolic PDEs. El Karoui, Kapoudjian, Pardoux, Peng and Quenez \cite{EKPPQ-97} studied reflected BSDEs in one dimensional and by combing it with a forward stochastic differential equation (SDE), they gave a probabilistic interpretation to the viscosity solution of the related obstacle problem for a parabolic PDE. As a generalisation, Cvitani\'{c} and Ma \cite{MC-01} studied reflected FBSDEs and used them to give a probabilistic interpretation for the viscosity solution of quasilinear variational inequalities with a Neumann boundary condition.

On the other hand, related with multi-dimensional reflected SDEs and BSDEs, stochastic variational inequalities (SVIs) were considered by Bensoussan and R\u{a}\c{s}canu in \cite{BR-94,BR-97}, Asiminoaei and R\u{a}\c{s}canu \cite{AR-97} (For more details, the reader is referred to \cite{PR-11}); BSDEs with subdifferential operators (which are called backward stochastic variational inequalities,
BSVIs) were studied by Pardoux and R\u{a}\c{s}canu \cite{PR-98}. Moreover, the authors of \cite{PR-98} obtained a  generalized Feymann-Kac type formula, which gives a probabilistic interpretation for the viscosity solution of parabolic variational inequalities (PVIs). Maticiuc, Pardoux, R\u{a}\c{s}canu and Z\v{a}linescu \cite{MPRZ-10} extended such PVIs to systems of PVIs. In our paper, motivated by \cite{PR-98} and \cite{Z-02}, we consider the following type of PVI
\[
\left\{
\begin{array}
[c]{l}%
{\displaystyle\frac{\partial u}{\partial s}(s,x)}+(\mathcal{L}u)(s,x,u(s,x),(\nabla u(s,x))^{\ast}\sigma(s,x,u(s,x)))\medskip\\
\qquad+f(s,x,u(s,x),(\nabla u(s,x))^{\ast}\sigma(s,x,u(s,x)))\in \partial\varphi(u(s,x))+\langle \partial\psi(x),\nabla u(s,x) \rangle, \medskip\\
\qquad\qquad\qquad\qquad\qquad\qquad\qquad\qquad\qquad\qquad\qquad\qquad (s,x) \in[0,T]\times Dom\psi,\medskip\\
u(T,x)=g(x),\quad x\in Dom\psi.
\end{array}
\right.
\]
This type of PVI is new since it is driven by two subdifferential operators, one operating over the state and the other operating in the domain and perturbing the direction of the gradient. We define the viscosity solution of such kind of PVIs and prove the uniqueness of the viscosity solutions when $\sigma$ does not depend on $y$. Indeed, by extending and adapting the approaches of Barles, Buckdahn and Pardoux \cite{BBP-97} and Cvitani\'{c} and Ma \cite{MC-01}, we prove the uniqueness of the viscosity solutions in the class of Lipschitz continuous functions. To prove the existence of a viscosity solution, a stochastic representation formula of Feymann-Kac type will be developed. For this end, we investigate the following general fully coupled FBSDEs with subdifferential operators in both the forward and the backward equations,
\[
\left\{
\begin{array}
[c]{l}%
dX_{t}+\partial\psi(X_{t})dt\ni b(t,X_{t},Y_{t},Z_{t})dt+\sigma(t,X_{t},Y_{t},Z_{t})dB_{t},\medskip\\
-dY_{t}+\partial\varphi(Y_{t})dt\ni f(t,X_{t},Y_{t},Z_{t})dt-Z_{t}dB_{t},\quad t\in[0,T],\medskip\\
X_{0}=x, \quad Y_{T}=g(X_{T}).
\end{array}
\right.
\]
We call this kind equations forward-backward stochastic variational inequalities (FBSVIs). Notice that this type of inequalities includes, as a special case, coupled systems composed of a forward and a backward equation, both reflected at the boundary  of a closed convex set. Such kind of FBSVIs are worthy of investigation themselves.

As concerns the fully coupled forward-backward stochastic differential equations (FBSDEs), a generalization of non-coupled forward-backward systems studied by Pardoux and Peng in \cite{PP-92} and \cite{P-91}, using the contraction mapping method, Antonelli \cite{A-93} was the first to prove the existence and the uniqueness for such equations on a small time interval.
To show the solvability of FBSDEs on arbitrary time interval, Ma, Protter and Yong \cite{MPY-94} introduced  the so-called Four-Step-Scheme,
which was inspired by the pioneering work of Ma and Yong \cite{MY-95}. In their approach, the study of FBSDE reduces to the problem of a certain parabolic PDE. However, for this approach one needs that the coefficients are deterministic and the diffusion coefficient has to be non-degenerate. Based on this approach, Delarue got more general results in \cite{D-02}. Without the above conditions, but with a monotonicity assumption,  Hu and Peng \cite{HP-95} used the continuation method to prove that the FBSDE has a unique adapted solution.  Peng and Wu \cite {PW-99} extended \cite{HP-95} to the multidimensional case, while Yong \cite{Y-97} weakened the monotonicity assumptions. On the other hand, Pardoux and Tang \cite{PT-99} obtained the solvability of the FBSDE under some natural monotonicity conditions different from those in \cite{HP-95} and \cite{Y-97}, by using the contraction mapping method. Moreover, they studied the connection between the solution of FBSDEs and associated quasilinear parabolic PDEs. Recently, Zhang \cite{Z-06} introduced a new approach and new general conditions to get the wellposedness of FBSDEs via the induction method and Ma, Wu, Zhang, Zhang \cite{MWZZ-11} found a unified scheme to show the wellposedness of the FBSDEs in a general non-Markovian framework. In the spirit of Pardoux and Tang \cite{PT-99}, Cvitani\'{c} and Ma \cite{MC-01} studied reflected FBSDEs and used them to give a probabilistic interpretation for the viscosity solution of quasilinear variational inequalities with a Neumann boundary condition.

In our paper, we will prove the existence and the uniqueness for FBSVIs, i.e., for coupled systems composed of a forward SVI and a BSVI. Unlike \cite{MC-01}, our FBSVI is more general. Indeed, our FBSVIs cover the case of reflected FBSDEs, where the reflection of the forward as well as the backward equation takes place at the border of closed convex sets. In addition, the backward equation in our case can be multidimensional. Compared with \cite{PR-98}, our FBSVI is fully coupled and the forward equation also includes a subdifferential operator, which induce some difficulties. Indeed, we study the penalized FBSDE using Yosida approximation for lower semicontinuous (l.s.c.) functions to approach our FBSVI. $L^{p}$-estimates for the solution of the penalized FBSDE on the whole interval are necessary (see the proof of our Proposition \ref{Boundedness proposition of two operators} and \ref{Lipschitz proposition of two operators}). However, the method in Cvitani\'{c} and Ma \cite{MC-01} can only give $L^{2}$-estimates. Consequently, we should adapt the induction method introduced by Delarue in \cite{D-02} to obtain the $L^{p}$-estimates in our framework (see Proposition \ref{proposition for high order estimates}).

Moreover, we will prove that the function $u$ defined through the solution of our FBSVI (see (\ref{definition of u})) is a viscosity solution of our new kind of quasilinear PVI. But because of the existence of the subdifferential operators, the continuity of $u$ is not obvious at all. Therefore, we give a detailed proof in Proposition \ref{Property of u}. For this, we separate the proof into two steps: To prove that $u$ is right continuous w.r.t. $t$ and continuous w.r.t. $x$ in step 1, as well as left continuous w.r.t. $t$ and continuous w.r.t. $x$ in step 2.

The paper is organized  as follows: In Section 2 we formulate the problem and we give the definition of the viscosity solution of our new kind of PVIs. Section 3 is devoted to prove the uniqueness of the viscosity solutions of PVIs.  In order to show the existence of the viscosity solution, in Section 4, we study general FBSVIs. More precisely, we obtain the existence and uniqueness of the solutions of FBSVIs. To do this, in subsection 4.1, we give the assumptions. In subsection 4.2, we introduce the penalized FBSDEs which are got by Yosida approximation for the subdifferential operators of the FBSVIs. Moreover, a priori estimates are established. Subsection 4.3 is devoted to the uniform $L^{p}$-estimates of the penalized FBSDEs. For this, we establish first the $L^{p}$-estimates on a small time interval and then we extend them to the whole interval by adapting the method developed by Delarue \cite{D-02}. In subsection 4.4, based on the classical estimates for solving forward SVIs and BSVIs (see Proposition \ref{Boundedness proposition of two operators} and \ref{Lipschitz proposition of two operators}), we prove the existence and the
uniqueness of the solution of FBSVIs. In Section 5, we prove that the function $u$ defined as in (\ref{definition of u}) is the viscosity solution of such PVI. Finally, in the appendix, we provide a priori estimates for the penalized FBSDEs and some auxiliary results.
\section{Formulation of the problem}
We consider the following quasilinear PVI:
\begin{equation}\label{PVI}
\left\{
\begin{array}
[c]{l}%
{\displaystyle\frac{\partial u}{\partial s}(s,x)}+(\mathcal{L}u)(s,x,u(s,x),(\nabla u(s,x))^{\ast}\sigma(s,x,u(s,x)))\medskip\\
\qquad+f(s,x,u(s,x),(\nabla u(s,x))^{\ast}\sigma(s,x,u(s,x)))\in \partial\varphi(u(s,x))+\langle \partial\psi(x),\nabla u(s,x) \rangle, \medskip\\
\qquad\qquad\qquad\qquad\qquad\qquad\qquad\qquad\qquad\qquad\qquad\qquad (s,x) \in[0,T]\times Dom\psi,\medskip\\
u(T,x)=g(x),\quad x\in Dom\psi,
\end{array}
\right.
\end{equation}
where the operator $\mathcal{L}$ is defined by
\[(\mathcal{L}v)(s,x,y,z):=\frac{1}{2}\sum_{i,j=1}^{n}(\sigma\sigma^{\ast})_{i,j}(s,x,y)\frac{\partial^{2}v}{\partial x_{i}\partial x_{j}}(s,x)
+\sum_{i=1}^{n}b_{i}(s,x,y,z)\frac{\partial v}{\partial x_{i}}(s,x).
\]
for $v\in C^{1,2}\big([0,T]\times \mathbb{R}^{n})$.
The functions $b:[0,T]\times\mathbb{R}^{n}\times\mathbb{R}\times\mathbb{R}^{1\times d}\rightarrow \mathbb{R}^{n}$,
$\sigma:[0,T]\times\mathbb{R}^{n}\times\mathbb{R}\rightarrow \mathbb{R}^{n\times d}$,
$f:[0,T]\times\mathbb{R}^{n}\times\mathbb{R}\times\mathbb{R}^{1\times d}\rightarrow \mathbb{R}$ and
$g:\mathbb{R}^{n}\rightarrow \mathbb{R}$  are jointly continuous.
The operator $\partial\psi$ (resp. $\partial\varphi$) is the subdifferential of function $\psi$ (resp. $\varphi$) which satisfies:
\begin{itemize}
\item[$\left(  H^{\prime}_{1}\right)  $] The function $\psi:\mathbb{R}^{n}\rightarrow(-\infty,+\infty]$ is convex l.s.c.
with $0\in Int(Dom\psi)$ and $\psi(z)\geq\psi(0)=0$, for all $z\in\mathbb{R}^{n}$.

\item[$\left(  H^{\prime}_{2}\right)  $] The function $\varphi:\mathbb{R}\rightarrow(-\infty,+\infty]$ is convex l.s.c. such that
$\varphi(y)\geq\varphi(0)=0$ for all $y\in\mathbb{R}$.
\end{itemize}
We put
\[
\begin{array}
[c]{l}
Dom~\psi=\{u\in\mathbb{R}^{n}:\psi(u)<\infty\},\medskip\\
\partial\psi(u)=\{u^{\ast}\in\mathbb{R}^{n}:\langle u^{\ast},v-u\rangle+\psi(u)\leq
\varphi(v),\; v\in\mathbb{R}^{n}\},\medskip\\
Dom(\partial\psi)=\{u\in\mathbb{R}^{n}:\partial\psi(u)\neq\emptyset
\},
\end{array}
\]
and we write
$
(u,u^{\ast})\in\partial\psi \text{ if } u\in Dom(\partial\psi
) \text{ and } u^{\ast}\in\partial\psi(u).
$
Here $\langle \cdot,\cdot\rangle$ denotes the scalar product in $\mathbb{R}^{n}$.

We also mention that the multivalued subdifferential operator $\partial\psi$ is a
monotone operator, i.e. $\langle u^{\ast}-v^{\ast},u-v\rangle\geq0,\;\text{for all }(u,u^{\ast}) ,(v,v^{\ast})\in
\partial\psi.$

\footnotetext[1]{{\scriptsize Let $u\in C\big([0,T]\times Dom\psi\big)$ and  $(t,x)\in [0,T]\times Dom\psi$. Denote by $\mathcal{P}^{2,+}u(t,x)$ the set of triples $(p,q,X)\in \mathbb{R}\times\mathbb{R}^{n}\times S(n)$ ($S(n)$ denotes the set of all $n\times n$ symmetric nonnegative matrices), such that
\[
u(s,y)\leq u(t,x)+p(s-t)+\langle q,y-x\rangle+\frac{1}{2}\langle X(y-x),y-x\rangle
+o\big(|s-t|+|y-x|^{2}\big), ~(s,y)\rightarrow(t,x)
\]
we call $\mathcal{P}^{2,+}u(t,x)$ the parabolic super-jet of $u$ at $(t,x)$. Similarly, we define the parabolic sub-jet of $u$ at $(t,x)$, denoted by $\mathcal{P}^{2,-}u(t,x)$ as the set of triples $(p,q,X)\in \mathbb{R}\times\mathbb{R}^{n}\times S(n)$ such that
\[
u(s,y)\ge u(t,x)+p(s-t)+\langle q,y-x\rangle+\frac{1}{2}\langle X(y-x),y-x\rangle
+o\big(|s-t|+|y-x|^{2}\big),~(s,y)\rightarrow(t,x).
\]}}

Now we give the definition of a viscosity solution of PVI (\ref{PVI}) in the language of sub- and super-jets:
\begin{definition}\label{definition of viscosity by superjet subjet}
Let $u\in C\big([0,T]\times Dom\psi\big)$ satisfies $u(T,x)=g(x),~x\in Dom\psi$. The function $u$ is called a viscosity subsolution (resp. supersolution) of PVI (\ref{PVI}), if for all $(t,x)\in [0,T]\times Dom\psi$, $u(t,x)\in Dom\varphi$ and for any $(p,q,X)\in\mathcal{P}^{2,+}u(t,x)^{\ast}$ (resp. $(p,q,X)\in\mathcal{P}^{2,-}u(t,x)$),
\begin{equation}\label{subsolution}
\begin{array}
[c]{rl}
-p-\frac{1}{2}Tr(\sigma\sigma^{\ast}(t,x,u(t,x))X)-\langle b(t,x,u(t,x),q^{\ast}\sigma(t,x,u(t,x))),q\rangle\medskip\\
-f(t,x,u(t,x),q^{\ast}\sigma(t,x,u(t,x)))\leq -\varphi_{-}^{\prime}(u(t,x))-\partial\psi_{\ast}(x,q)
\end{array}
\end{equation}
(resp.
\begin{equation}\label{supersolution}
\begin{array}
[c]{rl}
-p-\frac{1}{2}Tr(\sigma\sigma^{\ast}(t,x,u(t,x))X)-\langle b(t,x,u(t,x),q^{\ast}\sigma(t,x,u(t,x))),q\rangle\medskip\\
-f(t,x,u(t,x),q^{\ast}\sigma(t,x,u(t,x)))\ge -\varphi_{+}^{\prime}(u(t,x))-\partial\psi^{\ast}(x,q)).
\end{array}
\end{equation}
Here $\varphi_{-}^{\prime}(y)$ (resp. $\varphi_{+}^{\prime}(y)$) denotes the left (resp. right) derivative of $\varphi$ at point $y$, and
\[\partial\psi_{\ast}(x,q):=\liminf\limits_{(x^{\prime},q^{\prime})\rightarrow (x,q),~x^{\ast}\in\partial\psi(x^{\prime})}\langle x^{\ast},q^{\prime}\rangle,\quad (x,q)\in
Dom\psi\times \mathbb{R}^{d},
\]
and $\partial\psi^{\ast}(x,q):=-\partial\psi_{\ast}(x,-q)$ (for $\partial\psi_{\ast}$ and $\partial\psi^{\ast}$, see also \cite{Z-02}).

The function $u$ is called a viscosity solution of PVI (\ref{PVI}) if it is both a viscosity sub- and super-solution.
\end{definition}

\begin{remark}\label{viscosity solution in closed superjet}
Using the definition of $\partial\psi_{\ast}(x,q)$, and the fact that $y\mapsto\varphi_{-}^{\prime}(y)$ is left continuous in $Int(Dom\varphi)(\subset\mathbb{R})$ and increasing in $Dom\varphi$, we see that (\ref{subsolution}) is not only satisfied for $(p,q,X)\in\mathcal{P}^{2,+}u(t,x)$, but also for all
$(p,q,X)\in\overline{\mathcal{P}}^{2,+}u(t,x)$. Similarly, (\ref{supersolution}) holds also for $(p,q,X)\in\overline{\mathcal{P}}^{2,-}u(t,x)$.
\end{remark}
We shall also use the following equivalent definition of a viscosity solution.
\begin{definition}\label{definition of viscosity by smooth function}
Let $u\in C\big([0,T]\times Dom\psi\big)$ satisfies $u(T,x)=g(x), x\in Dom\psi$. The function $u$ is called a viscosity subsolution (resp. supersolution) of PVI (\ref{PVI}), if for all $(t,x)\in [0,T)\times Dom\psi$, $u(t,x)\in Dom(\varphi),$ and if whenever $\Phi\in C^{1,2}([0,T]\times Dom\psi)$ and $(t,x)\in [0,T)\times Dom\psi$ is a local maximum (resp. minimum) point of $u-\Phi$, we have
\[
\begin{array}
[c]{l}
{\displaystyle\frac{\partial \Phi}{\partial s}(t,x)}+(\mathcal{L}u)(t,x,u(t,x),(\nabla \Phi(t,x))^{\ast}\sigma(t,x,u(t,x)))\medskip\\
\qquad+f(t,x,u(t,x),(\nabla \Phi(t,x))^{\ast}\sigma(t,x,u(t,x)))\ge \varphi_{-}^{\prime}(u(t,x))+\partial\psi_{\ast}(x,\nabla \Phi(t,x))
\end{array}
\]
(resp.
\[
\begin{array}
[c]{l}
{\displaystyle\frac{\partial \Phi}{\partial s}(t,x)}+(\mathcal{L}u)(t,x,u(t,x),(\nabla \Phi(t,x))^{\ast}\sigma(t,x,u(t,x)))\medskip\\
\qquad+f(t,x,u(t,x),(\nabla \Phi(t,x))^{\ast}\sigma(t,x,u(t,x)))\leq\varphi_{+}^{\prime}(u(t,x))+\partial\psi^{\ast}(x,\nabla \Phi(t,x))).
\end{array}
\]

Finally, the function $u$ is called a viscosity solution of PVI (\ref{PVI}) if it is both a viscosity sub- and super-solution.
\end{definition}
\section{Uniqueness of viscosity solutions of PVIs}
In this section, we prove the uniqueness of the viscosity solution of PVI (\ref{PVI}) by extending and adapting the approaches of Barles, Buckdahn and Pardoux \cite{BBP-97} and Cvitani\'{c} and Ma \cite{MC-01}.
\begin{theorem}\label{uniqueness of pde}
We assume that $Dom\psi$ is locally compact, $b,\sigma,f,g$ are all jointly continuous. Moreover, suppose that $b,f$ are Lipschitz continuous w.r.t. $(x,y,z)$ and $\sigma(t,x,y)$ does not depend on $y$ as well as Lipschitz continuous w.r.t. $x$. Then under the assumptions of $(H^{\prime}_{1})$ and $(H^{\prime}_{2})$, PVI (\ref{PVI}) has at most one viscosity solution in the class of functions which are Lipschitz continuous in $x$ uniformly w.r.t. $t$ and continuous in $t$.
\begin{remark}
We mention that it is sufficient to show that if $u$ is a subsolution and $v$ is a supersolution such that $u(T,x)=g(x)=v(T,x),~x\in Dom\psi$ and $u,v$ are Lipschitz continuous in $x$ uniformly w.r.t. $t$ and continuous in $t$, then $u\leq v$ for all $(t,x)\in[0,T]\times Dom\psi$.

Since $u$ and $v$ are continuous, we only need to show that $u\leq v$ for all $(t,x)\in(0,T)\times Int(Dom\psi)$. Let us define for each $r>0$ a subset of $Dom\psi$£º
\[(Dom\psi)_{r}:=\{x\in Dom\psi\quad \big|\quad d(x,\partial Dom\psi)\ge r\},
\]
where $d(x,\partial Dom\psi)=\inf\limits_{x^{\prime}\in\partial Dom\psi}|x-x^{\prime}|$.
Let us choose $r_{0}>0$ such that $(Dom\psi)_{r_{0}}\neq\emptyset$ for all $0<r\leq r_{0}$.
Then it suffices to show that for every $0<r\leq r_{0}$, we have $u\leq v$ on $(t,x)\in(0,T)\times(Dom\psi)_{r}$.
\end{remark}
\end{theorem}

Before proving the Theorem \ref{uniqueness of pde}, we recall the following lemma:
\begin{lemma}[Lemma 2.3 \cite{Z-02}]\label{Z 2002 lemma}
Let us denote by $\overline{Dom\psi}$ the closure of $Dom\psi$. We have $\partial\psi_{\ast}(x,q)=\inf_{x^{\ast}\in\partial\psi(x)}\langle x^{\ast},q\rangle$, for $(x,q)\in Int(Dom\psi)\times\mathbb{R}^{n}$ or for $(x,q)\in\partial(Dom\psi)\times\mathbb{R}^{n}$ with $\inf_{n\in N_{\overline{Dom\psi}}(x)}\langle n,q\rangle>0$. Here
\[
N_{\overline{Dom\psi}}(x):=\left\{ x_{\ast}\in\mathbb{R}^{n}\Big||x_{\ast}|^{2}=1,\langle x_{\ast}, z-x\rangle\leq0, \text { for all } z\in \overline{Dom\psi}\right\}.
\]
\end{lemma}

In addition, we need the following lemma generalizing Lemma 7.2 \cite{MC-01}.
\begin{lemma}\label{viscosity subsolution for one kind of pde}
Suppose that the assumptions of Theorem \ref{uniqueness of pde} are satisfied and $u$ is a subsolution and $v$ is a supersolution of (\ref{PVI}) such that $u(T,x)=g(x)=v(T,x),~x\in Dom\psi$. Moreover, we assume that $u,v$ are Lipschitz continuous in $x$, uniformly w.r.t. $t$, and continuous in $t$. Then for all $0<r\leq r_{0}$, the function $\omega:=u-v$ is a viscosity subsolution of the following equation:
\begin{equation}\label{one kind of pde}
\left\{
\begin{array}
[c]{l}min\{\omega,F_{u,v}(t,x,\omega,\omega_{t},D\omega,D^{2}\omega)+\partial\psi_{\ast}(x,D\omega)\}=0,\qquad
 (t,x) \in[0,T)\times (Dom\psi)_{r},\medskip\\
\omega(T,x)=0,\quad x\in (Dom\psi)_{r},
\end{array}
\right.
\end{equation}
where
\begin{equation}\label{F depend on u and v}
F_{u,v}(t,x,r,p,q,X):=-p-\frac{1}{2}Tr\{\sigma\sigma^{\ast}(t,x)X\}-\langle b(t,x,u(t,x),0),q\rangle-\tilde{K}[|r|+|q|\cdot|\sigma(t,x)|].
\end{equation}
Here $\tilde{K}>0$ is a constant depending only on the Lipschitz constants of the function $b,f,u,v$.
\end{lemma}
\begin{proof}
Fix $r\in(0,r_{0})$ and $(t_{0},x_{0})\in [0,T)\times (Dom\psi)_{r}$. Let $\Phi\in C^{1,2}([0,T]\times Dom\psi)$ be such that $(t_{0},x_{0})$ is a maximal point of $\omega-\Phi$ and $\omega(t_{0},x_{0})=\Phi(t_{0},x_{0})$. We assume w.l.o.g. that $(t_{0},x_{0})$ is a strict, global maximum point of $\omega-\Phi$. Moreover, the Lipschitz property of $u$ and $v$ allows to assume that $D\Phi$ is uniformly bounded: $|D\Phi|\leq K_{u,v}$. We are going to prove that
\begin{equation}\label{inequality going to prove}
min\{\omega(t_{0},x_{0}),F_{u,v}(t_{0},x_{0},\omega,\Phi_{t},D\Phi,D^{2}\Phi)+\partial\psi_{\ast}(x_{0},D\Phi)\}\leq0.
\end{equation}
Now for arbitrarily given $\alpha>0$, we define
\[\Psi_{\alpha}(t,x,y):=u(t,x)-v(t,y)-\frac{\alpha}{2}|x-y|^{2}-\Phi(t,x).
\]
We choose $R>0$ sufficiently large such that $(t_{0},x_{0})\in(Dom\psi)_{r,R}$, where $(Dom\psi)_{r,R}:=\{x\in(Dom\psi)_{r}\big||x|\leq R\}$. Then there exists $(t_{\alpha},x_{\alpha},y_{\alpha})\in[0,T]\times(Dom\psi)_{r,R}^{2}$ such that
$\Psi_{\alpha}(t_{\alpha},x_{\alpha},y_{\alpha})=\max_{[0,T]\times(Dom\psi)_{r,R}^{2}}\Psi_{\alpha}(t,x,y)$.

From Proposition 3.7 \cite{CIL-92} we have
\[
\left\{
\begin{array}
[c]{l}%
(i)\qquad (t_{\alpha},x_{\alpha},y_{\alpha})\rightarrow (t_{0},x_{0},x_{0}), \text{ as } \alpha\rightarrow\infty;\medskip\\
(ii)\qquad \alpha|x_{\alpha}-y_{\alpha}|^{2}\text{ is bounded and tends to zero, as } \alpha\rightarrow\infty.
\end{array}
\right.
\]
Moreover, from $\omega(t_{\alpha},x_{\alpha})-\Phi(t_{\alpha},x_{\alpha})\leq \omega(t_{0},x_{0})-\Phi(t_{0},x_{0})=0$ we have
\[
\begin{array}
[c]{l}
0= \omega(t_{0},x_{0})-\Phi(t_{0},x_{0})=\Psi_{\alpha}(t_{0},x_{0},x_{0})\leq \Psi_{\alpha}(t_{\alpha},x_{\alpha},y_{\alpha})\medskip\\
\quad=\omega(t_{\alpha},x_{\alpha})-\Phi(t_{\alpha},x_{\alpha})+v(t_{\alpha},x_{\alpha})-v(t_{\alpha},y_{\alpha})-\frac{\alpha}{2}|x_{\alpha}-y_{\alpha}|^{2}
\medskip\\
\quad\leq v(t_{\alpha},x_{\alpha})-v(t_{\alpha},y_{\alpha})-\frac{\alpha}{2}|x_{\alpha}-y_{\alpha}|^{2},
\end{array}
\]
from where we deduce that
$\alpha|x_{\alpha}-y_{\alpha}|\leq 2\left|\frac{v(t_{\alpha},x_{\alpha})-v(t_{\alpha},y_{\alpha})}{x_{\alpha}-y_{\alpha}}\right|\leq 2K_{v}$, where $K_{v}$ is the Lipschitz constant of $v$ w.r.t. $x$.

We can assume that $u(t_{0},x_{0})>v(t_{0},x_{0})$; if not, (\ref{inequality going to prove}) holds obviously. Thus, from the continuity of $u$ and $v$, it follows that, for some $\alpha_{0}>0$, we have
 $u(t_{\alpha},x_{\alpha})>v(t_{\alpha},y_{\alpha})$ for $\alpha\ge \alpha_{0}$.

Let $\alpha>\alpha_{0}$. From Theorem 8.3  \cite{CIL-92}, we obtain that, for any $\delta>0$, there exists $(X^{\delta},Y^{\delta})\in S(n)\times S(n)$ and $c^{\delta}\in\mathbb{R}^{n}$ such that
\[
\begin{array}
[c]{l}%
(c^{\delta}+{\displaystyle\frac{\partial\Phi}{\partial t}(t_{\alpha},x_{\alpha})}, \alpha(x_{\alpha}-y_{\alpha})+D\Phi(x_{\alpha},y_{\alpha}),X^{\delta})\in\bar{\mathcal{P}}^{2,+}u(t_{\alpha},x_{\alpha});\medskip\\
(c^{\delta}, \alpha(x_{\alpha}-y_{\alpha}),Y^{\delta})\in\bar{\mathcal{P}}^{2,-}v(t_{\alpha},y_{\alpha})
\end{array}
\]
and
\[
\left(
\begin{array}{cc}
 X^{\delta} & 0 \\
 0 & -Y^{\delta}
\end{array}
\right)\leq A+\delta A^{2},
\]
where $A=\left(
\begin{array}{cc}
 D^{2}\Phi(t_{\alpha},x_{\alpha})+\alpha &  -\alpha \\
  -\alpha &  \alpha
\end{array}\right)$.

From Definition \ref{definition of viscosity by superjet subjet} and Remark \ref{viscosity solution in closed superjet} it follows that
\[
\begin{array}
[c]{l}
 c^{\delta}+{\displaystyle\frac{\partial\Phi}{\partial t}(t_{\alpha},x_{\alpha})}+\frac{1}{2}Tr(\sigma\sigma^{\ast}(t_{\alpha},x_{\alpha})X^{\delta})\medskip\\
\qquad +\langle b(t_{\alpha},x_{\alpha},u(t_{\alpha},x_{\alpha}),[\alpha(x_{\alpha}-y_{\alpha})+D\Phi(t_{\alpha},x_{\alpha})]^{\ast}\sigma(t_{\alpha},x_{\alpha})),
\alpha(x_{\alpha}-y_{\alpha})+D\Phi(t_{\alpha},x_{\alpha})\rangle\medskip\\
\qquad\qquad+f(t_{\alpha},x_{\alpha},u(t_{\alpha},x_{\alpha}),[\alpha(x_{\alpha}-y_{\alpha})+D\Phi(t_{\alpha},x_{\alpha})]^{\ast}\sigma(t_{\alpha},x_{\alpha}))
\medskip\\
\quad\ge \varphi_{-}^{\prime}(u(t_{\alpha},x_{\alpha}))+\partial\psi_{\ast}(x_{\alpha},\alpha(x_{\alpha}-y_{\alpha})+D\Phi(t_{\alpha},x_{\alpha})),
\end{array}
\]
and
\[
\begin{array}
[c]{l}%
c^{\delta}+\frac{1}{2}Tr(\sigma\sigma^{\ast}(t_{\alpha},y_{\alpha})Y^{\delta})+\langle b(t_{\alpha},y_{\alpha},v(t_{\alpha},y_{\alpha}),[\alpha(x_{\alpha}-y_{\alpha})]^{\ast}\sigma(t_{\alpha},y_{\alpha})),
\alpha(x_{\alpha}-y_{\alpha})\rangle\medskip\\
+f(t_{\alpha},y_{\alpha},v(t_{\alpha},y_{\alpha}),[\alpha(x_{\alpha}-y_{\alpha})]^{\ast}\sigma(t_{\alpha},y_{\alpha}))
\leq \varphi_{+}^{\prime}(v(t_{\alpha},y_{\alpha}))+\partial\psi^{\ast}(y_{\alpha},\alpha(x_{\alpha}-y_{\alpha})).
\end{array}
\]
Then we have
\begin{equation}\label{subtracting inequality}
\begin{array}
[c]{l}\varphi_{+}^{\prime}(v(t_{\alpha},y_{\alpha}))-\varphi_{-}^{\prime}(u(t_{\alpha},x_{\alpha}))\medskip\\
\qquad\quad+\partial\psi^{\ast}(y_{\alpha},\alpha(x_{\alpha}-y_{\alpha})) -\partial\psi_{\ast}(x_{\alpha},\alpha(x_{\alpha}-y_{\alpha})+D\Phi(t_{\alpha},x_{\alpha}))\ge-\frac{\partial\Phi}{\partial t}(t_{\alpha},x_{\alpha})\medskip\\
~-\left\{\frac{1}{2}Tr(\sigma\sigma^{\ast}(t_{\alpha},x_{\alpha})X^{\delta})
-\frac{1}{2}Tr(\sigma\sigma^{\ast}(t_{\alpha},y_{\alpha})Y^{\delta})\right\}\qquad\qquad\qquad\qquad(=:-I_{1}^{\alpha,\delta})\medskip\\
~-\Big\{\langle b(t_{\alpha},x_{\alpha},u(t_{\alpha},x_{\alpha}),[\alpha(x_{\alpha}-y_{\alpha})+D\Phi(t_{\alpha},x_{\alpha})]^{\ast}\sigma(t_{\alpha},x_{\alpha})),
\alpha(x_{\alpha}-y_{\alpha})+D\Phi(t_{\alpha},x_{\alpha})\rangle\medskip\\
\qquad\quad-b(t_{\alpha},y_{\alpha},v(t_{\alpha},y_{\alpha}),[\alpha(x_{\alpha}-y_{\alpha})]^{\ast}\sigma(t_{\alpha},y_{\alpha})),
\alpha(x_{\alpha}-y_{\alpha})\rangle\Big\}\qquad(=:-I_{2}^{\alpha,\delta})\medskip\\
~-\Big\{f(t_{\alpha},x_{\alpha},u(t_{\alpha},x_{\alpha}),
[\alpha(x_{\alpha}-y_{\alpha})+D\Phi(t_{\alpha},x_{\alpha})]^{\ast}\sigma(t_{\alpha},x_{\alpha}))\medskip\\
\qquad\quad-f(t_{\alpha},y_{\alpha},v(t_{\alpha},y_{\alpha}),[\alpha(x_{\alpha}-y_{\alpha})]^{\ast}\sigma(t_{\alpha},y_{\alpha}))\Big\}
\qquad\qquad\qquad\quad~~(=:-I_{3}^{\alpha,\delta}).
\end{array}
\end{equation}
We observe that the same argument as in Lemma 7.2 \cite{MC-01} yields
\begin{equation}\label{limit property for the sum}
\begin{array}
[c]{l}
\lim\limits_{\alpha\rightarrow\infty}\lim\limits_{\delta\rightarrow0}(I_{1}^{\alpha,\delta}+I_{2}^{\alpha,\delta}+I_{3}^{\alpha,\delta})
\medskip\\
\leq \frac{1}{2}Tr(\sigma\sigma^{\ast}(t_{0},x_{0})D^{2}\Phi(t_{0},x_{0}))
+\langle b(t_{0},x_{0},u(t_{0},x_{0}),0),D\Phi(t_{0},x_{0})\rangle\medskip\\
\quad+\tilde{K}\left\{|u(t_{0},x_{0})-v(t_{0},y_{0})|+|D\Phi(t_{0},x_{0})|\cdot|\sigma(t_{0},x_{0})|\right\}.
\end{array}
\end{equation}
Let us calculate now the left-hand side of (\ref{subtracting inequality}). Since $u(t_{\alpha},x_{\alpha})>v(t_{\alpha},y_{\alpha})$, we have
\begin{equation}\label{inequality for varphi}
\varphi_{+}^{\prime}(v(t_{\alpha},y_{\alpha}))-\varphi_{-}^{\prime}(u(t_{\alpha},x_{\alpha}))\leq 0.
\end{equation}
Morover, since $\psi$ is convex, $x_{\alpha},y_{\alpha}\in Int(Dom\psi)$ and
$\langle y^{\ast},\alpha(x_{\alpha}-y_{\alpha})\rangle
-\langle x^{\ast},\alpha(x_{\alpha}-y_{\alpha})\rangle \leq 0$, for $x^{\ast}\in\partial\psi(x_{\alpha})$, $y^{\ast}\in\partial\psi(y_{\alpha})$.
By using Lemma \ref{Z 2002 lemma}, it follows that
\begin{equation}\label{inequality for psi}
\begin{array}
[c]{l}
\partial\psi^{\ast}(y_{\alpha},\alpha(x_{\alpha}-y_{\alpha})) -\partial\psi_{\ast}(x_{\alpha},\alpha(x_{\alpha}-y_{\alpha})+D\Phi(t_{\alpha},x_{\alpha}))\medskip\\
=\sup_{y^{\ast}\in\partial\psi(y_{\alpha})}\langle y^{\ast},\alpha(x_{\alpha}-y_{\alpha})\rangle
-\inf_{x^{\ast}\in\partial\psi(x_{\alpha})}\langle x^{\ast},\alpha(x_{\alpha}-y_{\alpha})+D\Phi(t_{\alpha},x_{\alpha})\rangle\medskip\\
\leq\sup_{y^{\ast}\in\partial\psi(y_{\alpha})}\langle y^{\ast},\alpha(x_{\alpha}-y_{\alpha})\rangle
-\inf_{x^{\ast}\in\partial\psi(x_{\alpha})}\langle x^{\ast},\alpha(x_{\alpha}-y_{\alpha})\rangle\medskip\\
\qquad\qquad\qquad-\inf_{x^{\ast}\in\partial\psi(x_{\alpha})}\langle x^{\ast},D\Phi(t_{\alpha},x_{\alpha})\rangle\medskip\\
\leq-\inf_{x^{\ast}\in\partial\psi(x_{\alpha})}\langle x^{\ast},D\Phi(t_{\alpha},x_{\alpha})\rangle\medskip\\
=-\partial\psi_{\ast}(x_{\alpha},D\Phi(t_{\alpha},x_{\alpha})).
\end{array}
\end{equation}
Finally, letting $\delta\rightarrow0$ and after letting $\alpha\rightarrow\infty$ in (\ref{subtracting inequality}), by considering (\ref{limit property for the sum})-(\ref{inequality for psi}), we obtain
\[
\begin{array}
[c]{l}
-{\displaystyle\frac{\partial\Phi}{\partial t}(t_{0},x_{0})}-\frac{1}{2}Tr(\sigma\sigma^{\ast}(t_{0},x_{0})D^{2}\Phi(t_{0},x_{0}))
-\langle b(t_{0},x_{0},u(t_{0},x_{0}),0),D\Phi(t_{0},x_{0})\rangle\medskip\\
\quad-\tilde{K}\left\{|u(t_{0},x_{0})-v(t_{0},y_{0})|+|D\Phi(t_{0},x_{0})|\cdot|\sigma(t_{0},x_{0})|\right\}
\leq -\partial\psi_{\ast}(x_{0},D\Phi(t_{0},x_{0})).
\end{array}
\]
Therefore, from (\ref{F depend on u and v}),
\[
F_{u,v}(t_{0},x_{0},\omega,{\displaystyle\frac{\partial\Phi}{\partial t}},D\Phi,D^{2}\Phi)+\partial\psi_{\ast}(x_{0},D\Phi)\leq 0,
\]
evaluated at $(t_{0},x_{0})$.\hfill
\end{proof}
\bigskip

Using the approach in Lemma 3.8 \cite{BBP-97}, we construct a suitable supersolution for (\ref{one kind of pde}).
\begin{lemma}\label{an exactly supersolution}
For any $\tilde{A}>0$, there exists $\tilde{C}>0$ such that the function
\[
\chi(t,x)=\exp\left\{[\tilde{C}(T-t)+\tilde{A}]\eta(x)\right\},
\]
with
\[\eta(x)=\left\{\log\left[(|x|^{2}+1)^{1/2}\right]+1\right\}^{2},
\]
satisfies
\[min\{\chi(t,x),F_{u,v}(t,x,\chi,\frac{\partial\chi}{\partial t},D\chi,D^{2}\chi)+\partial\psi_{\ast}(x,D\chi)\}>0, \quad \text{ in } [t_{1},T]\times(Dom\psi)_{r},
\]
where $t_{1}=(T-\tilde{A}/\tilde{C})^{+}$.
\end{lemma}
\begin{proof}
A straightforward computation yields
\[
\begin{array}
[c]{l}
D\chi(t,x)=\left[\tilde{C}(T-t)+\tilde{A}\right]\chi(t,x)D\eta(x)=\left[\tilde{C}(T-t)+\tilde{A}\right]\chi(t,x)\frac{2[\eta(x)]^{1/2}}{1+|x|^{2}}x, \medskip\\
|D\chi(t,x)|\leq 4\tilde{A}\frac{[\eta(x)]^{1/2}}{[1+|x|^{2}]^{1/2}}\chi(t,x), \qquad |D^{2}\chi(t,x)|\leq 16(\tilde{A}^{2}+\tilde{A})\frac{\eta(x)}{[1+|x|^{2}]}\chi(t,x).
\end{array}
\]
Since $x\in(Dom\psi)_{r}\subset Int(Dom\psi)$ and $\langle x^{\ast},x\rangle\ge0$ for $x^{\ast}\in\partial\psi(x)$, it follows that
\[
\partial\psi_{\ast}(x,D\chi)=\inf_{x^{\ast}\in\partial\psi(x)}\langle x^{\ast}, D\chi(t,x)\rangle=\left[\tilde{C}(T-t)+\tilde{A}\right]\chi(t,x)\frac{2[\eta(x)]^{1/2}}{1+|x|^{2}}\inf_{x^{\ast}\in\partial\psi(x)}\langle x^{\ast}, x\rangle
\ge0.
\]
Since $b,\sigma$ grow at most linearly at infinity and $u$ is Lipschitz in $x$, uniformly w.r.t. $t$, we have, evaluating at $(t,x)$,
\[
\begin{array}
[c]{l}
F_{u,v}(t,x,\chi,{\displaystyle\frac{\partial\chi}{\partial t}},D\chi,D^{2}\chi)+\partial\psi_{\ast}(x,D\chi)\ge F_{u,v}(t,x,\chi,{\displaystyle\frac{\partial\chi}{\partial t}},D\chi,D^{2}\chi)\medskip\\
= -{\displaystyle\frac{\partial\chi}{\partial t}}-\frac{1}{2}Tr\{\sigma\sigma^{\ast}(t,x)D^{2}\chi\}-\langle b(t,x,u(t,x),0),D\chi\rangle-\tilde{K}[|\chi|+|D\chi|\cdot|\sigma(t,x)|]\medskip\\
\ge \chi(t,x)
\left[\tilde{C}\eta(x)-16(\tilde{A}^{2}+\tilde{A})C\eta(x)-4\tilde{A}C[\eta(x)]^{1/2}-\tilde{K}-4\tilde{A}\tilde{K}C[\eta(x)]^{1/2}\right],
\end{array}
\]
where $C$ is a constant independent of $\tilde{C}$. Since $\eta(x)\ge1$, we can choose $\tilde{C}$ large enough such that the quantity in the brackets is strictly positive. Consequently, taking into account that $\chi(t,x)>0$, we can conclude:
\[min\{\chi(t,x),F_{u,v}(t,x,\chi,\frac{\partial\chi}{\partial t},D\chi,D^{2}\chi)+\partial\psi_{\ast}(x,D\chi)\}>0, \quad \text{ in } [t_{1},T]\times(Dom\psi)_{r}.
\]
\hfill
\end{proof}

\begin{proof}[Proof of Theorem \ref{uniqueness of pde}] We only need to show that $\omega\leq 0$ on $[0,T]\times(Dom\psi)_{r}$ for any $r\in(0,r_{0}]$.
Now let us choose $\tilde{A}$ and $\tilde{C}$ as in Lemma \ref{an exactly supersolution}. Recalling that $\omega=u-v$ is Lipschitz in $x$, uniformly w.r.t. $t$, we remark that
\[\lim\limits_{|x|\rightarrow\infty}|\omega(t,x)|\exp\left\{-\tilde{A}\left[\log\left((|x|^{2}+1)^{1/2}\right)\right]^{2}\right\}=0,
\]
uniformly w.r.t. $t\in[0,T]$. This implies that, for all $\varepsilon>0$, $\left[\omega(t,x)-\varepsilon \chi(t,x)\right]e^{-\tilde{K}(T-t)}$ is bounded from above in $[t_{1},T]\times(Dom\psi)_{r}$, with $\tilde{K}$ as in (\ref{F depend on u and v}). Now we define
\[
M_{\varepsilon}^{r}:=\max\limits_{[t_{1},T]\times(Dom\psi)_{r}}\left[\omega(t,x)-\varepsilon \chi(t,x)\right]e^{-\tilde{K}(T-t)},
\]
and we suppose that the maximum $M_{\varepsilon}^{r}$ is achieved at some point $(t_{0},x_{0})$. We claim that $M_{\varepsilon}^{r}\leq0$ for all $r,\varepsilon>0$. This holds obviously true if $t_{0}=T$. Indeed, $M_{\varepsilon}^{r}=-\varepsilon \chi(T,x)\leq0$. Thus, we can assume that $t_{0}\in[0,T)$. Now we suppose that $M_{\varepsilon}^{r}>0$ for some $r,\varepsilon>0$, we will construct a contradiction. In fact, if we define
\[
\Phi(t,x):=\varepsilon \chi(t,x)+M_{\varepsilon}^{r}e^{\tilde{K}(T-t)}=\varepsilon \chi(t,x)+[\omega(t_{0},x_{0})-\varepsilon \chi(t_{0},x_{0})]e^{-\tilde{K}(t-t_{0})},
\]
we have $\Phi\in C^{1,2}([0,T]\times(Dom\psi)_{r})$, $\Phi(t_{0},x_{0})=\omega(t_{0},x_{0})$ and $\omega(t,x)-\Phi(t,x)\leq0$, for all $(t,x)\in[0,T]\times(Dom\psi)_{r}$. From $M_{\varepsilon}^{r}>0$, we have $\omega(t_{0},x_{0})>\varepsilon \chi(t_{0},x_{0})>0$. Then, from Lemma \ref{viscosity subsolution for one kind of pde} it follows that
\[
\begin{array}
[c]{l}
F_{u,v}(t_{0},x_{0},\Phi,{\displaystyle\frac{\partial\Phi}{\partial t}},D\Phi,D^{2}\Phi)+\partial\psi_{\ast}(x_{0},D\Phi)\leq0.
\end{array}
\]
Moreover, since, at $(t_{0},x_{0})$,
\[
\begin{array}
[c]{l}
0\ge F_{u,v}(t_{0},x_{0},\Phi,{\displaystyle\frac{\partial\Phi}{\partial t}},D\Phi,D^{2}\Phi)+\partial\psi_{\ast}(x_{0},D\Phi)\medskip\\
=-{\displaystyle\frac{\partial\Phi}{\partial t}}-\frac{1}{2}Tr\{\sigma\sigma^{\ast}(t_{0},x_{0})D^{2}\Phi\}-\langle b(t_{0},x_{0},u(t_{0},x_{0}),0),D\Phi\rangle
\medskip\\
\quad-\tilde{K}[|\Phi|+|D\Phi|\cdot|\sigma(t_{0},x_{0})|]+\partial\psi_{\ast}(x_{0},D\Phi)\medskip\\
=-\varepsilon{\displaystyle\frac{\partial\chi}{\partial t}}+\tilde{K}M_{\varepsilon}^{r}e^{\tilde{K}(T-t)}-\varepsilon\frac{1}{2}Tr\{\sigma\sigma^{\ast}(t_{0},x_{0})D^{2}\chi\}-\varepsilon\langle b(t_{0},x_{0},u(t_{0},x_{0}),0),D\chi\rangle
\medskip\\
\quad-\tilde{K}[\varepsilon|\chi|+M_{\varepsilon}^{r}e^{\tilde{K}(T-t)}+\varepsilon|D\chi|\cdot|\sigma(t_{0},x_{0})|]+\varepsilon\partial\psi_{\ast}(x_{0},D\chi)
\medskip\\
=\varepsilon \left[F_{u,v}(t_{0},x_{0},\chi,{\displaystyle\frac{\partial\chi}{\partial t}},D\chi,D^{2}\chi)+\partial\psi_{\ast}(x_{0},D\chi)\right]
\end{array}
\]
we have
\[
F_{u,v}(t_{0},x_{0},\chi,\frac{\partial\chi}{\partial t},D\chi,D^{2}\chi)+\partial\psi_{\ast}(x_{0},D\chi)\leq0,
\]
which contradicts Lemma \ref{an exactly supersolution}. Therefore, $M_{\varepsilon}^{r}\leq0$ which implies that $\omega(t,x)\leq\varepsilon \chi(t,x)$, for
all $(t,x)\in[t_{1},T]\times(Dom\psi)_{r}$. Letting $\varepsilon\rightarrow0$, we get $\omega(t,x)\leq0$ for
all $(t,x)\in[t_{1},T]\times(Dom\psi)_{r}$. Applying successively the similar argument on the interval $[t_{2},t_{1}]$, if necessary, where $t_{2}=(t_{1}-\tilde{A}/\tilde{C})^{+}$, and then if $t_{2}>0$, on $[t_{3},t_{2}]$, where $t_{3}=(t_{2}-\tilde{A}/\tilde{C})^{+}$, etc., we obtain, finally,
\[
\omega(t,x)\leq0,\text{ for all } (t,x)\in[0,T]\times(Dom\psi)_{r}, ~r\in(0,r_{0}].
\]
\hfill
\end{proof}
\section{FBSVIs}
In this section, in order to prepare our existence result for PVI (\ref{PVI}), we study one general kind of FBSVIs in order to give a probabilistic interpretation for the viscosity solution of PVI (\ref{PVI}).

Let $(\Omega,\mathcal{F},\mathbb{P})$ be a complete probability space and endowed with an $\mathbb{R}^{d}$-valued standard Brownian motion $\{B_{t}\}_{t\ge0}$. We denote by $\{\mathcal{F}_{t}\}_{t\ge0}$ the filtration generated by the Brownian motion $B$ and augmented by the class of $P$-null sets of $\mathcal{F}$.

We consider the following FBSVI:
\begin{equation}\label{FBSDE with two operators}
\left\{
\begin{array}
[c]{l}%
dX_{t}+\partial\psi(X_{t})dt\ni b(t,X_{t},Y_{t},Z_{t})dt+\sigma(t,X_{t},Y_{t},Z_{t})dB_{t},\medskip\\
-dY_{t}+\partial\varphi(Y_{t})dt\ni f(t,X_{t},Y_{t},Z_{t})dt-Z_{t}dB_{t},\quad t\in[0,T],\medskip\\
X_{0}=x, \quad Y_{T}=g(X_{T}),
\end{array}
\right.
\end{equation}
where the processes $X,Y,Z$ take values in $\mathbb{R}^{n},\mathbb{R}^{m}$ and $\mathbb{R}^{m\times d}$, respectively, and the functions $b,\sigma,f$ and $g$ satisfy standard assumptions which we will give later. The operator $\partial\psi$ (resp. $\partial\varphi$) is the subdifferential of the
convex l.s.c. function $\psi:\mathbb{R}^{n}\rightarrow
(-\infty,+\infty]$ (resp. $\varphi:\mathbb{R}^{m}\rightarrow
(-\infty,+\infty]$).

Let us introduce some spaces of processes, which will be needed in what follows.

Let $t\in[0,T]$, $k\ge1$. We define $BV([t,T];\mathbb{R}^{k})$ as the space of the functions $u:[t,T]\rightarrow \mathbb{R}^{k}$ with finite total variation on $[t,T]$, denoted by $\updownarrow u\updownarrow_{[t,T]}$. Moreover we endow this space with the norm
\[\|u\|_{BV([t,T];\mathbb{R}^{k})}=|u(t)|+\updownarrow u\updownarrow_{[t,T]}.
\]
We also introduce the space $L^{1}(\Omega;BV([t,T];\mathbb{R}^{k}))$ of the stochastic processes $u:\Omega\times[t,T]\rightarrow\mathbb{R}^{k}$ such that
\[E\|u\|_{BV([t,T];\mathbb{R}^{k})}<\infty.
\]
The space $M^{p}_{k}[t,T],~p\ge 2$, denotes the Hilbert space of $\mathbb{R}^{k}$-valued $\{\mathcal{F}_{s}\}$-progressively measurable processes $\{u(s), ~s\in[t,T]\}$ such that
\[\|u\|_{M^{p}[t,T]}:=\left(E\left(\int_{t}^{T}|u(s)|^{2}ds\right)^{p/2}\right)^{1/p}<\infty.
\]
When $p=2$, we set $\|\cdot\|_{M[t,T]}:=\|\cdot\|_{M^{p}[t,T]}$. For $\lambda\in\mathbb{R}$, we define an equivalent norm on $M^{2}_{k}[t,T]$:
\[\|u\|_{M_{\lambda}[t,T]}:=\left(E\int_{t}^{T}e^{-\lambda s}|u(s)|^{2}ds\right)^{1/2}.
\]
Moreover, let $H[t,T]$ be the subset of $M^{2}_{n}[t,T]$ consisting of all the continuous processes. For $\beta>0$ and $\lambda\in\mathbb{R}$, we denote its completion under the norm
\[\|u\|_{\lambda,\beta,[t,T]}:=e^{-\lambda T}E|u(T)|^{2}+\beta\|u\|_{M_{\lambda}[t,T]}
\]
by  $\bar{H}[t,T]$. The notation $\bar{H}[t,T]$ takes into account that the completion of $H[t,T]$ under the norm $\|\cdot\|_{\lambda,\beta,[t,T]}$ does not depend on $\lambda$ and $\beta$.

Let $S^{2}_{k}[t,T]$ be the set of all continuous $\{\mathcal{F}_{s}\}$-progressively measurable processes $\{u(s), ~s\in[t,T]\}$ which values in $\mathbb{R}^{k}$ such that
\[\|u\|_{S[t,T]}:=\left(E\sup\limits_{t\leq s\leq T}|u(s)|^{2}\right)^{1/2}<\infty.
\]
We also introduce an equivalent norm on $S^{2}_{k}[t,T]$:
\[\|u\|_{S_{\lambda}[t,T]}:=\left(E\sup\limits_{t\leq s\leq T}e^{-\lambda s}|u(s)|^{2}\right)^{1/2}.
\]
In what follows, if $t=0$, we simplify the notations by writing, for example: $M^{2}_{k}:=M^{2}_{k}[0,T], ~H:=H[0,T]$.

Let us give the following definition.
\begin{definition}\label{definition of the solution of FBSDE with two operators}
A quintuple $(X,Y,Z,V,U)$ of processes is called an adapted solution of FBSVI (\ref{FBSDE with two operators}), if the following conditions are satisfied:
\[%
\begin{array}
[c]{rl}%
(a_{1}) & X\in S^{2}_{n},~Y\in S^{2}_{m},~Z\in M^{2}_{m\times d},~V\in S^{2}_{n}\cap L^{1}(\Omega;BV([0,T];\mathbb{R}^{n})),V_{0}=0, ~U\in M^{2}_{m},\medskip\\
(a_{2})
 & X_{t}\in Dom\psi, ~d\mathbb{P}\otimes dt~a.e. ~and~\psi(X)\in L^{1}(\Omega\times[0,T];\mathbb{R}),\medskip\\
 & {\displaystyle\int_{s}^{t}}\langle z-X_{r},dV_{r}\rangle+{\displaystyle\int_{s}^{t}}\psi(X_{r})dr\leq (t-s)\psi(z),~ \text{ for all } z\in\mathbb{R}^{n}, ~ ~0\leq s\leq t\leq T, ~a.s.\medskip\\
(a_{3})& (Y_{t},U_{t})\in\partial\varphi,~d\mathbb{P}\otimes dt~a.e.\text{ on
}\Omega\times[0,T],\medskip\\
(a_{4}) &
X_{t}+V_{t}=x+{\displaystyle\int_{0}^{t}}b(s,X_{s},Y_{s},Z_{s})ds+{\displaystyle\int_{0}^{t}}\sigma(s,X_{s},Y_{s},Z_{s})dB_{s},\medskip\\
& Y_{t}+{\displaystyle\int_{t}^{T}}U_{s}ds=g(X_{T})+{\displaystyle\int
_{t}^{T}}f( s,X_{s},Y_{s},Z_{s})ds-{\displaystyle\int_{t}^{T}}Z_{s}dB_{s},~ t\in[0,T],\;a.s.
\end{array}
\]
\end{definition}
\subsection{Assumptions}
Now we give the following standard assumptions:
\begin{itemize}
\item[$\left(  H_{1}\right)  $] Let $\psi:\mathbb{R}^{n}\rightarrow(-\infty,+\infty]$ be a convex l.s.c.
function with $0\in Int(Dom\psi)$ and $\psi(z)\geq\psi(0)=0$, for all $z\in\mathbb{R}^{n}$. Moreover, the initial point $x$ from (\ref{FBSDE with two operators}) belongs to $Dom\psi$.

\item[$\left(  H_{2}\right)  $] Let $\varphi:\mathbb{R}^{m}\rightarrow(-\infty,+\infty]$ be a convex l.s.c.
function s.t.
$\varphi(y)\geq\varphi(0)=0$ for all $y\in\mathbb{R}^{m}$.

\item[$(H_{3})$] The coefficients $b,\sigma$ and $f$ are defined on $\Omega\times[0,T]\times\mathbb{R}^{n}\times\mathbb{R}^{m}\times\mathbb{R}^{m\times d}$, s.t. $b(\cdot,\cdot,x,y,z)$,\\$\sigma(\cdot,\cdot,x,y,z)$ and $f(\cdot,\cdot,x,y,z)$ are $\{\mathcal{F}_{t}\}$-progressively measurable processes, for all fixed $(x,y,z)\in\mathbb{R}^{n}\times\mathbb{R}^{m}\times\mathbb{R}^{m\times d}$. The coefficient $g$ is defined on $\Omega\times\mathbb{R}^{n}$ and $g(\cdot,x)$ is $\mathcal{F}_{T}$-measurable, for all fixed $x\in\mathbb{R}^{n}$.

\item[$\left(  H_{4}\right)  $] The mapping $y\mapsto f(\omega,t,x,y,z):\mathbb{R}^{m}\rightarrow\mathbb{R}^{m}$ is continuous and there exists a constant $L\ge0$ and $\eta\in M^{2}_{1}$, such that for all $(\omega,t,y)$, $|f(\omega,t,0,y,0)|\leq \eta(\omega,t)+L|y|$ (which implies that $f(\cdot,\cdot,0,0,0)\in M^{2}_{m}$). Moreover, $b(\cdot,\cdot,0,0,0)\in M^{2}_{n}$, $\sigma(\cdot,\cdot,0,0,0)\in M^{2}_{n\times d}$ and $E|g(\cdot,0)|^{2}<\infty$.

\item[$\left(  H_{5}\right)  $] There exist positive constants $K, k_{1}, k_{2}$ and a constant $\gamma\in\mathbb{R}$, such that for all $t,x,x_{1},x_{2}$,\\$y,y_{1},y_{2}$,$z,z_{1},z_{2}$, a.s.
\[
\begin{array}
[c]{rl}
(i) & | b(t,x_{1},y_{1},z_{1})-b(t,x_{2},y_{2},z_{2})| \leq
K(|x_{1}-x_{2}| +| y_{1}-y_{2}|
+|z_{1}-z_{2}| ),\medskip\\
(ii)& |\sigma(t,x_{1},y_{1},z_{1})-\sigma(t,x_{2},y_{2},z_{2})|^{2} \leq
K^{2}(| x_{1}-x_{2}|^{2} +| y_{1}-y_{2}|^{2})
+k_{1}^{2}|z_{1}-z_{2}|^{2},\medskip\\
(iii)& |g(x_{1})-g(x_{2})|\leq
k_{2}| x_{1}-x_{2}|,\medskip\\
(iv) & | f(t,x_{1},y,z_{1})-f(t,x_{2},y,z_{2})| \leq
K(|x_{1}-x_{2}|+|z_{1}-z_{2}| ),\medskip\\
(v) & \langle f(t,x,y_{1},z)-f(t,x,y_{2},z),y_{1}-y_{2}\rangle\leq\gamma|y_{1}-y_{2}|^{2}.
\end{array}
\]
\end{itemize}
\begin{remark}
(1) We shall also introduce the following conditions:
\begin{itemize}
\item[$\left(  H^{\prime}_{4}\right)  $] The mapping $y\mapsto f(\omega,t,x,y,z):\mathbb{R}^{m}\rightarrow\mathbb{R}^{m}$ is continuous and there exist constants $L\ge0$, $\rho_{0}>0$ and a process $\eta\in M^{3+\rho_{0}}_{1}$, such that for all $(\omega,t,y)$, $|f(\omega,t,0,y,0)|\leq \eta(\omega,t)+L|y|$ (which implies that $f(\cdot,\cdot,0,0,0)\in M^{3+\rho_{0}}_{m}$). Moreover, $b(\cdot,\cdot,0,0,0)\in M^{3+\rho_{0}}_{n}$, $\sigma(\cdot,\cdot,0,0,0)\in M^{3+\rho_{0}}_{n\times d}$ and $E|g(\cdot,0)|^{3+\rho_{0}}<\infty$.
\item[$\left(  H^{\prime}_{5}\right)  $] $k_{1}=0$, i.e., $\sigma$ does not depend on $z$.
\end{itemize}

(2) For simplification, we put $b^{0}(s):=b(\cdot,s,0,0,0)$, $\sigma^{0}(s):=\sigma(\cdot,s,0,0,0)$, $f^{0}(s):=f(\cdot,s,0,0,0)$ and $g^{0}:=g(\cdot,0)$.
\end{remark}

These assumptions will be completed by compatibility hypotheses which were introduced by Cvitani\'{c} and Ma \cite{MC-01}:
\[
\begin{array}
[c]{ll}
\left(  C1\right)   & 0\leq k_{1}k_{2}< 1,\medskip\\
\left(  C2\right)   & \text{If } k_{2}=0 \text{ then there exists }\alpha\in(0,1), \text{ s.t. }  ~\mu(\alpha,T)KC_{3}<\bar{\lambda}_{1},\medskip\\
\left(  C3\right)   & \text{If } k_{2}>0 \text{ then there exists }\alpha\in(k^{2}_{1}k^{2}_{2},1), \text{ s.t. } ~\mu(\alpha,T)k_{2}^{2}<1 \text{ and } \bar{\lambda}_{1}\ge\frac{KC_{3}}{k_{2}^{2}}.
\end{array}
\]
Here
\begin{equation}\label{notations}
\begin{array}
[c]{ll}%
\mu(\alpha,T):=K(C_{1}+K)B(\bar{\lambda}_{2},T)+\frac{A(\bar{\lambda}_{2},T)}{\alpha}(KC_{2}+k_{1}^{2}),\medskip\\
A(\lambda,t)=e^{-(\lambda\wedge0)t},~B(\lambda,t)={\displaystyle\int_{0}^{t}}e^{-\lambda s}ds,~t\in[0,T],\medskip\\
\bar{\lambda}_{1}=\lambda-K(2+C_{1}^{-1}+C_{2}^{-1})-K^{2},~\bar{\lambda}_{2}=-\lambda-2\gamma-K(C_{3}^{-1}+C_{4}^{-1}),\medskip\\
\text{for } \lambda\in\mathbb{R} \text{ and constants }C_{1},C_{2},C_{3},C_{4}>0.
\end{array}
\end{equation}
\begin{remark} We mention that in Cvitani\'{c} and Ma \cite{MC-01}, $\bar{\lambda}_{1}=\frac{KC_{3}}{k_{2}^{2}}$ in $(C3)$. However, it turns out that $\bar{\lambda}_{1}\ge\frac{KC_{3}}{k_{2}^{2}}$ is enough. See the proof of Proposition \ref{Priori estimate for X and Y} and Theorem \ref{Solvability of Penalized FBSDE in appendix} in the appendix of our paper.
\end{remark}
\subsection{Penalized FBSDEs and a priori estimates}
In this section, we give a priori estimates on penalized equations related with FBSVI (\ref{FBSDE with two operators}), inspired by \cite{AR-97}, \cite {MC-01} and \cite{PR-98}.

We begin with recalling the Yosida approximation for our convex l.s.c. function $\varphi$:
\[
\varphi_{\varepsilon}(u):=\inf\left\{  \frac{1}{2\varepsilon}|u-v|^{2}%
+\varphi(v):v\in\mathbb{R}^{m}\right\},~ \varepsilon>0,~u\in\mathbb{R}^{m}.
\]
It is well known that the function $\varphi_{\varepsilon}$ is convex and belongs to the class $C^{1}(\mathbb{R}^{m})$. The gradient $\nabla\varphi_{\varepsilon}$ is a
Lipschitz function with Lipschitz constant $1/\varepsilon$. We set
\[
J_{\varepsilon,\varphi}(u)=u-\varepsilon\nabla\varphi_{\varepsilon}(u), ~u\in\mathbb{R}^{m}.
\]
The approximation $\varphi_{\varepsilon}$ has the following properties (see
\cite{B-76} and \cite{PR-98}):

For all $u,v\in\mathbb{R}^{m},$ and $\varepsilon,\delta>0$, we have
\begin{equation}\label{Property of Yosida approximation}
\begin{array}
[c]{ll}%
\left(  a\right)   & \varphi_{\varepsilon}(u)=\frac{\varepsilon}{2}%
|\nabla\varphi_{\varepsilon}(u)|^{2}+\varphi(J_{\varepsilon,\varphi}(u)),\medskip\\
\left(  b\right)   & \left\vert J_{\varepsilon,\varphi}(u)-J_{\varepsilon,\varphi}(v)\right\vert
\leq|u-v|,\medskip\\
\left(  c\right)   & \nabla\varphi_{\varepsilon}(u)\in\partial\varphi
(J_{\varepsilon,\varphi}(u)),\medskip\\
\left(  d\right)   & 0\leq\varphi_{\varepsilon}(u)\leq \langle\nabla\varphi_{\varepsilon
}(u),u\rangle,\medskip\\
\left(  e\right)   & \langle\nabla\varphi_{\varepsilon}(u)-\nabla
\varphi_{\delta}(v),u-v\rangle\geq-(\varepsilon+\delta)
|\nabla\varphi_{\varepsilon}(u)||\nabla\varphi_{\delta}(v)|.
\end{array}
\end{equation}
Moreover for our convex l.s.c. function $\psi$, we have in addition, the following property (see \cite{AR-97,B-76}):
\begin{equation}\label{Property f}
\begin{array}
[c]{l}
(f)\quad \text{ For all } u_{0}\in Int(Dom\psi), \text{ there exist } ~r_{0}>0, ~ M_{0}>0,\text{ such that }\medskip\\
\qquad\qquad r_{0}|\nabla\psi_{\varepsilon}(x)|\leq \langle\nabla\psi_{\varepsilon}(x),x-u_{0}\rangle+M_{0}, ~\text{ for all }~\varepsilon>0,~x\in\mathbb{R}^{n}.
\end{array}
\end{equation}

Let $\varepsilon>0$. We consider the following penalized FBSDE using the Yosida approximation for $\psi$ and $\varphi$:
\begin{equation}\label{Penalized FBSDE with two operators}
\left\{
\begin{array}
[c]{l}%
X_{t}^{\varepsilon}+{\displaystyle\int_{0}^{t}}\nabla\psi_{\varepsilon}(X_{s}^{\varepsilon})ds=x+{\displaystyle\int_{0}^{t}}b(s,X_{s}^{\varepsilon},Y_{s}^{\varepsilon},Z_{s}^{\varepsilon})ds
+{\displaystyle\int_{0}^{t}}\sigma(s,X_{s}^{\varepsilon},Y_{s}^{\varepsilon},Z_{s}^{\varepsilon})dB_{s},\medskip\\
Y_{t}^{\varepsilon}+{\displaystyle\int_{t}^{T}}\nabla\varphi_{\varepsilon}(Y_{s}^{\varepsilon})ds=g(X_{T}^{\varepsilon})+{\displaystyle\int
_{t}^{T}}f( s,X_{s}^{\varepsilon},Y_{s}^{\varepsilon},Z_{s}^{\varepsilon})ds-{\displaystyle\int_{t}^{T}}Z_{s}^{\varepsilon}dB_{s},~t\in[0,T].
\end{array}
\right.
\end{equation}
From Theorem \ref{Solvability of Penalized FBSDE in appendix} (see Appendix), we have
\begin{lemma}\label{Solvability of Penalized FBSDE}
Let the assumptions $(H_{1})$-$(H_{5})$ be satisfied. We also assume $(C1)$ and either $(C2)$ or $(C3)$ hold for some $\lambda,\alpha,C_{1},C_{2},C_{3}$ and $C_{4}=\frac{1-\alpha}{K}$. Then penalized FBSDE (\ref{Penalized FBSDE with two operators}) has a unique adapted solution $(X^{\varepsilon},Y^{\varepsilon},Z^{\varepsilon})\in S^{2}_{n}\times S^{2}_{m}\times M^{2}_{n\times d}$. Moreover, if among the compatibility conditions, only $(C1)$ holds, then penalized  FBSDE (\ref{Penalized FBSDE with two operators}) has a unique adapted solution on $[0,T_{0}]$, but only for $T_{0}>0$ small enough.
\end{lemma}
\begin{remark}
(1) In our paper we only discuss the assumption $(C1)$ in combination with either $(C2)$ or $(C3)$. For the case that only $(C1)$ holds and $T_{0}>0$ small enough, all our results can be obtained similarly, so we omit it.

(2)As explained in Remark 3.2 \cite{MC-01}, the compatibility conditions do not include the case of an arbitrary $T$ since the constants introduced in
(\ref{notations}) depend on $T$. However, under the condition $(C1)$ , if $\gamma\leq -\Upsilon$, for some $\Upsilon>0$ depending only on $K,k_{1},k_{2}$, then all our results holds for arbitrary $T$.
\end{remark}

\begin{proposition}\label{Difference 2}
Under the assumptions of Lemma \ref{Solvability of Penalized FBSDE}, if $(X^{t,x},Y^{t,x},Z^{t,x},V^{t,x},U^{t,x})$ (resp. $(\tilde{X}^{t,\tilde{x}},\tilde{Y}^{t,\tilde{x}},\tilde{Z}^{t,\tilde{x}},\tilde{V}^{t,\tilde{x}},\tilde{U}^{t,\tilde{x}})$) is a solution of the FBSVI with initial time $t$ and parameters $(x,b,\sigma,f,g)$ (resp. $(\tilde{x},\tilde{b},\tilde{\sigma},\tilde{f},\tilde{g})$), then there exists a constant $C$ independent of $(t,x,\tilde{x})$, such that
\begin{equation}\label{Difference of two solutions of two operators}
\begin{array}
[c]{l}
\quad  E\Big\{\sup\limits_{t\leq s\leq T}|X_{s}^{t,x}-\tilde{X}_{s}^{t,\tilde{x}}|^{2}+\sup\limits_{t\leq s\leq T}|Y_{s}^{t,x}-\tilde{Y}_{s}^{t,\tilde{x}}|^{2}+{\displaystyle\int_{t}^{T}}|Z_{s}^{t,x}-\tilde{Z}_{s}^{t,\tilde{x}}|^{2}ds\Big\}\leq C\Delta_{1},
\end{array}
\end{equation}
where
\[
\begin{array}
[c]{l}%
\Delta_{1}=e^{-\lambda t}|x-\tilde{x}|^{2}+E|g(X_{T})-\tilde{g}(X_{T})|^{2}+E{\displaystyle\int_{t}^{T}}|b-\tilde{b}|^{2}(s,X_{s}^{t,x},Y_{s}^{t,x},Z_{s}^{t,x})ds
\medskip\\
\qquad+E{\displaystyle\int_{t}^{T}}|f-\tilde{f}|^{2}(s,X_{s}^{t,x},Y_{s}^{t,x},Z_{s}^{t,x})ds
+E{\displaystyle\int_{t}^{T}}|\sigma-\tilde{\sigma}|^{2}(s,X_{s}^{t,x},Y_{s}^{t,x},Z_{s}^{t,x})ds.
\end{array}
\]
\end{proposition}
\begin{proof}
From Definition \ref{definition of the solution of FBSDE with two operators} $(a_{2})$, it follows that
\[
{\displaystyle\int_{a}^{b}}\langle z-X_{r}^{t,x},dV_{r}^{t,x}\rangle+{\displaystyle\int_{a}^{b}}\psi(X_{r}^{t,x})dr\leq (b-a)\psi(z),~ z\in\mathbb{R}^{n},  ~t\leq a\leq b\leq T,
\]
Recalling Proposition 1.2 of \cite{AR-97}, we know that it is equivalent to
\[
{\displaystyle\int_{a}^{b}}\langle y_{r}-X_{r}^{t,x},dV_{r}^{t,x}\rangle+{\displaystyle\int_{a}^{b}}\psi(X_{r}^{t,x})dr\leq {\displaystyle\int_{a}^{b}}\psi(y_{r})dr,
~ y\in C([t,T];\mathbb{R}^{n}),  ~t\leq a\leq b\leq T.
\]
Consequently, we have
\[
\begin{array}
[c]{l}%
{\displaystyle\int_{a}^{b}}\langle \tilde{X}_{r}^{t,\tilde{x}}-X_{r}^{t,x},dV_{r}^{t,x}\rangle+{\displaystyle\int_{a}^{b}}\psi(X_{r})dr\leq {\displaystyle\int_{a}^{b}}\psi(\tilde{X}_{r}^{t,\tilde{x}})dr,
~t\leq a\leq b\leq T,\medskip\\
{\displaystyle\int_{a}^{b}}\langle X_{r}^{t,x}-\tilde{X}_{r}^{t,\tilde{x}},d\tilde{V}_{r}^{t,x}\rangle+{\displaystyle\int_{a}^{b}}\psi(\tilde{X}_{r}^{t,\tilde{x}})dr\leq {\displaystyle\int_{a}^{b}}\psi(X_{r}^{t,x})dr,
~t\leq a\leq b\leq T,
\end{array}
\]
which yields
\begin{equation}\label{equation 19}
{\displaystyle\int_{a}^{b}}\langle X_{r}^{t,x}-\tilde{X}_{r}^{t,\tilde{x}},dV_{r}-d\tilde{V}_{r}^{t,x}\rangle\ge 0.
\end{equation}
Moreover, from $(Y,U),(\tilde{Y},\tilde{U})\in\partial\varphi$, it follows
\begin{equation}\label{equation 20}
\langle Y-\tilde{Y},U-\tilde{U}\rangle\ge 0.
\end{equation}
Using (\ref{equation 19}) and (\ref{equation 20}), similarly to the estimates (\ref{Estimate 5})-(\ref{Estimate 8}) of Proposition \ref{Priori estimate for X and Y} (see the appendix), it follows
\begin{equation}\label{Estimate 13}
\begin{array}
[c]{l}%
e^{-\lambda T}E|X_{T}^{t,x}-\tilde{X}_{T}^{t,\tilde{x}}|^{2}+\bar{\lambda}_{1}^{\delta}\|X^{t,x}-\tilde{X}^{t,\tilde{x}}\|^{2}_{M_{\lambda}[t,T]}\leq K[C_{1}+K(1+\delta)]\|Y^{t,x}-\tilde{Y}^{t,\tilde{x}}\|^{2}_{M_{\lambda}[t,T]}\medskip\\
\qquad\qquad\qquad+[KC_{2}+k_{1}^{2}(1+\delta)]\|Z^{t,x}-\tilde{Z}^{t,\tilde{x}}\|^{2}_{M_{\lambda}[t,T]}+e^{-\lambda t}|x-\tilde{x}|^{2}\medskip\\
\qquad
+\frac{1}{\delta}\|(b-\tilde{b})(X^{t,x},Y^{t,x},Z^{t,x})\|_{M_{\lambda}[t,T]}
+(1+\frac{1}{\delta})\|(\sigma-\tilde{\sigma})(X^{t,x},Y^{t,x},Z^{t,x})\|_{M_{\lambda}[t,T]},
\end{array}
\end{equation}
\begin{equation}\label{Estimate 15}
\begin{array}
[c]{l}%
\|Y^{t,x}-\tilde{Y}^{t,\tilde{x}}\|^{2}_{M_{\lambda}[t,T]}\leq B(\bar{\lambda}_{2}^{\delta},T)\Big\{
k_{2}^{2}(1+\delta)e^{-\lambda T}E|X_{T}^{t,x}-\tilde{X}_{T}^{t,\tilde{x}}|^{2}+KC_{3}\|X^{t,x}-\tilde{X}^{t,\tilde{x}}\|^{2}_{M_{\lambda}[t,T]}\medskip\\
\qquad\qquad
\frac{1}{\delta}\|(f-\tilde{f})(X^{t,x},Y^{t,x},Z^{t,x})\|_{M_{\lambda}[t,T]}+(1+\frac{1}{\delta})e^{-\lambda T}E|g(X_{T}^{t,x})-\tilde{g}(X_{T}^{t,x})|^{2}\Big\},
\end{array}
\end{equation}
and
\begin{equation}\label{Estimate 16}
\begin{array}
[c]{l}%
\|Z^{t,x}-\tilde{Z}^{t,\tilde{x}}\|^{2}_{M_{\lambda}[t,T]}\leq \frac{A(\bar{\lambda}_{2}^{\delta},T)}{\alpha}\Big\{
k_{2}^{2}(1+\delta)e^{-\lambda T}E|X_{T}^{t,x}-\tilde{X}_{T}^{t,\tilde{x}}|^{2}+KC_{3}\|X^{t,x}-\tilde{X}^{t,\tilde{x}}\|^{2}_{M_{\lambda}[t,T]}\medskip\\
\qquad\qquad
\frac{1}{\delta}\|(f-\tilde{f})(X^{t,x},Y^{t,x},Z^{t,x})\|_{M_{\lambda}[t,T]}+(1+\frac{1}{\delta})e^{-\lambda T}E|g(X_{T}^{t,x})-\tilde{g}(X_{T}^{t,x})|^{2}\Big\}.
\end{array}
\end{equation}
Then using the same argument as in Proposition \ref{Priori estimate for X and Y}, we obtain our results.\hfill
\end{proof}
\begin{remark}\label{L2 estimates of the solution of fbsvi}
Putting $(X^{t,x},Y^{t,x},Z^{t,x},V^{t,x},U^{t,x})=(0,0,0,0,0)$ which is the solution of FBSVI with initial time $t$ and parameters $(0,0,0,0,0)$, we see from Proposition \ref{Difference 2} that
there exists a constant $C$ independent of $(t,\tilde{x})$, such that
\begin{equation}\label{L2 estimates}
\begin{array}
[c]{l}
\quad  E\Big\{\sup\limits_{t\leq s\leq T}|\tilde{X}_{s}^{t,\tilde{x}}|^{2}+\sup\limits_{t\leq s\leq T}|\tilde{Y}_{s}^{t,\tilde{x}}|^{2}+{\displaystyle\int_{t}^{T}}|\tilde{Z}_{s}^{t,\tilde{x}}|^{2}ds\Big\}\leq C\Delta_{2},
\end{array}
\end{equation}
where
\[
\begin{array}
[c]{l}%
\Delta_{2}=e^{-\lambda t}|\tilde{x}|^{2}+E|\tilde{g}^{0}|^{2}+E{\displaystyle\int_{t}^{T}}|\tilde{b}^{0}(s)|^{2}ds+E{\displaystyle\int_{t}^{T}}|\tilde{f}^{0}(s)|^{2}ds
+E{\displaystyle\int_{t}^{T}}|\tilde{\sigma}^{0}(s)|^{2}ds.
\end{array}
\]
\end{remark}

\begin{proposition}\label{Lipschitz sequence 1}
Let the assumptions $(H_{1})$-$(H_{5})$ be satisfied. We also assume $(C1)$ and either $(C2)$ or $(C3)$ hold for some $\lambda,\alpha,C_{1},C_{2},C_{3}$ and $C_{4}=\frac{1-\alpha}{K}$. Then for all $\varepsilon_{1},\varepsilon_{2}>0$,
\begin{equation}\label{Lipschitz sequence equation 1}
\begin{array}
[c]{l}
\quad E\Big\{\sup\limits_{0\leq t\leq T}|X_{t}^{\varepsilon_{1}}-X_{t}^{\varepsilon_{2}}|^{2}
+\sup\limits_{0\leq t\leq T}|Y_{t}^{\varepsilon_{1}}-Y_{t}^{\varepsilon_{2}}|^{2}
+{\displaystyle\int_{0}^{T}}|Z_{s}^{\varepsilon_{1}}-Z_{s}^{\varepsilon_{2}}|^{2}ds\Big\}\medskip\\
\leq C(\varepsilon_{1}+\varepsilon_{2})
E\left[{\displaystyle\int_{0}^{T}}|\nabla\psi_{\varepsilon_{1}}(X^{\varepsilon_{1}}_{s})|
|\nabla\psi_{\varepsilon_{2}}(X^{\varepsilon_{2}}_{s})|ds+
{\displaystyle\int_{0}^{T}}|\nabla\varphi_{\varepsilon_{1}}(Y^{\varepsilon_{1}}_{s})|
|\nabla\varphi_{\varepsilon_{2}}(Y^{\varepsilon_{2}}_{s})|ds\right],
\end{array}
\end{equation}
where $C$ is a constant which does not depend on $\varepsilon_{1}$ nor on $\varepsilon_{2}$.
\end{proposition}
\begin{proof}
We apply It\^{o}'s formula to $\left(e^{-\lambda s}e^{-\lambda^{\prime}(t-s)}|X_{s}^{\varepsilon_{1}}-X_{s}^{\varepsilon_{2}}|^{2}\right)_{0\leq s\leq t}$. From $\langle\nabla\psi_{\varepsilon_{1}}(u)-\nabla
\psi_{\varepsilon_{2}}(v),u-v\rangle\geq-(\varepsilon_{1}+\varepsilon_{2})
|\nabla\psi_{\varepsilon_{1}}(u)||\nabla\psi_{\varepsilon_{2}}(v)|$ ( see (\ref{Property of Yosida approximation})-(e) ), similarly to Lemma 5.1 \cite{MC-01}, it follows that
\begin{equation}\label{Estimate 9}
\begin{array}
[c]{l}%
e^{-\lambda T}E|X_{T}^{\varepsilon_{1}}-X_{T}^{\varepsilon_{2}}|^{2}+\bar{\lambda}_{1}\|X^{\varepsilon_{1}}-X^{\varepsilon_{2}}\|^{2}_{M_{\lambda}}\leq K(C_{1}+K)\|Y^{\varepsilon_{1}}-Y^{\varepsilon_{2}}\|^{2}_{M_{\lambda}}\medskip\\
\qquad+(KC_{2}+k_{1}^{2})\|Z^{\varepsilon_{1}}-Z^{\varepsilon_{2}}\|^{2}_{M_{\lambda}}
+2(\varepsilon_{1}+\varepsilon_{2})E{\displaystyle\int_{0}^{T}}e^{-\lambda s}|\nabla\psi_{\varepsilon_{1}}(X_{s}^{\varepsilon_{1}})||\nabla\psi_{\varepsilon_{2}}(X_{s}^{\varepsilon_{2}})|ds,
\end{array}
\end{equation}
where $\bar{\lambda}_{1}=\lambda-K(2+C^{-1}_{1}+C^{-1}_{2})-K^{2}$ and $C_{1},C_{2}$ are positive constants.

We apply again It\^{o}'s formula but now to $\left( e^{-\lambda s}e^{-\lambda^{\prime} (s-t)}|Y_{s}^{\varepsilon_{1}}-Y_{s}^{\varepsilon_{2}}|^{2}\right)_{t\leq s\leq T}$.
Observing that $\langle\nabla\varphi_{\varepsilon_{1}}(u)-\nabla
\varphi_{\varepsilon_{2}}(v),u-v\rangle\geq-(\varepsilon_{1}+\varepsilon_{2})
|\nabla\varphi_{\varepsilon_{1}}(u)||\nabla\varphi_{\varepsilon_{2}}(v)|$, we obtain
\begin{equation}\label{Estimate 11}
\begin{array}
[c]{ll}%
\|Y^{\varepsilon_{1}}-Y^{\varepsilon_{2}}\|^{2}_{M_{\lambda}}\leq B(\bar{\lambda}_{2},T)\Big[
k_{2}^{2}e^{-\lambda T}E|X_{T}^{\varepsilon_{1}}-X_{T}^{\varepsilon_{2}}|^{2}+KC_{3}\|Y^{\varepsilon_{1}}-Y^{\varepsilon_{2}}\|^{2}_{M_{\lambda}}\medskip\\
\qquad\qquad\qquad\qquad
+2(\varepsilon_{1}+\varepsilon_{2})E{\displaystyle\int_{0}^{T}}e^{-\lambda s}|\nabla\varphi_{\varepsilon_{1}}(Y_{s}^{\varepsilon_{1}})||\nabla\varphi_{\varepsilon_{2}}(Y_{s}^{\varepsilon_{2}})|ds\Big],
\end{array}
\end{equation}
\begin{equation}\label{Estimate 12}
\begin{array}
[c]{ll}%
\|Z^{\varepsilon_{1}}-Z^{\varepsilon_{2}}\|^{2}_{M_{\lambda}}\leq \frac{A(\bar{\lambda}_{2},T)}{\alpha}\Big[
k_{2}^{2}e^{-\lambda T}E|X_{T}^{\varepsilon_{1}}-X_{T}^{\varepsilon_{2}}|^{2}+KC_{3}\|Y^{\varepsilon_{1}}-Y^{\varepsilon_{2}}\|^{2}_{M_{\lambda}}\medskip\\
\qquad\qquad\qquad\qquad
+2(\varepsilon_{1}+\varepsilon_{2})E{\displaystyle\int_{0}^{T}}e^{-\lambda s}|\nabla\varphi_{\varepsilon_{1}}(Y_{s}^{\varepsilon_{1}})||\nabla\varphi_{\varepsilon_{2}}(Y_{s}^{\varepsilon_{2}})|ds\Big],
\end{array}
\end{equation}
where $\bar{\lambda}_{2}=-\lambda-2\gamma-K(C^{-1}_{3}+C^{-1}_{4})$ and $C_{3},C_{4}$ are positive constants.
Now from (\ref{Estimate 9})-(\ref{Estimate 12}) as well as $(C1)$,$(C2)$ or $(C1)$,$(C3)$, by using Burkholder-Davis-Gundy (BDG) inequality, we check that there exists a constant $C$ independent of $\varepsilon_{1}$ and $\varepsilon_{2}$, such that (\ref{Lipschitz sequence equation 1}) holds.\hfill
\end{proof}
\subsection{$L^{p}$-estimates for the penalized equations}
We begin our study with the $L^{2}$-estimates for penalized FBSDE (\ref{Penalized FBSDE with two operators}).

\begin{proposition}\label{proposition of difference dominated by increasing function}
Let the assumptions $(H_{1})$-$(H_{5})$ be satisfied. We also assume $(C1)$ and either $(C2)$ or $(C3)$ hold for some $\lambda,\alpha,C_{1},C_{2},C_{3}$ and $C_{4}=\frac{1-\alpha}{K}$. Then
\begin{equation}\label{difference dominated by increasing function}
\begin{array}
[c]{l}%
\|X^{\varepsilon,t,x_{1}}-X^{\varepsilon,t,x_{2}}\|^{2}_{S[t,T]}
+\|Y^{\varepsilon,t,x_{1}}-Y^{\varepsilon,t,x_{2}}\|^{2}_{S[t,T]}\medskip\\
\qquad \qquad\qquad+\|Z^{\varepsilon,t,x_{1}}-Z^{\varepsilon,t,x_{2}}\|^{2}_{M[t,T]}\leq C_{T}|x_{1}-x_{2}|^{2},~x_{1},x_{2}\in\mathbb{R}^{n},
\end{array}
\end{equation}
with a constant $C_{T}$ which is independent of $(t,x_{1},x_{2})$ and $\varepsilon$.
\end{proposition}
\begin{proof}
We put $\hat{X}^{\varepsilon}=X^{\varepsilon,t,x_{1}}-X^{\varepsilon,t,x_{2}}$, $\hat{Y}^{\varepsilon}=Y^{\varepsilon,t,x_{1}}-Y^{\varepsilon,t,x_{2}}$
and $\hat{Z}^{\varepsilon}=Z^{\varepsilon,t,x_{1}}-Z^{\varepsilon,t,x_{2}}$. From (\ref{Estimate 1})-(\ref{Estimate 4}), using $e^{-\lambda t}\leq e^{-(\lambda\wedge0) t}\leq e^{-(\lambda\wedge0) T}$, $A(\lambda,T-t)\leq A(\lambda,T)$ and $B(\lambda,T-t)\leq B(\lambda,T)$, we have the following estimates:
\begin{equation}\label{Estimate 1 for increase}
\begin{array}
[c]{l}%
e^{-\lambda T}E|\hat{X}^{\varepsilon}_{T}|^{2}+\bar{\lambda}_{1}\|\hat{X}^{\varepsilon}\|^{2}_{M_{\lambda}[t,T]}\leq e^{-(\lambda\wedge0) T}|x_{1}-x_{2}|^{2}\medskip\\
\qquad+K(C_{1}+K)\|\hat{Y}^{\varepsilon}\|^{2}_{M_{\lambda}[t,T]}
+(KC_{2}+k_{1}^{2})\|\hat{Z}^{\varepsilon}\|^{2}_{M_{\lambda}[t,T]},
\end{array}
\end{equation}
\begin{equation}\label{Estimate 3 for increase}
\|\hat{Y}^{\varepsilon}\|^{2}_{M_{\lambda}[t,T]}\leq B(\bar{\lambda}_{2},T)\left[
k_{2}^{2}e^{-\lambda T}E|\hat{X}^{\varepsilon}_{T}|^{2}+KC_{3}\|\hat{X}^{\varepsilon}\|^{2}_{M_{\lambda}[t,T]}\right],
\end{equation}
and
\begin{equation}\label{Estimate 4 for increase}
\|\hat{Z}^{\varepsilon}\|^{2}_{M_{\lambda}[t,T]}\leq \frac{A(\bar{\lambda}_{2},T)}{\alpha}\left[
k_{2}^{2}e^{-\lambda T}E|\hat{X}^{\varepsilon}_{T}|^{2}+KC_{3}\|\hat{X}^{\varepsilon}\|^{2}_{M_{\lambda}[t,T]}\right].
\end{equation}
From these estimates, recalling the definition of $\mu(\alpha,T)$, we get
\begin{equation}\label{auxiliary inequality for increase}
(1-\mu(\alpha,T)k_{2}^{2})e^{-\lambda T}E|\hat{X}^{\varepsilon}_{T}|^{2}+(\bar{\lambda}_{1}-\mu(\alpha,T)KC_{3})\|\hat{X}^{\varepsilon}\|^{2}_{M_{\lambda}[t,T]}
\leq e^{-(\lambda\wedge0) T}|x_{1}-x_{2}|^{2},
\end{equation}
from where we obtain
\begin{equation}\label{estimate for X}
e^{-\lambda T}E|\hat{X}^{\varepsilon}_{T}|+\|\hat{X}^{\varepsilon}\|_{M_{\lambda}[t,T]}\leq C_{T}|x_{1}-x_{2}|^{2},
\end{equation}
with
\begin{equation}\label{CT}
\begin{array}
[c]{l}
C_{T}=\max\Bigg\{\frac{e^{-(\lambda\wedge0) T}}{1-\mu(\alpha,T)k_{2}^{2}},
\frac{e^{-(\lambda\wedge0) T}}{\bar{\lambda}_{1}-\mu(\alpha,T)KC_{3}}
\Bigg\}.
\end{array}
\end{equation}
Observe that $C_{T}$ does not depend on $(t,x_{1},x_{2})$ nor on $\varepsilon$. From (\ref{estimate for X}), (\ref{Estimate 3 for increase}) and (\ref{Estimate 4 for increase}), we have
\[
\|\hat{X}^{\varepsilon}\|_{M_{\lambda}[t,T]}+\|\hat{Y}^{\varepsilon}\|_{M_{\lambda}[t,T]}
+\|\hat{Z}^{\varepsilon}\|_{M_{\lambda}[t,T]}\leq C_{T}|x_{1}-x_{2}|^{2}.
\]
Here $C_{T}$ differs from (\ref{CT}), independent of $(t,x_{1},x_{2})$ and $\varepsilon$ and it may vary line by line in the following discussion.
Finally, from It\^{o}'s formula and the BDG inequality, we conclude that
\[
\|\hat{X}^{\varepsilon}\|_{S[t,T]}+\|\hat{Y}^{\varepsilon}\|_{S[t,T]}
+\|\hat{Z}^{\varepsilon}\|_{M[t,T]}\leq C_{T}|x_{1}-x_{2}|^{2}.
\]
The proof is completed now.
\hfill
\end{proof}
\begin{remark} \label{Lipschitz remark}
Similar to Proposition \ref{proposition of difference dominated by increasing function} we show that for all $\varepsilon>0$, $0\leq t\leq T$ and for all $\xi_{1},\xi_{2}\in L^{2}(\Omega,\mathcal{F}_{t},\mathbb{P};\mathbb{R}^{n})$, there is some constant $C_{T}$ independent of $(t,\xi_{1},\xi_{2})$ and $\varepsilon$, such that
\[
\begin{array}
[c]{l}
\quad E^{\mathcal{F}_{t}}\Bigg(\sup\limits_{t\leq r\leq T}|X_{r}^{\varepsilon,t,\xi_{1}}-X_{r}^{\varepsilon,t,\xi_{2}}|^{2}
+\sup\limits_{t\leq r\leq T}|Y_{r}^{\varepsilon,t,\xi_{1}}-Y_{r}^{\varepsilon,t,\xi_{2}}|^{2}
+{\displaystyle\int_{t}^{T}}|Z_{r}^{\varepsilon,t,\xi_{1}}-Z_{r}^{\varepsilon,t,\xi_{2}}|^{2}dr\Bigg)\medskip\\
\leq C_{T}|\xi_{1}-\xi_{2}|^{2},~\mathbb{P}-a.s.
\end{array}
\]
In particular,
$|Y_{t}^{\varepsilon,t,\xi_{1}}-Y_{t}^{\varepsilon,t,\xi_{2}}|\leq C_{T}|\xi_{1}-\xi_{2}|,~a.s.$
\end{remark}

Now we introduce the random field $\theta^{\varepsilon}(t,x):=Y^{\varepsilon,t,x}_{t}$, for $(t,x)\in [0,T]\times\mathbb{R}^{n}$. Then
\begin{equation}\label{Lipschitz for u}
|\theta^{\varepsilon}(t,x)-\theta^{\varepsilon}(t,x^{\prime})|\leq C_{T}|x-x^{\prime}|, ~a.s.
\end{equation}
Moreover, we have the following proposition:
\begin{proposition}\label{proposition for realtion of u and Y}
Let us suppose the assumptions $(H_{1})$-$(H_{5})$ as well as $(C1)$ combined either with $(C2)$ or with $(C3)$ for some $\lambda,\alpha,C_{1},C_{2},C_{3}$ and $C_{4}=\frac{1-\alpha}{K}$. Then, for any $t\in[0,T]$, and $\zeta\in L^{2}(\Omega,\mathcal{F}_{t},\mathbb{P};\mathbb{R}^{m})$, we have
\[\theta^{\varepsilon}(t,\zeta)=Y^{\varepsilon,t,\zeta}_{t}, ~\mathbb{P}-a.s.
\]
\end{proposition}
The proof of the above Proposition can be obtained by combining the arguments of Peng (\cite{P-97}, Theorem 4.7)
with the uniqueness of the solution of our penalized FBSDE.\medskip

In the following discussion, we recall the assumption:
\begin{itemize}
\item[$\left(  H^{\prime}_{3}\right)  $] $k_{1}=0$, i.e., $\sigma$ does not depend on $z$.
\end{itemize}
Under $(H^{\prime}_{3})$, FBSVI (\ref{FBSDE with two operators}) becomes
\begin{equation}\label{FBSVI}
\left\{
\begin{array}
[c]{l}%
dX_{r}+\partial\psi(X_{r})dr\ni b(r,X_{r},Y_{r},Z_{r})dr+\sigma(r,X_{r},Y_{r})dB_{r},\medskip\\
-dY_{r}+\partial\varphi(Y_{r})dr\ni f(r,X_{r},Y_{r},Z_{r})dr-Z_{r}dB_{r},\quad r\in[0,T],\medskip\\
X_{0}=x, \quad Y_{T}=g(X_{T}),
\end{array}
\right.
\end{equation}
and the corresponding penalized FBSDE is
\begin{equation}\label{Penalized FBSVI with two operators}
\left\{
\begin{array}
[c]{l}%
X_{s}^{\varepsilon}+{\displaystyle\int_{0}^{s}}\nabla\psi_{\varepsilon}(X_{r}^{\varepsilon})dr=
x+{\displaystyle\int_{0}^{s}}b(r,X_{r}^{\varepsilon},Y_{r}^{\varepsilon},Z_{r}^{\varepsilon})dr
+{\displaystyle\int_{0}^{s}}\sigma(r,X_{r}^{\varepsilon},Y_{r}^{\varepsilon})dB_{r},\medskip\\
Y_{s}^{\varepsilon}+{\displaystyle\int_{s}^{T}}\nabla\varphi_{\varepsilon}(Y_{r}^{\varepsilon})dr=g(X_{T}^{\varepsilon})+{\displaystyle\int
_{s}^{T}}f(r,X_{r}^{\varepsilon},Y_{r}^{\varepsilon},Z_{r}^{\varepsilon})dr-{\displaystyle\int_{s}^{T}}Z_{r}^{\varepsilon}dB_{r}.
\end{array}
\right.
\end{equation}
Unlike \cite{MC-01}, we need the following uniform $L^{p}$-estimates of the solution of (\ref{Penalized FBSVI with two operators}) in our framework:
\begin{proposition}\label{proposition for high order estimates}
Let the assumptions $(H_{1})$-$(H_{5})$ and $(H^{\prime}_{4}),(H^{\prime}_{5})$ be satisfied. We also assume $(C1)$ and either $(C2)$ or $(C3)$ hold for some $\lambda,\alpha,C_{1},C_{2},C_{3}$ and $C_{4}=\frac{1-\alpha}{K}$. Then, for every $1\leq p\leq \frac{3+\rho_{0}}{2}$, there exists a constant $C$ independent of $\varepsilon$ and $x$, such that
\begin{equation}\label{high order estimates}
\begin{array}
[c]{l}%
\qquad E(\sup\limits_{0\leq r\leq T}|X_{r}^{\varepsilon}|^{2p}+\sup\limits_{0\leq r\leq T}|Y_{r}^{\varepsilon}|^{2p})+E\left\{\left({\displaystyle\int_{0}^{T}}|Z_{r}^{\varepsilon}|^{2}dr\right)^{p}\right\}\medskip\\
\leq CE\left\{|x|^{2p}+|g^{0}|^{2p}+\left({\displaystyle\int_{0}^{T}}|b^{0}(r)|^{2}dr\right)^{p}
+\left({\displaystyle\int_{0}^{T}}|f^{0}(r)|^{2}dr\right)^{p}+
\left({\displaystyle\int_{0}^{T}}|\sigma^{0}(r)|^{2}dr\right)^{p}\right\}.
\end{array}
\end{equation}
\end{proposition}
\begin{proof}
For the proof, we use an approach based on Theorem A.5 Delarue \cite{D-02}. But unlike \cite{D-02}, our coefficients $\psi_{\varepsilon}$ and $\varphi_{\varepsilon}$ depend on $\varepsilon$ so that we have to pay some special care. Let us give a sketch of the proof.

Given $\xi\in L^{2}(\Omega,\mathcal{F}_{t},\mathbb{P};\mathbb{R}^{n})$, we construct the following sequence $\left\{(X^{\varepsilon,k},Y^{\varepsilon,k},Z^{\varepsilon,k})\right\}_{k\ge1}$ of processes:
\begin{equation}\label{approximation sequence}
\left\{
\begin{array}
[c]{l}%
X_{s}^{\varepsilon,k+1}+{\displaystyle\int_{t}^{s}}\nabla\psi_{\varepsilon}(X_{r}^{\varepsilon,k+1})dr
=\xi\medskip\\
\qquad\qquad\qquad\qquad+{\displaystyle\int_{t}^{s}}b(r,X_{r}^{\varepsilon,k+1},Y_{r}^{\varepsilon,k},Z_{r}^{\varepsilon,k})dr
+{\displaystyle\int_{t}^{s}}\sigma(r,X_{r}^{\varepsilon,k+1},Y_{r}^{\varepsilon,k})dB_{r},\medskip\\
Y_{s}^{\varepsilon,k+1}+{\displaystyle\int_{s}^{T}}\nabla\varphi_{\varepsilon}(Y_{r}^{\varepsilon,k+1})dr=g(X_{T}^{\varepsilon,k+1})\medskip\\
\qquad\qquad\qquad\qquad+{\displaystyle\int
_{t}^{T}}f( r,X_{r}^{\varepsilon,k+1},Y_{r}^{\varepsilon,k+1},Z_{r}^{\varepsilon,k+1})dr-{\displaystyle\int_{t}^{T}}Z_{r}^{\varepsilon,k+1}dB_{r},
~s\in[t,T].
\end{array}
\right.
\end{equation}
If we choose $(X^{\varepsilon,0},Y^{\varepsilon,0},Z^{\varepsilon,0})$ s.t.
$E(\sup\limits_{t\leq r\leq T}|X_{r}^{\varepsilon,0}|^{2p}+\sup\limits_{t\leq r\leq T}|Y_{r}^{\varepsilon,0}|^{2p})+E\left({\displaystyle\int_{t}^{T}}|Z_{r}^{\varepsilon,0}|^{2}dr\right)^{p}<\infty$,
following the argument at page 264-265 \cite{D-02}, by using Proposition \ref{Delarue} (see the appendix), we obtain the existence of  a constant $\delta_{K,k_{2},\gamma,p}$ small enough, such that for all $T-t\leq\delta_{K,k_{2},\gamma,p}\wedge T_{0}$ ($T_{0}$ is the constant depending on $K,\gamma,k_{1},k_{2}$ chosen as in Theorem \ref{Solvability of Penalized FBSDE in appendix} of the appendix), we have
\[
\begin{array}
[c]{l}%
\qquad E\sup\limits_{t\leq r\leq T}|X_{r}^{\varepsilon,k}-X_{r}^{\varepsilon,l}|^{2p}+E\sup\limits_{t\leq r\leq T}|Y_{r}^{\varepsilon,k}-Y_{r}^{\varepsilon,l}|^{2p}
+E\left({\displaystyle\int_{t}^{T}}|Z_{r}^{\varepsilon,k}-Z_{r}^{\varepsilon,l}|^{2}dr\right)^{p}
\rightarrow 0,
\end{array}
\]
as $k,l\rightarrow\infty$. This means that $(X^{\varepsilon,k},Y^{\varepsilon,k},Z^{\varepsilon,k})$ converges to some $(X^{\varepsilon},Y^{\varepsilon},Z^{\varepsilon})$ in the sense that
\begin{equation}\label{hign order estimates in small interval}
\begin{array}
[c]{l}%
\qquad E\sup\limits_{t\leq r\leq T}|X_{r}^{\varepsilon,k}-X_{r}^{\varepsilon}|^{2p}+E\sup\limits_{t\leq r\leq T}|Y_{r}^{\varepsilon,k}-Y_{r}^{\varepsilon}|^{2p}
+E\left({\displaystyle\int_{t}^{T}}|Z_{r}^{\varepsilon,k}-Z_{r}^{\varepsilon}|^{2}dr\right)^{p}
\rightarrow 0.
\end{array}
\end{equation}
Then (\ref{approximation sequence}), (\ref{hign order estimates in small interval}) and Theorem \ref{Solvability of Penalized FBSDE in appendix} yield that $(X^{\varepsilon},Y^{\varepsilon},Z^{\varepsilon})$ is the unique solution of FBSDE:
\begin{equation}\label{FBSDE on small interval 1}
\left\{
\begin{array}
[c]{l}%
X_{s}^{\varepsilon}+{\displaystyle\int_{t}^{s}}\nabla\psi_{\varepsilon}(X_{r}^{\varepsilon})dr
=\xi+{\displaystyle\int_{t}^{s}}b(r,X_{r}^{\varepsilon},Y_{r}^{\varepsilon},Z_{r}^{\varepsilon})dr
+{\displaystyle\int_{t}^{s}}\sigma(r,X_{r}^{\varepsilon},Y_{r}^{\varepsilon})dB_{r},\medskip\\
Y_{s}^{\varepsilon}+{\displaystyle\int_{s}^{T}}\nabla\varphi_{\varepsilon}(Y_{r}^{\varepsilon})dr=g(X_{T}^{\varepsilon})+{\displaystyle\int
_{s}^{T}}f(r,X_{r}^{\varepsilon},Y_{r}^{\varepsilon},Z_{r}^{\varepsilon})dr-{\displaystyle\int_{s}^{T}}Z_{r}^{\varepsilon}dB_{r}, ~s\in[t,T].
\end{array}
\right.
\end{equation}
Moreover from (\ref{hign order estimates in small interval}), it follows that, for $T-t\leq\delta_{K,k_{2},\gamma,p}\wedge T_{0}$,
\[
E(\sup\limits_{t\leq r\leq T}|X_{r}^{\varepsilon}|^{2p}+\sup\limits_{t\leq r\leq T}|Y_{r}^{\varepsilon}|^{2p})+E\left({\displaystyle\int_{t}^{T}}|Z_{r}^{\varepsilon}|^{2}dr\right)^{p}<\infty.
\]
Now applying (\ref{difference in small interval}) for $(X_{s}^{\varepsilon},Y_{s}^{\varepsilon},Z_{s}^{\varepsilon})$ and $(0,0,0)$ (Obviously that $(0,0,0)$ is a solution of FBSDE (\ref{FBSDE on small interval 1}) with parameters $(\xi,b,\sigma,f,g)=(0,0,0,0,0)$), we see
there exists $0<\delta_{K,k_{2},\gamma,p}^{\prime}\leq\delta_{K,k_{2},\gamma,p}$ small enough and a constant $C_{K,k_{2},\gamma,p}$ such that, for $T-t\leq\delta_{K,k_{2},\gamma,p}^{\prime}\wedge T_{0}$,
\begin{equation}\label{estimates for small interval}
\begin{array}
[c]{l}%
\qquad E(\sup\limits_{t\leq r\leq T}|X_{r}^{\varepsilon}|^{2p}+\sup\limits_{t\leq r\leq T}|Y_{r}^{\varepsilon}|^{2p})+E\left({\displaystyle\int_{t}^{T}}|Z_{r}^{\varepsilon}|^{2}dr\right)^{p}\medskip\\
\leq C_{K,k_{2},\gamma,p}E\Bigg\{|\xi|^{2p}+|g^{0}|^{2p}
+\left({\displaystyle\int_{t}^{T}}|\sigma^{0}(r)|^{2}dr\right)^{p}
+\left({\displaystyle\int_{t}^{T}}(|b^{0}(r)|+|f^{0}(r)|)
dr\right)^{2p}\Bigg\}.
\end{array}
\end{equation}
If we take $t=0$, we know that the above inequality yields our result for $T$ small enough ($T\leq\delta_{K,k_{2},\gamma,p}^{\prime}\wedge T_{0}$). In what follows, we will show that this inequality can be extended to the whole interval.

We use the notation $\theta^{\varepsilon}(t,x)=Y^{\varepsilon,t,x}_{t}$, $(t,x)\in [0,T]\times\mathbb{R}^{n}$. From (\ref{Lipschitz for u}), it follows that for all $t\in[0,T]$,
\[
|\theta^{\varepsilon}(t,x)-\theta^{\varepsilon}(t,x^{\prime})|\leq C_{T}|x-x^{\prime}|, ~a.s.
\]
Now we divide the interval $[0,T]$ into subintervals defined by
$(t_{i})_{i=0,\ldots,N}$ where $t_{i}=\frac{iT}{N}$, $N\in\mathbb{N}$, such that $\frac{T}{N}\leq \delta^{\prime}_{K,C_{T},\gamma,p}\wedge T_{0}$, where $T_{0}$ is a constant as in Theorem \ref{Solvability of Penalized FBSDE in appendix}, but corresponding to $K,\gamma,k_{1},C_{T}$. From Theorem \ref{Solvability of Penalized FBSDE in appendix} we know that $Y_{t_{i+1}}^{\varepsilon,0,x}=Y_{t_{i+1}}^{\varepsilon,t_{i+1},X_{t_{i+1}}^{\varepsilon,0,x}}$.
Then Proposition \ref{proposition for realtion of u and Y} yields
\[
\theta^{\varepsilon}(t_{i+1},X_{t_{i+1}}^{\varepsilon,0,x})=Y_{t_{i+1}}^{\varepsilon,t_{i+1},X_{t_{i+1}}^{\varepsilon,0,x}}
=Y_{t_{i+1}}^{\varepsilon,0,x},~\mathbb{P}-a.s.
\]

The above discussion and (\ref{estimates for small interval}) allow to follow the argument developed at page 266 \cite{D-02} to obtain by induction that
\[
\begin{array}
[c]{l}%
\qquad E(\sup\limits_{0\leq r\leq T}|X_{r}^{\varepsilon}|^{2p}+\sup\limits_{0\leq r\leq T}|Y_{r}^{\varepsilon}|^{2p})
+E\left({\displaystyle\int_{0}^{T}}|Z_{r}^{\varepsilon}|^{2}dr\right)^{p}\medskip\\
\leq C_{K,C_{T},\gamma,p}E\Bigg\{|x|^{2p}+|g^{0}|^{2p}
+\left({\displaystyle\int_{0}^{T}}|\sigma^{0}(r)|^{2}dr\right)^{p}
+\left({\displaystyle\int_{0}^{T}}(|b^{0}(r)|+|f^{0}(r)|)
dr\right)^{2p}\Bigg\}.
\end{array}
\]
\hfill
\end{proof}
\subsection{Existence and uniqueness of solutions of FBSVIs}
In this subsection, we study the existence and the uniqueness of the solution of FBSVI (\ref{FBSDE with two operators}).
For this, we shall introduce the following condition:
\begin{itemize}
\item[$\left(  H_{6}\right)  $]
There exists a random variable $\zeta\in L^{1}(\Omega)$ and a constant $L$  such that
\[
|\varphi(g(\omega,x))|\leq \zeta(\omega)+L|x|^{3+\rho_{0}}, ~a.s.
\]
\end{itemize}
Based on the idea of \cite{PR-98} and \cite{PR-11}, we give the following auxiliary proposition:
\begin{proposition}\label{Boundedness proposition of two operators}
Let the assumptions $(H_{1})-(H_{6})$ and $(H^{\prime}_{4}),(H^{\prime}_{5})$ be satisfied. Assume also $(C1)$ and either $(C2)$ or $(C3)$ hold for some $\lambda,\alpha,C_{1},C_{2},C_{3}$ and $C_{4}=\frac{1-\alpha}{K}$. Then for all $0<\rho\leq \frac{1+\rho_{0}}{4}\wedge1$,
\[
\begin{array}
[c]{l}%
(i) \quad E\sup\limits_{0\leq r\leq T}|\nabla\psi_{\varepsilon}(X_{r}^{\varepsilon})|^{2+4\rho}\leq \frac{C}{\varepsilon^{2+3\rho}},\medskip\\
(ii) \quad E\sup\limits_{0\leq r\leq T}|X_{r}^{\varepsilon}-J_{\varepsilon,\psi}(X_{r}^{\varepsilon})|^{2+4\rho}\leq C\varepsilon^{\rho},\medskip\\
(iii)  \quad E{\displaystyle\int_{s}^{T}}|\nabla\varphi_{\varepsilon}(Y_{r}^{\varepsilon}))|^{2}dr\leq C, ~\text{ for all } s\in[0,T], \medskip\\
(iv)  \quad  E\varphi(J_{\varepsilon,\varphi}(Y_{s}^{\varepsilon}))+E{\displaystyle\int_{s}^{T}}\varphi(J_{\varepsilon,\varphi}(Y_{r}^{\varepsilon}))dr\leq C,~\text{ for all } s\in[0,T],\medskip\\
(v)  \quad  E|Y_{s}^{\varepsilon}-J_{\varepsilon,\varphi}(Y_{s}^{\varepsilon})|^{2}\leq \varepsilon C,~\text{ for all } s\in[0,T],.
\end{array}
\]
\end{proposition}
\begin{proof}
We first prove $(iii)-(v)$.  Similar to Proposition 2.2 \cite{PR-98}, we can obtain
\begin{equation}\label{Inequality 2}
\begin{array}
[c]{l}%
e^{-\lambda s}\varphi_{\varepsilon}(Y^{\varepsilon}_{s})
+{\displaystyle\int_{s}^{T}}e^{-\lambda r}|\nabla\varphi_{\varepsilon
}(Y^{\varepsilon}_{r})|^{2}dr
\leq e^{-\lambda T}\varphi_{\varepsilon}(Y^{\varepsilon}_{T})-{\displaystyle\int_{s}^{T}}e^{-\lambda r}\langle\nabla\varphi_{\varepsilon}( Y^{\varepsilon}_{r}),Z^{\varepsilon}_{r}dB_{r}\rangle\medskip\\
\qquad\qquad\qquad +{\displaystyle\int_{s}^{T}}
e^{-\lambda r}\langle\nabla\varphi_{\varepsilon}(Y^{\varepsilon}_{r}),
|\lambda|Y^{\varepsilon}_{r}+f(r,X^{\varepsilon}_{r},Y^{\varepsilon}_{r},Z^{\varepsilon}_{r})\rangle dr,
\end{array}
\end{equation}
and then
\[
\begin{array}
[c]{l}%
\frac{1}{2}{\displaystyle\int_{s}^{T}}e^{-\lambda r}E|\nabla\varphi_{\varepsilon
}(Y^{\varepsilon}_{r})|^{2}dr
\leq e^{-\lambda T}E\varphi(g(X^{\varepsilon}_{T}))\medskip\\
\qquad\qquad\qquad +E{\displaystyle\int_{s}^{T}}
e^{-\lambda r}\big[4|\eta(t)|^{2}+4K^{2}(|X^{\varepsilon}_{r}|^{2}+|Z^{\varepsilon}_{r}|^{2})+(4L^{2}+|\lambda|^{2})|Y^{\varepsilon}_{r}|^{2}\big]dr.
\end{array}
\]
By using Proposition \ref{proposition for high order estimates} and $(H_{6})$, we have
${\displaystyle\int_{s}^{T}}e^{-\lambda r}E|\nabla\varphi_{\varepsilon}(Y_{r}^{\varepsilon}))|^{2}dr\leq C(1+|x|^{3+\rho_{0}}),$ which yields $(iii)$. Here $C$ is a constant independent of $\varepsilon$ and $x$.

From (\ref{Inequality 2}), $(iii)$ and $\varphi(J_{\varepsilon,\varphi}(y))\leq\varphi_{\varepsilon}(y)$, we get $Ee^{-\lambda t}\varphi(J_{\varepsilon,\varphi}(Y_{t}^{\varepsilon}))\leq C$, for any $t\in[0,T]$. Thus $(iv)$ follows easily.

Moreover, since $|y-J_{\varepsilon,\varphi}(y)|^{2}=|\nabla\varphi_{\varepsilon}(y)|^{2}\leq2\varepsilon\varphi_{\varepsilon}(y)$, $y\in\mathbb{R}^{m}$, $(v)$ is obtained from $(iv)$.

Now we are focusing on the proof of $(i)$ and $(ii)$. We shall follow the argument as in Theorem 4.20 \cite{PR-11}, so we only give a sketch of the proof here. Since $\psi_{\varepsilon}$ is a function of class $C^{1}(\mathbb{R}^{n};\mathbb{R}_{+})$
and the gradient $\nabla\psi_{\varepsilon}(u)$ is a Lipschitz function with Lipschitz constant $1/\varepsilon$, from Remark \ref{remark for extend ito formula} (see the appendix), we have
\[
\begin{array}
[c]{l}
\qquad \psi^{1+2\rho}_{\varepsilon}(X^{\varepsilon}_{s})
+(1+2\rho){\displaystyle\int_{0}^{s}}\psi^{2\rho}_{\varepsilon}(X^{\varepsilon}_{r})|\nabla\psi_{\varepsilon}(X^{\varepsilon}_{r})|^{2}dr\medskip\\
\leq \psi^{1+2\rho}_{\varepsilon}(x)
+(1+2\rho){\displaystyle\int_{0}^{s}}\psi^{2\rho}_{\varepsilon}(X^{\varepsilon}_{r})
\langle\nabla\psi_{\varepsilon}(X^{\varepsilon}_{r}),b(r,X^{\varepsilon}_{r},Y^{\varepsilon}_{r},Z^{\varepsilon}_{r})\rangle dr\medskip\\
\qquad +\frac{1+2\rho}{2\varepsilon}{\displaystyle\int_{0}^{s}}
\psi^{2\rho}_{\varepsilon}(X^{\varepsilon}_{r})|\sigma(r,X^{\varepsilon}_{r},Y^{\varepsilon}_{r})|^{2}dr\medskip\\
\qquad +\rho(1+2\rho){\displaystyle\int_{0}^{s}}
\psi^{2\rho-1}_{\varepsilon}(X^{\varepsilon}_{r})|\nabla\psi_{\varepsilon}(X^{\varepsilon}_{r})|^{2}
|\sigma(r,X^{\varepsilon}_{r},Y^{\varepsilon}_{r})|^{2}dr\medskip\\
\qquad +(1+2\rho){\displaystyle\int_{0}^{s}}\psi^{2\rho}_{\varepsilon}(X^{\varepsilon}_{r})
\langle\nabla\psi_{\varepsilon}(X^{\varepsilon}_{r}),\sigma(r,X^{\varepsilon}_{r},Y^{\varepsilon}_{r})dB_{r}\rangle. \medskip\\
\end{array}
\]
Then, similarly, following the argument as in Theorem 4.20 \cite{PR-11}, we can prove that
there exists a constant $C$ independent of $\varepsilon$ and $x$, such that
\begin{equation}\label{equation 12}
\begin{array}
[c]{l}
E\sup\limits_{0\leq r\leq T}\psi^{1+2\rho}_{\varepsilon}(X^{\varepsilon}_{r})
+E{\displaystyle\int_{0}^{T}}\psi^{2\rho}_{\varepsilon}(X^{\varepsilon}_{r})|\nabla\psi_{\varepsilon}(X^{\varepsilon}_{r})|^{2}dr
\leq 2\psi^{1+2\rho}_{\varepsilon}(x)+\frac{C}{\varepsilon}E{\displaystyle\int_{0}^{T}}\psi^{2\rho}_{\varepsilon}(X^{\varepsilon}_{r})A^{\varepsilon}_{r}dr.
\end{array}
\end{equation}
Here
\[A^{\varepsilon}_{r}=|b^{0}(r)|^{2}+|\sigma^{0}(r)|^{2}+|X^{\varepsilon}_{r}|^{2}+|Y^{\varepsilon}_{r}|^{2}+|X^{\varepsilon}_{r}||Z^{\varepsilon}_{r}|
~r\in[0,T].
\]
Moreover, let $r_{0}=\frac{1+\rho}{\rho}$. Then from Young's inequality,
\[
\begin{array}
[c]{l}
\frac{C}{\varepsilon}\psi^{2\rho}_{\varepsilon}(X^{\varepsilon}_{r})A^{\varepsilon}_{r}
=\psi^{2\rho-\frac{2}{r_{0}}}_{\varepsilon}(X^{\varepsilon}_{r})
\psi^{\frac{2}{r_{0}}}_{\varepsilon}(X^{\varepsilon}_{s})A^{\varepsilon}_{r}\frac{C}{\varepsilon}
\leq \psi^{2\rho-\frac{2}{r_{0}}}_{\varepsilon}(X^{\varepsilon}_{r})
|\nabla\psi_{\varepsilon}(X^{\varepsilon}_{r})|^{\frac{2}{r_{0}}}|X^{\varepsilon}_{r}|^{\frac{2}{r_{0}}}A^{\varepsilon}_{r}\frac{C}{\varepsilon}\medskip\\
\leq \frac{1}{r_{0}}\left(\psi^{2\rho-\frac{2}{r_{0}}}_{\varepsilon}(X^{\varepsilon}_{r})
|\nabla\psi_{\varepsilon}(X^{\varepsilon}_{r})|^{\frac{2}{r_{0}}}\right)^{r_{0}}+
\frac{r_{0}-1}{r_{0}}\left(|X^{\varepsilon}_{r}|^{\frac{2}{r_{0}}}A^{\varepsilon}_{r}\frac{C}{\varepsilon}\right)^{\frac{r_{0}}{r_{0}-1}}\medskip\\
=\frac{\rho}{1+\rho}\psi^{2\rho}_{\varepsilon}(X^{\varepsilon}_{r})
|\nabla\psi_{\varepsilon}(X^{\varepsilon}_{r})|^{2}
+\frac{1}{\varepsilon^{1+\rho}}\frac{C^{1+\rho}}{1+\rho}|X^{\varepsilon}_{r}|^{2\rho}|A^{\varepsilon}_{r}|^{1+\rho}.
\end{array}
\]
From the above estimate, (\ref{high order estimates}), (\ref{equation 12}) and Young's inequality it follows that
\begin{equation}\label{equation 13}
\begin{array}
[c]{l}
E\sup\limits_{0\leq r\leq T}\psi^{1+2\rho}_{\varepsilon}(X^{\varepsilon}_{r})
\leq 2\psi^{1+2\rho}_{\varepsilon}(x)
+\frac{C}{\varepsilon^{1+\rho}}E{\displaystyle\int_{0}^{T}}|X^{\varepsilon}_{r}|^{2\rho}|A^{\varepsilon}_{r}|^{1+\rho}dr
\leq 2\psi^{1+2\rho}_{\varepsilon}(x)\medskip\\
\qquad+\frac{C}{\varepsilon^{1+\rho}}E{\displaystyle\int_{0}^{T}}
\Big(|b^{0}(r)|^{2+4\rho}+|\sigma^{0}(r)|^{2+4\rho}
+|X^{\varepsilon}_{r}|^{2+4\rho}+|Y^{\varepsilon}_{r}|^{2+4\rho}
+|X^{\varepsilon}_{r}|^{1+3\rho}|Z^{\varepsilon}_{r}|^{1+\rho}\Big)dr\medskip\\
\leq 2\psi^{1+2\rho}_{\varepsilon}(x)+\frac{C}{\varepsilon^{1+\rho}}\Bigg[E\sup\limits_{0\leq r\leq T}|X_{r}^{\varepsilon}|^{2+4\rho}+E\sup\limits_{0\leq r\leq T}|Y_{r}^{\varepsilon}|^{2+4\rho}\medskip\\
\qquad\qquad+E{\displaystyle\int_{0}^{T}}
\Big(|b^{0}(r)|^{2+4\rho}+|\sigma^{0}(r)|^{2+4\rho}\Big)dr+E\left\{\sup\limits_{0\leq r\leq T}|X^{\varepsilon}_{r}|^{1+3\rho}\left({\displaystyle\int_{0}^{T}}
|Z^{\varepsilon}_{r}|^{2}dr\right)^{\frac{1+\rho}{2}}\right\}\Bigg]\medskip\\
\leq 2\psi^{1+2\rho}(x)+\frac{C}{\varepsilon^{1+\rho}}\Bigg[E\sup\limits_{0\leq r\leq T}|X_{r}^{\varepsilon}|^{2+4\rho}+E\sup\limits_{0\leq r\leq T}|Y_{r}^{\varepsilon}|^{2+4\rho}\medskip\\
\qquad\qquad\qquad\qquad+E{\displaystyle\int_{0}^{T}}
\Big(|b^{0}(r)|^{2+4\rho}+|\sigma^{0}(r)|^{2+4\rho}\Big)dr+E\left\{\left({\displaystyle\int_{0}^{T}}
|Z^{\varepsilon}_{r}|^{2}dr\right)^{1+2\rho}\right\}\Bigg]\medskip\\
\leq \frac{C}{\varepsilon^{1+\rho}}\left(1+|x|^{3+\rho_{0}}+\psi^{1+2\rho}(x)\right),
\end{array}
\end{equation}
where C is a constant independent of $\varepsilon$ and $x$ and it can vary from line to line. (We can assume that $\varepsilon<1$. Also recall that $0\leq \rho\leq \frac{1+\rho_{0}}{4}\wedge1$ so that $2+4\rho\leq 3+\rho_{0}$).

Finally, since $\frac{\varepsilon}{2}|\nabla\psi_{\varepsilon}(X^{\varepsilon}_{r})|^{2}\leq \psi_{\varepsilon}(X^{\varepsilon}_{r})$ and
$X^{\varepsilon}_{r}-J_{\varepsilon,\psi}(X^{\varepsilon}_{r})=\varepsilon\nabla\psi_{\varepsilon}(X^{\varepsilon}_{r})$,
we obtain $(i)$ and $(ii)$ directly from (\ref{equation 13}). \hfill
\end{proof}
\begin{proposition}\label{Lipschitz proposition of two operators}
Under the assumptions of Proposition \ref{Boundedness proposition of two operators}, setting $\frac{1-\rho_{0}}{4+4\rho_{0}}\vee0<\rho\leq\frac{1+\rho_{0}}{4}\wedge1$, we have
\[
\begin{array}
[c]{l}
(i) \quad E{\displaystyle\int_{0}^{T}}\Big\{|X_{r}^{\varepsilon_{1}}-X_{r}^{\varepsilon_{2}}|^{2}+|Y_{r}^{\varepsilon_{1}}-Y_{r}^{\varepsilon_{2}}|^{2}
+|Z_{r}^{\varepsilon_{1}}-Z_{r}^{\varepsilon_{2}}|^{2}\Big\}dr\leq C\left(\varepsilon_{1}^{\frac{\rho}{2+4\rho}}
+\varepsilon_{2}^{\frac{\rho}{2+4\rho}}\right),\medskip\\
(ii)  \quad  E\Big\{\sup\limits_{0\leq r\leq T}|X_{r}^{\varepsilon_{1}}-X_{r}^{\varepsilon_{2}}|^{2}+\sup\limits_{0\leq r\leq T}|Y_{r}^{\varepsilon_{1}}-Y_{r}^{\varepsilon_{2}}|^{2}\Big\}\leq  C\left(\varepsilon_{1}^{\frac{\rho}{2+4\rho}}
+\varepsilon_{2}^{\frac{\rho}{2+4\rho}}\right).
\end{array}
\]
Here $C$ is a constant which depends neither on $\varepsilon_{1}$ nor on $\varepsilon_{2}$.
\end{proposition}
\begin{proof}
From Proposition \ref{Lipschitz sequence 1} we have (\ref{Lipschitz sequence equation 1}). In the same manner as in \cite{PR-11},
applying the H\"{o}lder inequality to the right-hand side of (\ref{Lipschitz sequence equation 1}), it follows that
\begin{equation}\label{equation 15}
\begin{array}
[c]{l}
(\varepsilon_{1}+\varepsilon_{2})E{\displaystyle\int_{0}^{T}}
|\nabla\psi_{\varepsilon_{1}}(X^{\varepsilon_{1}}_{r})||\nabla\psi_{\varepsilon_{2}}(X^{\varepsilon_{2}}_{r})|dr\medskip\\
\leq \varepsilon_{1}\left(E\sup\limits_{0\leq r\leq T}|\nabla\psi_{\varepsilon_{1}}(X^{\varepsilon_{1}}_{r})|^{2+4\rho}\right)^{\frac{1}{2+4\rho}}
\left[E\left({\displaystyle\int_{0}^{T}}|\nabla\psi_{\varepsilon_{2}}(X^{\varepsilon_{2}}_{r})|dr\right)
^{\frac{2+4\rho}{1+4\rho}}\right]^{\frac{1+4\rho}{2+4\rho}}\medskip\\
\qquad+\varepsilon_{2}\left(E\sup\limits_{0\leq r\leq T}|\nabla\psi_{\varepsilon_{2}}(X^{\varepsilon_{2}}_{r})|^{2+4\rho}\right)^{\frac{1}{2+4\rho}}
\left[E\left({\displaystyle\int_{0}^{T}}|\nabla\psi_{\varepsilon_{1}}(X^{\varepsilon_{1}}_{r})|dr\right)
^{\frac{2+4\rho}{1+4\rho}}\right]^{\frac{1+4\rho}{2+4\rho}}.
\end{array}
\end{equation}
Now we calculate $E\left\{\left({\displaystyle\int_{0}^{T}}
|\nabla\psi_{\varepsilon}(X^{\varepsilon}_{r})|dr\right)^{\frac{2+4\rho}{1+4\rho}}\right\}$.
From (\ref{Property f}) we know
\begin{equation}\label{to proof Bounded variation 1}
\begin{array}
[c]{l}
2r_{0}{\displaystyle\int_{0}^{T}}|\nabla\psi_{\varepsilon}(X^{\varepsilon}_{s})|ds\leq 2{\displaystyle\int_{0}^{T}}\langle\nabla\psi_{\varepsilon}(X^{\varepsilon}_{s}),X^{\varepsilon}_{s}-u_{0}\rangle ds+2M_{0}T.
\end{array}
\end{equation}
By applying It\^{o}'s formula to $|X^{\varepsilon}_{r}-u_{0}|^{2}$, we obtain
\begin{equation}\label{to proof Bounded variation 2}
\begin{array}
[c]{l}
|X^{\varepsilon}_{s}-u_{0}|^{2}+2{\displaystyle\int_{0}^{s}}\langle\nabla\psi_{\varepsilon}(X^{\varepsilon}_{r}),X^{\varepsilon}_{r}-u_{0}\rangle dr
=|x-u_{0}|^{2}+{\displaystyle\int_{0}^{s}}|\sigma(r,X^{\varepsilon}_{r},Y^{\varepsilon}_{r})|^{2} dr\medskip\\
\qquad+{\displaystyle\int_{0}^{s}}2\langle X^{\varepsilon}_{r}-u_{0},b(r,X^{\varepsilon}_{r},Y^{\varepsilon}_{r},Z^{\varepsilon}_{r})\rangle dr
+{\displaystyle\int_{0}^{s}}2\langle X^{\varepsilon}_{r}-u_{0},\sigma(r,X^{\varepsilon}_{r},Y^{\varepsilon}_{r})dB_{r}\rangle.
\end{array}
\end{equation}
Thus
\[
\begin{array}
[c]{l}
2r_{0}{\displaystyle\int_{0}^{T}}|\nabla\psi_{\varepsilon}(X^{\varepsilon}_{r})|dr\leq 2|x|^{2}+2|u_{0}|^{2}+2M_{0}T
+{\displaystyle\int_{0}^{T}}|\sigma(r,X^{\varepsilon}_{r},Y^{\varepsilon}_{r})|^{2} dr\medskip\\
\qquad+{\displaystyle\int_{0}^{T}}2\langle X^{\varepsilon}_{r}-u_{0},b(r,X^{\varepsilon}_{r},Y^{\varepsilon}_{r},Z^{\varepsilon}_{r})\rangle dr
+{\displaystyle\int_{0}^{T}}2\langle X^{\varepsilon}_{r}-u_{0},\sigma(r,X^{\varepsilon}_{r},Y^{\varepsilon}_{r})dB_{r}\rangle.
\end{array}
\]
Consequently, for fixed $u_{0}\in\mathbb{R}^{n}$, we deduce from Proposition \ref{proposition for high order estimates}, that for all $1\leq q\leq \frac{3+\rho_{0}}{2}$,
\begin{equation}\label{equation 16}
\begin{array}
[c]{l}
E\left\{\left(2r_{0}{\displaystyle\int_{0}^{T}}|\nabla\psi_{\varepsilon}(X^{\varepsilon}_{r})|dr\right)^{q}\right\}
\leq C\Bigg[1+E\sup\limits_{0\leq r\leq T}|X^{\varepsilon}_{r}|^{2q}+E\sup\limits_{0\leq r\leq T}|Y^{\varepsilon}_{r}|^{2q}\medskip\\
\quad+E\left\{\left({\displaystyle\int_{0}^{T}}|Z_{r}^{\varepsilon}|^{2}dr\right)^{q}\right\}
+E\left({\displaystyle\int_{0}^{T}}
|b^{0}(r)|^{2}dr\right)^{q}+E\left({\displaystyle\int_{0}^{T}}
|\sigma^{0}(r)|^{2}dr\right)^{q}\Bigg]\leq C.
\end{array}
\end{equation}
We choose now $q=\frac{2+4\rho}{1+4\rho}$ and we observe that $q\leq\frac{3+\rho_{0}}{2}$. (Indeed, recall that $\rho\ge\frac{1-\rho_{0}}{4+4\rho_{0}}\vee 0$ ). Then we obtain from (\ref{equation 15}),
(\ref{equation 16}) and Proposition \ref{Boundedness proposition of two operators} $(i)$, that
\begin{equation}\label{equation 17}
\begin{array}
[c]{l}
\quad (\varepsilon_{1}+\varepsilon_{2})E{\displaystyle\int_{0}^{T}}
|\nabla\psi_{\varepsilon_{1}}(X^{\varepsilon_{1}}_{r})||\nabla\psi_{\varepsilon_{2}}(X^{\varepsilon_{2}}_{r})|dr\medskip\\
\leq C\varepsilon_{1}\left(E\sup\limits_{0\leq r\leq T}|\nabla\psi_{\varepsilon_{1}}(X^{\varepsilon_{1}}_{r})|^{2+4\rho}\right)^{\frac{1}{2+4\rho}}
+C\varepsilon_{2}\left(E\sup\limits_{0\leq r\leq T}|\nabla\psi_{\varepsilon_{2}}(X^{\varepsilon_{2}}_{r})|^{2+4\rho}\right)^{\frac{1}{2+4\rho}}\medskip\\
\leq C\varepsilon_{1}\left(\frac{C}{\varepsilon_{1}^{2+3\rho}}\right)^{\frac{1}{2+4\rho}}
+C\varepsilon_{2}\left(\frac{C}{\varepsilon_{2}^{2+3\rho}}\right)^{\frac{1}{2+4\rho}}
\leq C\varepsilon_{1}^{\frac{\rho}{2+4\rho}}
+C\varepsilon_{2}^{\frac{\rho}{2+4\rho}}
\end{array}
\end{equation}
On the other hand, using Proposition \ref{Boundedness proposition of two operators} $(iii)$, it follows
\begin{equation}\label{equation 18}
\begin{array}
[c]{l}
\quad (\varepsilon_{1}+\varepsilon_{2})E{\displaystyle\int_{0}^{T}}
|\nabla\varphi_{\varepsilon_{1}}(X^{\varepsilon_{1}}_{r})||\nabla\varphi_{\varepsilon_{2}}(X^{\varepsilon_{2}}_{r})|dr\medskip\\
\leq \frac{\varepsilon_{1}+\varepsilon_{2}}{2}E{\displaystyle\int_{0}^{T}}
\left(|\nabla\varphi_{\varepsilon_{1}}(X^{\varepsilon_{1}}_{r})|^{2}+|\nabla\varphi_{\varepsilon_{2}}(X^{\varepsilon_{2}}_{r})|^{2}\right)dr
\leq C(\varepsilon_{1}+\varepsilon_{2}).
\end{array}
\end{equation}
Consequently, (\ref{Lipschitz sequence equation 1}), (\ref{equation 17}) and (\ref{equation 18}) allow to complete the proof.\hfill
\end{proof}

Now we are able to give our main results:
\begin{theorem} \label{Existence and uniqueness of two operators}
Suppose $(H_{1})-(H_{6})$ and $(H^{\prime}_{5})$ are satisfied and $(H^{\prime}_{4})$ holds with $\rho_{0}\ge1$. Moreover, we assume that $(C1)$ and either $(C2)$ or $(C3)$ hold for some $\lambda,\alpha,C_{1},C_{2},C_{3}$ and $C_{4}=\frac{1-\alpha}{K}$. Then there exists a unique solution $(X,Y,Z,V,U)$
of FBSVI (\ref{FBSVI}).
\end{theorem}
\begin{proof} The uniqueness is a consequence of Proposition \ref{Difference 2}. Thus, it remains to show the existence.
From Proposition \ref{Lipschitz proposition of two operators}, we know that there exist $X\in S^{2}_{n}$, $Y\in S^{2}_{m}$, and $Z\in M^{2}_{m\times d}$,  such that
\begin{equation}\label{equation 21}
\begin{array}
[c]{l}%
\lim\limits_{\varepsilon\rightarrow0}E\Big\{\sup\limits_{0\leq r\leq T}|X_{r}^{\varepsilon}-X_{r}|^{2}+\sup\limits_{0\leq r\leq T}|Y_{r}^{\varepsilon}-Y_{r}|^{2}+{\displaystyle\int_{0}^{T}}|Z_{r}^{\varepsilon}-Z_{r}|^{2}dr\Big\}=0,
\end{array}
\end{equation}
and from Proposition \ref{Boundedness proposition of two operators} $(ii)$ and $(v)$, we have
\begin{equation}\label{equation 22}
\begin{array}
[c]{l}
\quad\lim\limits_{\varepsilon\rightarrow0}E\sup\limits_{0\leq r\leq T}|J_{\varepsilon,\psi}(X_{r}^{\varepsilon})-X_{r}|^{2}=0,
\quad\lim\limits_{\varepsilon\rightarrow0}E|Y_{r}-J_{\varepsilon,\varphi}(Y_{r}^{\varepsilon})|^{2}=0, \text{ for all } r\in[0,T],\medskip\\
\text{ and }\lim\limits_{\varepsilon\rightarrow0}{\displaystyle\int_{0}^{T}}E|Y_{r}-J_{\varepsilon,\varphi}(Y_{r}^{\varepsilon})|^{2}dr=0.
\end{array}
\end{equation}
Let us define $V_{s}^{\varepsilon}:={\displaystyle\int_{0}^{s}}\nabla\psi_{\varepsilon}(X_{r}^{\varepsilon})dr$. Then from (\ref{Penalized FBSVI with two operators}),
\[
X_{s}^{\varepsilon}+V_{s}^{\varepsilon}=x+{\displaystyle\int_{0}^{s}}b(r,X_{r}^{\varepsilon},Y_{r}^{\varepsilon},Z_{r}^{\varepsilon})dr
+{\displaystyle\int_{0}^{s}}\sigma(r,X_{r}^{\varepsilon},Y_{r}^{\varepsilon})dB_{r},~s\in[0,T].
\]
The Lipschitz condition for $b$ and $\sigma$, Proposition \ref{Lipschitz proposition of two operators} and the BDG inequality yield that
\[
E\sup\limits_{0\leq s\leq T}|V_{s}^{\varepsilon_{1}}-V_{s}^{\varepsilon_{2}}|^{2}\leq C\left(\varepsilon_{1}^{\frac{\rho}{2+4\rho}}
+\varepsilon_{2}^{\frac{\rho}{2+4\rho}}\right),~\varepsilon_{1},\varepsilon_{2}>0.
\]
Consequently, there exists $V\in S^{2}_{n}$, such that
\begin{equation}\label{limit equation for V}
\lim\limits_{\varepsilon\rightarrow0}E\sup\limits_{0\leq s\leq T}|V_{s}^{\varepsilon}-V_{s}|^{2}=0
\end{equation}
and
\[
X_{s}+V_{s}=x+{\displaystyle\int_{0}^{s}}b(r,X_{r},Y_{r},Z_{r})dr+{\displaystyle\int_{0}^{s}}\sigma(r,X_{r},Y_{r})dB_{r}.
\]
Then, from (\ref{equation 16}) and $V^{\varepsilon}(0)=0$ we have
$E\left\{\|V^{\varepsilon}\|^{q}_{BV([0,T];\mathbb{R}^{n})}\right\}=E\updownarrow V^{\varepsilon}\updownarrow^{q}_{[0,T]}\leq C$, for $1\leq q\leq\frac{3+\rho_{0}}{2}$.
In particular, $E\left\{\|V^{\varepsilon}\|_{BV([0,T];\mathbb{R}^{n})}\right\}\leq C$.
Moreover, if $\rho_{0}\ge1$, then also
\[
E\left\{\|V^{\varepsilon}\|^{2}_{BV([0,T];\mathbb{R}^{n})}\right\}\leq C.
\]
Recalling Proposition 1.25 \cite{PR-11} (see also Proposition \ref{proposition of pardoux and rascanu 2011} in the appendix)
as well as (\ref{equation 21}) and (\ref{limit equation for V}), we have
\begin{equation}\label{square integrable of bounded variation}
E\left\{\|V\|^{2}_{BV([0,T];\mathbb{R}^{n})}\right\}=E\updownarrow V\updownarrow^{2}_{[0,T]}\leq C,
\end{equation}
and for all $0\leq s\leq t\leq T$,
\begin{equation}\label{equation 24}
{\displaystyle\int_{s}^{t}}\langle X_{r}^{\varepsilon},dV_{r}^{\varepsilon}\rangle
\xrightarrow[]{\;\mathbb{P}\;}{\displaystyle\int_{s}^{t}}\langle X_{r},dV_{r}\rangle, ~as~ \varepsilon\rightarrow0.
\end{equation}
From (\ref{Property of Yosida approximation}-(a)) and the convexity of $\psi_{\varepsilon}$, we have
\[
\begin{array}
[c]{l}%
\psi(J_{\varepsilon,\psi}(x))\leq \psi_{\varepsilon}(x)\leq \psi_{\varepsilon}(z)+\langle x-z,\nabla\psi_{\varepsilon}(x)\rangle
\leq \psi(z)+\langle x-z,\nabla\psi_{\varepsilon}(x)\rangle, ~z\in\mathbb{R}^{n},
\end{array}
\]
Then for all $0\leq s\leq t\leq T$, it follows that
\[
\begin{array}
[c]{l}
{\displaystyle\int_{s}^{t}}\psi(J_{\varepsilon,\psi}(X_{r}^{\varepsilon}))dr\leq (t-s)\psi(z)+{\displaystyle\int_{s}^{t}}\langle X_{r}^{\varepsilon}-z,dV_{r}^{\varepsilon}\rangle,~z\in\mathbb{R}^{n}.
\end{array}
\]
From (\ref{equation 22}), (\ref{equation 24}), Fatou's lemma and the fact that $\psi$ is l.s.c., letting $\varepsilon\rightarrow0$, we deduce
\begin{equation}\label{equation 25}
{\displaystyle\int_{s}^{t}}\psi(X_{r})dr\leq (t-s)\psi(z)+{\displaystyle\int_{s}^{t}}\langle X_{r}-z,dV_{r}\rangle, z\in\mathbb{R}^{n}, ~0\leq s\leq t\leq T, ~a.s.
\end{equation}
This proves the inequality in Definition \ref{definition of the solution of FBSDE with two operators} $(a_{2})$.
Moreover, taking $z=0$ in (\ref{equation 25}), we have
\[
\begin{array}
[c]{rl}
0\leq E{\displaystyle\int_{0}^{T}}\psi(X_{r})dr\leq E\Big|{\displaystyle\int_{0}^{T}}\langle X_{r},dV_{r}\rangle\Big|
\leq \frac{1}{2}E\sup\limits_{0\leq t\leq T}|X_{r}|^{2}
+\frac{1}{2}E\updownarrow V\updownarrow^{2}_{[0,T]}
<\infty,
\end{array}
\]
i.e.,  $X$ takes its values in $Dom\psi$ and $\psi(X)\in L^{1}(\Omega\times[0,T];\mathbb{R})$.

On the other hand, for any $\varepsilon>0$, we define $U^{\varepsilon}_{s}=\nabla\varphi_{\varepsilon}(Y^{\varepsilon}_{s})$ and $\bar{U}^{\varepsilon}_{s}={\displaystyle\int_{0}^{s}}U^{\varepsilon}_{r}dr$. Then from (\ref{Penalized FBSVI with two operators}) and Proposition \ref{Lipschitz proposition of two operators} we deduce the existence of a $\bar{U}\in S^{2}_{m}$, s.t.
\begin{equation}\label{equation 1}
\lim\limits_{\varepsilon\rightarrow0}E\sup\limits_{0\leq s\leq T}|\bar{U}_{s}^{\varepsilon}-\bar{U}_{s}|^{2}=0.
\end{equation}
Moreover, Proposition \ref{Boundedness proposition of two operators} $(iii)$ yields
$\sup\limits_{\varepsilon>0}\mathbb{E}{\displaystyle\int_{0}^{T}}|U^{\varepsilon}_{r}|^{2}dr\leq C$. Consequently, the sequence $\{\bar{U}^{\varepsilon}\}_{\varepsilon>0}$ is bounded in $L^{2}(\Omega;H^{1}(0,T))$. Thus there exists a subsequence which converges weakly to a limit in $L^{2}(\Omega;H^{1}(0,T))$. But from (\ref{equation 1}) we conclude that this limit is nothing but $\bar{U}_{t}$, and the whole sequence $\{\bar{U}^{\varepsilon}\}_{\varepsilon>0}$ converges weakly to $\bar{U}$. Moreover, $\bar{U}$ takes the form $\bar{U}_{s}={\displaystyle\int_{0}^{s}}U_{r}dr,~s\in[0,T]$, where
$U^{\varepsilon}\xrightarrow[\text{weakly}]{\;L^{2}(\Omega\times[0,T])\;}U$. Now we take the limit in probability in FBSDE (\ref{Penalized FBSVI with two operators}), and we get
\[
Y_{s}+{\displaystyle\int_{s}^{T}}U_{r}dr=g(X_{T})+{\displaystyle\int
_{s}^{T}}f(r,X_{r},Y_{r},Z_{r})dr-{\displaystyle\int_{s}^{T}}Z_{r}dB_{r},~\text{ for all } s\in[0,T],\;a.s..
\]
Finally, it remains to show that $(Y_{s},U_{s})\in\partial\varphi$. In fact, from $U^{\varepsilon}_{s}\in\partial\varphi(J_{\varepsilon,\varphi}Y^{\varepsilon}_{s})$, it follows ,
\[
\langle U^{\varepsilon}_{s},v_{s}-J_{\varepsilon,\varphi}(Y^{\varepsilon}_{s})\rangle+\varphi
(J_{\varepsilon}(Y^{\varepsilon}_{s}))\leq\varphi(v_{s})~\text{ for all } v\in M^{2}_{m}.
\]
Then integrating both sides from $a$ to $b$,   for all $0\leq a\leq b\leq T$, we obtain
\[
{\displaystyle\int_{a}^{b}}\langle U^{\varepsilon}_{s},v_{s}-J_{\varepsilon,\varphi}(Y^{\varepsilon}_{s})\rangle ds+{\displaystyle\int_{a}^{b}}\varphi
(J_{\varepsilon,\varphi}(Y^{\varepsilon}_{s}))ds\leq {\displaystyle\int_{a}^{b}}\varphi(v_{s})ds.
\]
Let us take now the limit as $\varepsilon\rightarrow 0$. By using (\ref{equation 22}), the weak convergence of $U^{\varepsilon}$ to $U$ as well as the fact that $\varphi$ is a proper convex l.s.c. function, we get
\begin{equation}\label{equation 2}
{\displaystyle\int_{a}^{b}}\langle U_{s},v_{s}-Y_{s}\rangle ds+{\displaystyle\int_{a}^{b}}\varphi(Y_{s})ds\leq{\displaystyle\int_{a}^{b}}\varphi(v_{s})ds.
\end{equation}
Indeed,
${\displaystyle\int_{a}^{b}}\langle U_{s}^{\varepsilon},v_{s}-Y_{s}^{\varepsilon}\rangle ds\xrightarrow[\;\;]{\mathbb{P}}{\displaystyle\int_{a}^{b}}\langle U_{s},v_{s}-Y_{s}\rangle ds$
and
${\displaystyle\int_{a}^{b}}\langle U_{s}^{\varepsilon},J_{\varepsilon,\varphi}(Y_{s}^{\varepsilon})-Y_{s}^{\varepsilon}\rangle ds\xrightarrow[\;\;]{\mathbb{P}}0$.
Consi\-dering that $a,b,v$ are arbitrary, we conclude from (\ref{equation 2}) that $(Y_{s},U_{s})\in\partial\varphi$, $d\mathbb{P}\otimes dt ~a.e.$

\hfill
\end{proof}

\section{Existence of viscosity solutions of PVIs}
In this section we will prove that the solution of our FBSVI provides a probabilistic interpretation for the solution of PVI (\ref{PVI}).
For this we assume that
\begin{itemize}
\item[$(H_{7})$] The coefficients $b,\sigma,f,g$ are all deterministic and jointly continuous and $m=1$.
\end{itemize}

We collect the following assumptions which we denote by $(A1)$:
\begin{itemize}
\item [$(A1)$] The conditions $(H_{1})-(H_{7}),(H^{\prime}_{5})$ are satisfied and $(H^{\prime}_{4})$ holds with $\rho_{0}\ge1$. Moreover, $(C1)$ and either $(C2)$ or $(C3)$ hold true for some $\lambda,\alpha,C_{1},C_{2},C_{3}$ and $C_{4}=\frac{1-\alpha}{K}$.
\end{itemize}
For each $(t,x)\in[0,T]\times Dom\psi$, we consider the following FBSVI over the interval $[t,T]$:
\begin{equation}\label{FBSVI related to PDE}
\left\{
\begin{array}
[c]{l}%
dX_{s}+\partial\psi(X_{s})ds\ni b(s,X_{s},Y_{s},Z_{s})ds+\sigma(s,X_{s},Y_{s})dB_{s},\medskip\\
-dY_{s}+\partial\varphi(Y_{s})ds\ni f(s,X_{s},Y_{s},Z_{s})ds-Z_{s}dB_{s},\quad s\in[t,T],\medskip\\
X_{t}=x, \quad Y_{T}=g(X_{T}),
\end{array}
\right.
\end{equation}

It is clearly that Theorem \ref{Existence and uniqueness of two operators} remains true on the interval $[t,T]$, and we denote the unique solution of equation (\ref{FBSVI related to PDE}) by $(X_{s}^{t,x},Y_{s}^{t,x},Z_{s}^{t,x},V_{s}^{t,x},U_{s}^{t,x})$. We define
\begin{equation}\label{definition of u}
u(t,x):=Y_{t}^{t,x}, \text{ for } t\in[0,T],x\in Dom\psi.
\end{equation}
Observe that under $(H_{7})$, $Y_{t}^{t,x}$ is deterministic. Indeed, by using the standard "time shifting" technique, we can check that $Y_{s}^{t,x}$ is $\mathcal{F}_{s}^{t}$ adapted, where $\mathcal{F}_{s}^{t}:=\sigma\{B_{r}:t\leq r\leq s\}$ augmented by the $\mathbb{P}$-null sets. The function $u(t,x)=Y_{t}^{t,x}$ has the following properties:
\begin{proposition}\label{Property of u}
Suppose $(A1)$ holds, then $u(t,x)\in Dom\varphi$, for all $(t,x)\in[0,T]\times Dom\psi$, and $u\in C([0,T]\times Dom\psi)$.
\end{proposition}
\begin{proof}
From Proposition \ref{Boundedness proposition of two operators} $(iv)$ and the convergence in (\ref{equation 22}), we deduce that $\varphi(u(t,x))=E\varphi(Y_{t}^{t,x})<\infty$ for all $x\in Dom\psi$. Consequently, $u(t,x)\in Dom\varphi$, $(t,x)\in[0,T]\times Dom\psi$. Let us prove that $u\in C([0,T]\times Dom\psi)$. We split this proof into two steps.

\textbf{Step 1:} We first prove that $u$ is right continuous w.r.t. $t$ and continuous w.r.t. $x$. For this we assume that $(t_{n},x_{n})\rightarrow(t,x)$ as $n\rightarrow\infty$, for $t_{n}\ge t$. We put $\hat{X}^{n}_{s}=X_{s}^{t_{n},x_{n}}-X_{s}^{t,x},\hat{Y}^{n}_{s}=Y_{s}^{t_{n},x_{n}}-Y_{s}^{t,x},\hat{Z}^{n}_{s}=Z_{s}^{t_{n},x_{n}}-Z_{s}^{t,x}$, $s\in[t_{n},T]$.
From (\ref{Estimate 13})-(\ref{Estimate 16}) of Proposition \ref{Difference 2} (Now for $[t_{n},T]$), it follows
\begin{equation}\label{Estimate 17}
\begin{array}
[c]{l}
e^{-\lambda T}E|\hat{X}^{n}_{T}|^{2}+\bar{\lambda}_{1}\|\hat{X}^{n}\|^{2}_{M_{\lambda}[t_{n},T]}\leq K(C_{1}+K)\|\hat{Y}^{n}\|^{2}_{M_{\lambda}[t_{n},T]}\medskip\\
\qquad\qquad+(KC_{2}+k_{1}^{2})\|\hat{Z}^{n}\|^{2}_{M_{\lambda}[t_{n},T]}+e^{-\lambda t_{n}}E|\hat{X}^{n}_{t_{n}}|^{2},
\end{array}
\end{equation}
\begin{equation}\label{Estimate 19}
\|\hat{Y}^{n}\|^{2}_{M_{\lambda}[t_{n},T]}\leq B(\bar{\lambda}_{2},T)\left[
k_{2}^{2}e^{-\lambda T}E|\hat{X}^{n}_{T}|^{2}+KC_{3}\|\hat{X}^{n}\|^{2}_{M_{\lambda}[t_{n},T]}\right],
\end{equation}
\begin{equation}\label{Estimate 20}
\|\hat{Z}^{n}\|^{2}_{M_{\lambda}[t_{n},T]}\leq \frac{A(\bar{\lambda}_{2},T)}{\alpha}\left[
k_{2}^{2}e^{-\lambda T}E|\hat{X}^{n}_{T}|^{2}+KC_{3}\|\hat{X}^{n}\|^{2}_{M_{\lambda}[t_{n},T]}\right].
\end{equation}
Indeed, we observe that here, unlike in (\ref{Estimate 13})-(\ref{Estimate 16}), the coefficients $b,\tilde{b}$, $\sigma,\tilde{\sigma}$ and $f,\tilde{f}$ as well as $g,\tilde{g}$ coincide, so we can consider $\delta\rightarrow 0$ in (\ref{Estimate 13})-(\ref{Estimate 16}).
In the following $C$ denotes a constant independent of $(t,x)$ and $(t_{n},x_{n})$, which may vary from line to line.

Using the compatibility conditions $(C1)$,$(C2)$ or $(C1)$, $(C3)$, we can check with the help of (\ref{Estimate 17})-(\ref{Estimate 20}), that
\begin{equation}
|\hat{X}^{n}|^{2}_{\lambda,\beta,[t_{n},T]}\leq Ce^{-\lambda t_{n}}E|\hat{X}^{n}_{t_{n}}|^{2}\leq C(1+e^{-\lambda T})E|x_{n}-X_{t_{n}}^{t,x}|^{2}.
\end{equation}
Then plugging this inequality into (\ref{Estimate 19}) and (\ref{Estimate 20}), we obtain
\[\|\hat{X}^{n}\|^{2}_{M_{\lambda}[t_{n},T]}+\|\hat{Y}^{n}\|^{2}_{M_{\lambda}[t_{n},T]}+\|\hat{Z}^{n}\|^{2}_{M_{\lambda}[t_{n},T]}\leq CE|x_{n}-X_{t_{n}}^{t,x}|^{2}.
\]
By applying  BDG inequality to the equations for $\hat{X}^{n}$, $\hat{Y}^{n}$, we can prove
\begin{equation}\label{equation 5}
\|\hat{X}^{n}\|^{2}_{S_{\lambda}[t_{n},T]}+\|\hat{Y}^{n}\|^{2}_{S_{\lambda}[t_{n},T]}+\|\hat{Z}^{n}\|^{2}_{M_{\lambda}[t_{n},T]}
\leq CE|x_{n}-X_{t_{n}}^{t,x}|^{2}.
\end{equation}
On the other hand, using Proposition \ref{Priori estimate for X and Y}, we have
\[
\begin{array}
[c]{l}
\qquad\|X^{t,x}_{s}\|^{2}_{S_{\lambda}[t,T]}+\|Y^{t,x}_{s}\|^{2}_{S_{\lambda}[t,T]}+\|Z^{t,x}_{s}\|^{2}_{M_{\lambda}[t,T]}\medskip\\
\leq C\left(e^{-\lambda t}|x|^{2}+\|b^{0}\|^{2}_{M_{\lambda}[t,T]}+\|\sigma^{0}\|^{2}_{M_{\lambda}[t,T]}+\|f^{0}\|^{2}_{M_{\lambda}[t,T]}+E|g^{0}|^{2}\right)\medskip\\
\leq C\left(e^{-(\lambda\wedge0) T}|x|^{2}+\|b^{0}\|^{2}_{M_{\lambda}}+\|\sigma^{0}\|^{2}_{M_{\lambda}}+\|f^{0}\|^{2}_{M_{\lambda}}+E|g^{0}|^{2}\right)\medskip\\
\leq C(1+|x|^{2}).
\end{array}
\]
Then we deduce from
\[X_{s}^{t,x}- x_{n}=x-x_{n}+{\displaystyle\int_{t}^{s}}b(r,X_{r}^{t,x},Y_{r}^{t,x},Z_{r}^{t,x})dr
+{\displaystyle\int_{t}^{s}}\sigma(r,X_{r}^{t,x},Y_{r}^{t,x})dB_{r}-dV_{s}^{t,x},
\]
that
\begin{equation}\label{equation 6}
E\Big\{\sup\limits_{t\leq s\leq t_{n}}|X_{s}^{t,x}- x_{n}|^{2}\Big\}\leq C\Big\{|x-x_{n}|^{2}+(1+|x|^{2})|t_{n}-t|+E\updownarrow V^{t,x}\updownarrow_{[t,t_{n}]}^{2}\Big\}.
\end{equation}
Consequently, from (\ref{equation 5}) and (\ref{equation 6}), we have
\[E\sup\limits_{t_{n}\leq s\leq T}|Y_{s}^{t_{n},x_{n}}-Y_{s}^{t,x}|^{2}\leq C\Big\{|x-x_{n}|^{2}+(1+|x|^{2})|t_{n}-t|+E\updownarrow V^{t,x}\updownarrow_{[t,t_{n}]}^{2}\Big\}.
\]
From (\ref{limit equation for V}), we see that $V^{t,x}$ is a process with continuous paths. But for any continuous bounded variation function $g$, we have $\updownarrow g\updownarrow_{[0,t]}$ is continuous in $t$. Thus, $\updownarrow V^{t,x}\updownarrow_{[t,t_{n}]}\rightarrow 0$,  as $n\rightarrow\infty$, $\mathbb{P}$-a.s., and Dominated Convergence Theorem (recall (\ref{square integrable of bounded variation}), i.e., $\updownarrow V^{t,x}\updownarrow_{[t,T]}\in  L^{2}(\Omega)$), yields that $E\updownarrow V^{t,x}\updownarrow_{[t,t_{n}]}^{2}\rightarrow 0, \text{ as } n\rightarrow\infty.$ Thus, $E\sup\limits_{t_{n}\leq s\leq T}|Y_{s}^{t_{n},x_{n}}-Y_{s}^{t,x}|^{2}\rightarrow0$, as $n\rightarrow\infty$ and thanks to the $L^{2}$-continuity of $Y^{t,x}$ we have $E|Y_{t_{n}}^{t,x}-Y_{t}^{t,x}|^{2}\rightarrow0$ as $n\rightarrow\infty$. Consequently,
\[
\begin{array}
[c]{l}
|u(t_{n},x_{n})-u(t,x)|^{2}=E|Y_{t_{n}}^{t_{n},x_{n}}-Y_{t}^{t,x}|^{2}
\leq 2E\sup\limits_{t_{n}\leq s\leq T}|Y_{s}^{t_{n},x_{n}}-Y_{s}^{t,x}|^{2}+2E|Y_{t_{n}}^{t,x}-Y_{t}^{t,x}|^{2}\medskip\\
\qquad\qquad\qquad\qquad\qquad\rightarrow0, \text{ as }n\rightarrow\infty.
\end{array}
\]
\textbf{Step 2:} Let us now show that $u$ is left continuous w.r.t. $t$ and continuous w.r.t. $x$.
For this we consider $(t_{n},x_{n})\rightarrow(t,x)$ as $n\rightarrow\infty$, for $t_{n}\leq t$. We extend $X^{t,x}_{s},Y^{t,x}_{s},Z^{t,x}_{s},V^{t,x}_{s},U^{t,x}_{s}$ to $s\in[t_{n},T]$ by choosing $X^{t,x}_{s}=x$, $Y^{t,x}_{s}=Y^{t,x}_{t}$,
$Z^{t,x}_{s}=0$, $dV^{t,x}_{s}=x^{\ast}ds$, $U^{t,x}_{s}=u^{\ast}$, $s\in[t_{n},t]$, for arbitrarily chosen $x^{\ast}\in \partial\psi(x)$ and $u^{\ast}\in \partial\varphi(Y^{t,x}_{t})$. Then we know that
\[
\langle z-X_{r}^{t,x},x^{\ast}\rangle+\psi(X_{r}^{t,x})\leq \psi(z),~ \text{ for all } z\in\mathbb{R}^{n}, ~ t_{n}\leq r\leq t, ~a.s.
\]
which yields that
\[
{\displaystyle\int_{a}^{b}}\langle z-X_{r}^{t,x},dV_{r}^{t,x}\rangle+{\displaystyle\int_{a}^{b}}\psi(X_{r}^{t,x})dr\leq (b-a)\psi(z),~ \text{ for all } z\in\mathbb{R}^{n}, ~ ~t_{n}\leq a\leq b\leq t, ~a.s.
\]
Moreover, it holds that
${\displaystyle\int_{a}^{b}}\langle z-X_{r}^{t,x},dV_{r}^{t,x}\rangle+{\displaystyle\int_{a}^{b}}\psi(X_{r}^{t,x})dr\leq (b-a)\psi(z)$, for all $z\in\mathbb{R}^{n}$, $t_{n}\leq a\leq b\leq T, ~a.s.$ and $(Y_{r}^{t,x},U_{r}^{t,x})\in\partial\varphi,~d\mathbb{P}\otimes dt~a.e.$ on $\Omega\times[t_{n},T]$.

Using the above extension, we have
\[
\left\{
\begin{array}
[c]{l}
X_{s}^{t,x}+V_{s}^{t,x}=x+{\displaystyle\int_{t_{n}}^{s}}\tilde{b}(r,X_{r}^{t,x},Y_{r}^{t,x},Z_{r}^{t,x})ds
+{\displaystyle\int_{t_{n}}^{s}}\tilde{\sigma}(r,X_{r}^{t,x},Y_{r}^{t,x})dB_{r},\medskip\\
Y_{s}^{t,x}+{\displaystyle\int_{s}^{T}}U_{r}^{t,x}dr=g(X_{T}^{t,x})+{\displaystyle\int
_{s}^{T}}\tilde{f}( r,X_{r}^{t,x},Y_{r}^{t,x},Z_{r}^{t,x})dr-{\displaystyle\int_{s}^{T}}Z_{r}^{t,x}dB_{r},~ s\in[t_{n},T],\;a.s.
\end{array}
\right.
\]
where we define $\tilde{b}(r,x,y,z)=x^{\ast}1_{[t_{n},t]}(r)+b(r,x,y,z)1_{[t,T]}(r)$, $\tilde{\sigma}(r,x,y)=\sigma(r,x,y)1_{[t,T]}(r)$ and
$\tilde{f}(r,x,y,z)=u^{\ast}1_{[t_{n},t]}(r)+f(r,x,y,z)1_{[t,T]}(r)$.

From  Proposition \ref{Difference 2}, there exists a constant $C$ which does not depend on $(t_{n}, x_{n})$ such that
\begin{equation}\label{estimate for Y with different intial time and value}
\begin{array}
[c]{l}
\quad  E\Big\{\sup\limits_{t_{n}\leq s\leq T}|X_{s}^{t,x}-X_{s}^{t_{n},x_{n}}|^{2}+\sup\limits_{t_{n}\leq s\leq T}|Y_{s}^{t,x}-Y_{s}^{t_{n},x_{n}}|^{2}+{\displaystyle\int_{t_{n}}^{T}}|Z_{s}^{t,x}-Z_{s}^{t_{n},x_{n}}|^{2}ds\Big\}\leq C\Delta_{3},
\end{array}
\end{equation}
where
\begin{equation}\label{Delta 2}
\begin{array}
[c]{l}%
\Delta_{3}=e^{-\lambda t_{n}}|x-x_{n}|^{2}+E{\displaystyle\int_{t_{n}}^{T}}|\tilde{b}-b|^{2}(s,X_{s}^{t,x},Y_{s}^{t,x},Z_{s}^{t,x})ds
\medskip\\
\qquad+E{\displaystyle\int_{t_{n}}^{T}}|\tilde{f}-f|^{2}(s,X_{s}^{t,x},Y_{s}^{t,x},Z_{s}^{t,x})ds
+E{\displaystyle\int_{t_{n}}^{T}}|\tilde{\sigma}-\sigma|^{2}(s,X_{s}^{t,x},Y_{s}^{t,x})ds\medskip\\
=e^{-\lambda t_{n}}|x-x_{n}|^{2}+E{\displaystyle\int_{t_{n}}^{t}}|x^{\ast}-b(s,X_{t}^{t,x},Y_{t}^{t,x},0)|^{2}ds
\medskip\\
\qquad+E{\displaystyle\int_{t_{n}}^{t}}|u^{\ast}-f(s,X_{t}^{t,x},Y_{t}^{t,x},0)|^{2}ds
+E{\displaystyle\int_{t_{n}}^{t}}|\sigma(s,X_{t}^{t,x},Y_{t}^{t,x})|^{2}ds.
\end{array}
\end{equation}
From Remark \ref{L2 estimates of the solution of fbsvi} we also have
\begin{equation}\label{boundness}
\begin{array}
[c]{l}%
\qquad E\sup\limits_{t\leq r\leq T}|X_{r}^{t,x}|^{2}+E\sup\limits_{t\leq r\leq T}|Y_{r}^{t,x}|^{2}+E{\displaystyle\int_{t}^{T}}|Z_{r}^{t,x}|^{2}dr\medskip\\
\leq CE\left\{|x|^{2}+|g^{0}|^{2}+{\displaystyle\int_{t}^{T}}|b^{0}(r)|^{2}dr
+{\displaystyle\int_{t}^{T}}|f^{0}(r)|^{2}dr+
{\displaystyle\int_{t}^{T}}|\sigma^{0}(r)|^{2}dr\right\}.
\end{array}
\end{equation}
Hence, by combining (\ref{estimate for Y with different intial time and value})-(\ref{boundness}) and considering that $x^{\ast}$, $u^{\ast}$ only depends on $t$ and $x$, we obtain that
$E\sup\limits_{t_{n}\leq s\leq T}|Y_{s}^{t,x}-Y_{s}^{t_{n},x_{n}}|^{2}\rightarrow 0, \text{ as } n\rightarrow\infty$.
Consequently,
\[
\begin{array}
[c]{ll}
|u(t_{n},x_{n})-u(t,x)|^{2}=E|Y_{t_{n}}^{t_{n},x_{n}}-Y_{t}^{t,x}|^{2}\leq E\sup\limits_{t_{n}\leq s\leq T}|Y_{s}^{t_{n},x_{n}}-Y_{s}^{t,x}|^{2}\rightarrow 0, \text{ as } n\rightarrow\infty.
\end{array}
\] \hfill
\end{proof}
\begin{remark}\label{remark for extended u}
From (\ref{equation 5}), setting $t_{n}=t$, we know that $u(t,x)$ is Lipschitz continuous w.r.t. $x$ on $[0,T]\times Dom\psi$.
\end{remark}
\begin{theorem}\label{viscosity solution existence}
Suppose $(A1)$ holds and $Dom\psi$ is locally compact, then
the function $u(t,x)=Y^{t,x}_{t}$, $(t,x)\in[0,T]\times Dom\psi$ is a viscosity solution of PVI (\ref{PVI}).
\end{theorem}
As a consequence of Theorem \ref{uniqueness of pde} and Theorem \ref{viscosity solution existence}, we have
\begin{theorem}\label{uniqueness of pvi}
We assume that $Dom\psi$ is locally compact, $\sigma(t,x,y)$ does not depend on $y$ and $f$ is Lipschitz continuous w.r.t. $y$. Then under the assumptions of $(A1)$, PVI (\ref{PVI}) has a unique viscosity solution in the class of functions which are Lipschitz continuous in $x$ uniformly w.r.t. $t$ and continuous in $t$.
\end{theorem}
\begin{proof}[Proof of Theorem \ref{viscosity solution existence}]
We use the following penalized equation to approximate (\ref{FBSVI related to PDE}):
\begin{equation}\label{Penalized FBSDE from t}
\left\{
\begin{array}
[c]{l}%
X_{s}^{\varepsilon}+{\displaystyle\int_{t}^{s}}\nabla\psi_{\varepsilon}(X_{r}^{\varepsilon})dr=x+{\displaystyle\int_{t}^{s}}b(r,X_{r}^{\varepsilon},Y_{r}^{\varepsilon},Z_{r}^{\varepsilon})dr
+{\displaystyle\int_{t}^{s}}\sigma(r,X_{r}^{\varepsilon},Y_{r}^{\varepsilon})dB_{r},\medskip\\
Y_{s}^{\varepsilon}+{\displaystyle\int_{s}^{T}}\nabla\varphi_{\varepsilon}(Y_{r}^{\varepsilon})dr=g(X_{T}^{\varepsilon})+{\displaystyle\int
_{s}^{T}}f(r,X_{r}^{\varepsilon},Y_{r}^{\varepsilon},Z_{r}^{\varepsilon})dr-{\displaystyle\int_{s}^{T}}Z_{r}^{\varepsilon}dB_{r},
\end{array}
\right.
\end{equation}
$s\in[t,T]$. From Theorem \ref{Solvability of Penalized FBSDE in appendix}, we know that for any $x\in\mathbb{R}^{n}$, the above FBSDE has a unique solution $\left\{(X_{s}^{\varepsilon; t,x},Y_{s}^{\varepsilon; t,x},Z_{s}^{\varepsilon; t,x}),~s\in[t,T]\right\}$. Putting
\[u^{\varepsilon}(t,x):=Y_{t}^{\varepsilon; t,x},\qquad 0\leq t\leq T,~x\in\mathbb{R}^{n},
\]
We can argue similarly to Proposition \ref{Property of u} in order to show that $u^{\varepsilon}$ is continuous on $[0,T]\times \mathbb{R}^{n}$. Moreover, with the help of Theorem 5.1 \cite{PT-99} we can prove that $u^{\varepsilon}(t,x)$
is the viscosity solution of the following backward quasilinear second-order parabolic PDE:
\begin{equation}\label{penalized PVI}
\left\{
\begin{array}
[c]{l}%
{\displaystyle\frac{\partial u^{\varepsilon}}{\partial s}(s,x)}+(\mathcal{L}u^{\varepsilon})(s,x,u^{\varepsilon}(s,x),(\nabla u^{\varepsilon}(s,x))^{\ast}\sigma(s,x,u^{\varepsilon}(s,x)))\medskip\\
\qquad+f(s,x,u^{\varepsilon}(s,x),(\nabla u^{\varepsilon}(s,x))^{\ast}\sigma(s,x,u^{\varepsilon}(s,x)))=\nabla\varphi_{\varepsilon}(u^{\varepsilon}(s,x))+\langle \nabla\psi_{\varepsilon}(x),\nabla u^{\varepsilon}(s,x)\rangle, \medskip\\
\qquad\qquad\qquad\qquad\qquad\qquad\qquad\qquad\qquad\qquad\qquad\qquad (s,x) \in[0,T]\times \mathbb{R}^{n},\medskip\\
u^{\varepsilon}(T,x)=g(x),\quad x\in\mathbb{R}^{n}.
\end{array}
\right.
\end{equation}
Similarly to the proof of Proposition \ref{Lipschitz proposition of two operators} $(ii)$, taking $\varepsilon_{2}\rightarrow0$, we can prove that
\begin{equation}\label{equation 2 w.r.t. u}
\begin{array}
[c]{l}
|u^{\varepsilon}(t,x)-u(t,x)|^{2}\leq E\sup\limits_{t\leq s\leq T}|Y_{s}^{\varepsilon; t,x}-Y_{s}^{t,x}|^{2}\medskip\\
\leq C(1+|x|^{3+\rho_{0}})\left[1+|x|^{3+\rho_{0}}+\psi^{1+2\rho}(x)\right]\varepsilon^{\frac{\rho}{2+4\rho}},
\text{ for all } (t,x)\in[0,T]\times Dom\psi.
\end{array}
\end{equation}
Here $C$ is a constant which does not depend on $(t,x)$ and $\varepsilon$.

Now we will show that $u$ is a subsolution of PVI (\ref{PVI}). From Lemma 6.1 \cite{CIL-92} we know that for any point $(p,q,X)\in\mathcal{P}^{2,+}u(t,x)$, there exist sequences:
\[
0<\varepsilon_{n}\rightarrow0, \quad (t_{n},x_{n})\rightarrow(t,x),\quad (p_{n},q_{n},X_{n})\in\mathcal{P}^{2,+}u^{\varepsilon_{n}}(t_{n},x_{n}),
\]
such that
\[(p_{n},q_{n},X_{n})\rightarrow(p,q,X).
\]
For any $n$, using that $(p_{n},q_{n},X_{n})\in\mathcal{P}^{2,+}u^{\varepsilon_{n}}(t_{n},x_{n})$, we obtain
\begin{equation}\label{equation 3}
\begin{array}
[c]{l}
-p_{n}-\frac{1}{2}Tr(\sigma\sigma^{\ast}(t_{n},x_{n},u^{\varepsilon_{n}}(t_{n},x_{n}))X_{n})
-\langle b(t_{n},x_{n},u^{\varepsilon_{n}}(t_{n},x_{n}),q^{\ast}_{n}\sigma(t_{n},x_{n},u^{\varepsilon_{n}}(t_{n},x_{n}))),
q_{n}\rangle\medskip\\
-f(t_{n},x_{n},u^{\varepsilon_{n}}(t_{n},x_{n}),q^{\ast}_{n}\sigma(t_{n},x_{n},u^{\varepsilon_{n}}(t_{n},x_{n})))\leq -\nabla\varphi_{\varepsilon_{n}}(u^{\varepsilon_{n}}(t_{n},x_{n}))-\langle \nabla\psi_{\varepsilon_{n}}(x_{n}),q_{n}\rangle.
\end{array}
\end{equation}
We can assume that $u(t,x)> \inf(Dom\varphi)$. Indeed, if we have $u(t,x)=\inf(Dom\varphi)$, then $\varphi^{\prime}_{-}(u(t,x))=-\infty$, and inequality (\ref{subsolution}) would hold obviously. Let $y\in Dom\varphi, y<u(t,x)$. Since $u^{\varepsilon}$ converges to $u$ uniformly on compacts (see (\ref{equation 2 w.r.t. u})), there exists $n_{0}\in\mathbb{N}$ s.t. $y<u^{\varepsilon_{n}}(t_{n},x_{n})$, $n>n_{0}$. Multiplying $u^{\varepsilon_{n}}(t_{n},x_{n})-y$ with both sides of (\ref{equation 3}) and using
\[
\begin{array}
[c]{rl}
\varphi(J_{\varepsilon,\varphi}x)  \leq \varphi_{\varepsilon}(x)
 \leq \varphi_{\varepsilon}(z)+(x-z,\nabla\varphi_{\varepsilon}(x))
\leq \varphi(z)+(x-z,\nabla\varphi_{\varepsilon}(x)), ~ z\in\mathbb{R},
\end{array}
\]
we obtain
\begin{equation}\label{equation 4}
\begin{array}
[c]{l}
\Big[-p_{n}-\frac{1}{2}Tr(\sigma\sigma^{\ast}(t_{n},x_{n},u^{\varepsilon_{n}}(t_{n},x_{n}))X_{n})
-\langle b(t_{n},x_{n},u^{\varepsilon_{n}}(t_{n},x_{n}),q^{\ast}_{n}\sigma(t_{n},x_{n},u^{\varepsilon_{n}}(t_{n},x_{n}))),
q_{n}\rangle\medskip\\
\qquad\qquad\qquad-f(t_{n},x_{n},u^{\varepsilon_{n}}(t_{n},x_{n}),q^{\ast}_{n}\sigma(t_{n},x_{n},u^{\varepsilon_{n}}(t_{n},x_{n})))
\Big]\Big[u^{\varepsilon_{n}}(t_{n},x_{n})-y\Big]\medskip\\
\qquad\qquad\qquad +\langle \nabla\psi_{\varepsilon_{n}}(x_{n}),q_{n}\rangle\Big[u^{\varepsilon_{n}}(t_{n},x_{n})-y\Big]
+\varphi(J_{\varepsilon_{n},\varphi}u^{\varepsilon_{n}}(t_{n},x_{n}))\leq \varphi(y).
\end{array}
\end{equation}
Let us take now $\liminf_{n\rightarrow\infty}$ in the above inequality. Recalling (\ref{Property of Yosida approximation}-c), $\nabla\psi_{\varepsilon_{n}}(x_{n})\in\partial\psi(J_{\varepsilon_{n},\psi}(x_{n}))$, $J_{\varepsilon_{n},\psi}(x_{n})\rightarrow x$ and $J_{\varepsilon_{n},\varphi}(u^{\varepsilon_{n}}(t_{n},x_{n}))\rightarrow u(t,x)$, the lower limit in (\ref{equation 4}) yields
\[
\begin{array}
[c]{ll}
\Big[-p-\frac{1}{2}Tr(\sigma\sigma^{\ast}(t,x,u(t,x))X)-\langle b(t,x,u(t,x),q^{\ast}\sigma(t,x,u(t,x))),q\rangle\medskip\\
\qquad\qquad\qquad-f(t,x,u(t,x),q^{\ast}\sigma(t,x,u(t,x)))\Big]\Big[u(t,x)-y\Big]\medskip\\
\qquad\qquad\qquad\qquad+\partial\psi_{\ast}(x,q)\Big[u(t,x)-y\Big]+\varphi(u(t,x))\leq \varphi(y).
\end{array}
\]
Then
\[
\begin{array}
[c]{l}
-p-\frac{1}{2}Tr(\sigma\sigma^{\ast}(t,x,u(t,x))X)-\langle b(t,x,u(t,x),q^{\ast}\sigma(t,x,u(t,x))),q\rangle\medskip\\
\qquad-f(t,x,u(t,x),q^{\ast}\sigma(t,x,u(t,x)))\leq -\frac{\varphi(u(t,x))-\varphi(y)}{u(t,x)-y}-\partial\psi_{\ast}(x,q),\text{ for all } y<u(t,x),
\end{array}
\]
and taking the limit $y\rightarrow u(t,x)$ yields (\ref{subsolution}). Therefore $u$ is a viscosity subsolution of PVI (\ref{PVI}).
Similarly, we prove that $u$ is a viscosity supersolution of PVI (\ref{PVI}).\hfill
\end{proof}

\section{Appendix}
\subsection{A priori estimates for penalized FBSDEs}
In this subsection, following the ideas of \cite{MC-01}, we give a priori estimates for the solution of the following penalized FBSDE:
\begin{equation}\label{Penalized FBSDE with two operators in appendix}
\left\{
\begin{array}
[c]{l}%
X_{s}^{\varepsilon}+{\displaystyle\int_{t}^{s}}\nabla\psi_{\varepsilon}(X_{r}^{\varepsilon})dr
=\xi+{\displaystyle\int_{t}^{s}}b(r,X_{r}^{\varepsilon},Y_{r}^{\varepsilon},Z_{r}^{\varepsilon})dr
+{\displaystyle\int_{t}^{s}}\sigma(r,X_{r}^{\varepsilon},Y_{r}^{\varepsilon},Z_{r}^{\varepsilon})dB_{r},\medskip\\
Y_{s}^{\varepsilon}+{\displaystyle\int_{s}^{T}}\nabla\varphi_{\varepsilon}(Y_{r}^{\varepsilon})dr=g(X_{T}^{\varepsilon})+{\displaystyle\int
_{s}^{T}}f(r,X_{r}^{\varepsilon},Y_{r}^{\varepsilon},Z_{r}^{\varepsilon})dr-{\displaystyle\int_{s}^{T}}Z_{r}^{\varepsilon}dB_{r}, ~s\in[t,T],
\end{array}
\right.
\end{equation}
where $t\in[0,T)$, $\xi\in L^{2}(\Omega,\mathcal{F}_{t},\mathbb{P};\mathbb{R}^{n})$ and the coefficients are su\-pposed to satisfy $(H_{1})-(H_{5})$.

\begin{lemma}\label{Difference of two solutions}
Assume $(H_{1})$-$(H_{5})$ hold. Suppose that $(X^{\varepsilon,t,\xi_{i}},Y^{\varepsilon,t,\xi_{i}},Z^{\varepsilon,t,\xi_{i}})$ is the solution of FBSDE (\ref{Penalized FBSDE with two operators in appendix}) with $\xi=\xi_{i}\in L^{2}(\Omega,\mathcal{F}_{t},\mathbb{P};\mathbb{R}^{n})$ for both $i=1,2$.
Let us put $\hat{\xi}=\xi_{1}-\xi_{2}$, $\hat{X}^{\varepsilon}=X^{\varepsilon,t,\xi_{1}}-X^{\varepsilon,t,\xi_{2}}$, $\hat{Y}^{\varepsilon}=Y^{\varepsilon,t,\xi_{1}}-Y^{\varepsilon,t,\xi_{2}}$
and $\hat{Z}^{\varepsilon}=Z^{\varepsilon,t,\xi_{1}}-Z^{\varepsilon,t,\xi_{2}}$. Then, for arbitrary $\lambda\in\mathbb{R}$ and arbitrary positive constants $C_{1},C_{2},C_{3},C_{4}$, we have
\begin{equation}\label{Estimate 1}
\begin{array}
[c]{ll}%
e^{-\lambda T}E|\hat{X}_{T}^{\varepsilon}|^{2}+\bar{\lambda}_{1}\|\hat{X}^{\varepsilon}\|^{2}_{M_{\lambda}[t,T]}\leq
e^{-\lambda t}E|\hat{\xi}|^{2}+ K(C_{1}+K)\|\hat{Y}^{\varepsilon}\|^{2}_{M_{\lambda}[t,T]}\medskip\\
\qquad\qquad\qquad\qquad\qquad\qquad\qquad\qquad+(KC_{2}+k_{1}^{2})\|\hat{Z}^{\varepsilon}\|^{2}_{M_{\lambda}[t,T]}.
\end{array}
\end{equation}
Furthermore, if in addition $KC_{4}=1-\alpha$ for some $0<\alpha<1$, then
\begin{equation}\label{Estimate 3}
\|\hat{Y}^{\varepsilon}\|^{2}_{M_{\lambda}[t,T]}\leq B(\bar{\lambda}_{2},T-t)\left[
k_{2}^{2}e^{-\lambda T}E|\hat{X}_{T}^{\varepsilon}|^{2}+KC_{3}\|\hat{X}^{\varepsilon}\|^{2}_{M_{\lambda}[t,T]}\right],
\end{equation}
\begin{equation}\label{Estimate 4}
\|\hat{Z}^{\varepsilon}\|^{2}_{M_{\lambda}[t,T]}\leq \frac{A(\bar{\lambda}_{2},T-t)}{\alpha}\left[
k_{2}^{2}e^{-\lambda T}E|\hat{X}_{T}^{\varepsilon}|^{2}+KC_{3}\|\hat{X}^{\varepsilon}\|^{2}_{M_{\lambda}[t,T]}\right].
\end{equation}
Here $\bar{\lambda}_{1}=\lambda-K(2+C^{-1}_{1}+C^{-1}_{2})-K^{2}$ and $\bar{\lambda}_{2}=-\lambda-2\gamma-K(C^{-1}_{3}+C^{-1}_{4})$.
Recall that $A$ and $B$ are defined in (\ref{notations}) as $A(\lambda,s)=e^{-(\lambda\wedge0)s}$ and $B(\lambda,s)={\displaystyle\int_{0}^{s}}e^{-\lambda r}dr$ .
\end{lemma}
\begin{proof}
We apply It\^{o}'s formula to  $\left(e^{-\lambda r}e^{-\lambda^{\prime} (s-r)}|\hat{X}^{\varepsilon}_{r}|^{2}\right)_{r\in[t,s]}$ and $\left( e^{-\lambda r}e^{-\lambda^{\prime} (r-s)}|\hat{Y}^{\varepsilon}_{r}|^{2}\right)_{r\in[s,T]}$. Thanks to the convexity of $\psi$ and $\varphi$, we have $\langle x_{1}-x_{2},\nabla\psi_{\varepsilon}(x_{1})-\nabla\psi_{\varepsilon}(x_{2})\rangle\ge0$ and $\langle y_{1}-y_{2},\nabla\varphi_{\varepsilon}(y_{1})-\nabla\varphi_{\varepsilon}(y_{2})\rangle\ge0$, which allow us to repeat the argument of Lemma 5.1 \cite{MC-01} to obtain our result.
\hfill
\end{proof}

We recall that we have introduced the notations $b^{0}(s):=b(\cdot,s,0,0,0)$, $\sigma^{0}(s):=\sigma(\cdot,s,0,0,0)$, $f^{0}(s):=f(\cdot,s,0,0,0)$, $g^{0}:=g(\cdot,0)$.
\begin{proposition}\label{Priori estimate for X and Y}
Let the assumptions $(H_{1})$-$(H_{5})$ be satisfied. We also assume $(C1)$ and either $(C2)$ or $(C3)$ hold for some $\lambda,\alpha,C_{1},C_{2},C_{3}$, and $C_{4}=\frac{1-\alpha}{K}$. Then for $t\in[0,T]$ and $\xi\in L^{2}(\Omega,\mathcal{F}_{t},\mathbb{P};\mathbb{R}^{n})$,  there exists a constant $C$ independent of $\varepsilon$ and the initial data $(t,\xi)$, such that
\begin{equation}
\qquad \|X^{\varepsilon,t,\xi}\|^{2}_{S_{\lambda}[t,T]}+\|Y^{\varepsilon,t,\xi}\|^{2}_{S_{\lambda}[t,T]}
+\|Z^{\varepsilon,t,\xi}\|^{2}_{M_{\lambda}[t,T]}
\leq C\Gamma_{1},
\end{equation}
where
\[\Gamma_{1}=e^{-\lambda t}E|\xi|^{2}+\|b^{0}\|^{2}_{M_{\lambda}[t,T]}+\|\sigma^{0}\|^{2}_{M_{\lambda}[t,T]}+\|f^{0}\|^{2}_{M_{\lambda}[t,T]}+E|g^{0}|^{2}.
\]
Moreover, if among the compatibility conditions, only $(C1)$ holds, then there exists a constant $T_{0}>0$ small enough (depending only on $K,\gamma,k_{1},k_{2}$) such that for all $t\in[0,T]$ with $T-t\leq T_{0}$, we have the above estimates.
\end{proposition}
\begin{proof}
Similarly to Lemma 5.2 \cite{MC-01}, using $A(\lambda,T-t)\leq A(\lambda,T)$, $B(\lambda,T-t)\leq B(\lambda,T)$, $\langle x_{1}-x_{2},\nabla\psi_{\varepsilon}(x_{1})-\nabla\psi_{\varepsilon}(x_{2})\rangle\ge0$ and $\langle y_{1}-y_{2},\nabla\varphi_{\varepsilon}(y_{1})-\nabla\varphi_{\varepsilon}(y_{2})\rangle\ge0$, we get,
for arbitrary $\lambda\in\mathbb{R}$, $\delta>0$ and positive constants $C_{1},C_{2},C_{3},C_{4}$
\begin{equation}\label{Estimate 5}
\begin{array}
[c]{l}
e^{-\lambda T}E|X_{T}^{\varepsilon,t,\xi}|^{2}+\bar{\lambda}_{1}^{\delta}\|X^{\varepsilon,t,\xi}\|^{2}_{M_{\lambda}[t,T]}\leq K[C_{1}+K(1+\delta)]\|Y^{\varepsilon,t,\xi}\|^{2}_{M_{\lambda}[t,T]}\medskip\\
\quad+[KC_{2}+k_{1}^{2}(1+\delta)]\|\hat{Z}^{\varepsilon,t,\xi}\|^{2}_{M_{\lambda}[t,T]}
+e^{-\lambda t}E|\xi|^{2}+\frac{1}{\delta}\|b^{0}\|_{M_{\lambda}[t,T]}+(1+\frac{1}{\delta})\|\sigma^{0}\|_{M_{\lambda}[t,T]}.
\end{array}
\end{equation}
Furthermore, if in addition, $KC_{4}=1-\alpha$, for some $0<\alpha<1$, then
\begin{equation}\label{Estimate 7}
\begin{array}
[c]{l}
\|Y^{\varepsilon,t,\xi}\|^{2}_{M_{\lambda}[t,T]}\leq B(\bar{\lambda}_{2}^{\delta},T)\Big\{
k_{2}^{2}(1+\delta)e^{-\lambda T}E|X_{T}^{\varepsilon,t,\xi}|^{2}+KC_{3}\|X^{\varepsilon,t,\xi}\|^{2}_{M_{\lambda}[t,T]}\medskip\\
\qquad\qquad\qquad\qquad\qquad\qquad\qquad
\frac{1}{\delta}\|f^{0}\|_{M_{\lambda}[t,T]}+(1+\frac{1}{\delta})e^{-\lambda T}E|g^{0}|^{2}\Big\},
\end{array}
\end{equation}
\begin{equation}\label{Estimate 8}
\begin{array}
[c]{l}%
\|Z^{\varepsilon,t,\xi}\|^{2}_{M_{\lambda}[t,T]}\leq \frac{A(\bar{\lambda}_{2}^{\delta},T)}{\alpha}\Big\{
k_{2}^{2}(1+\delta)e^{-\lambda T}E|X_{T}^{\varepsilon,t,\xi}|^{2}+KC_{3}\|X^{\varepsilon,t,\xi}\|^{2}_{M_{\lambda}[t,T]}\medskip\\
\qquad\qquad\qquad\qquad\qquad\qquad\qquad
\frac{1}{\delta}\|f^{0}\|_{M_{\lambda}[t,T]}+(1+\frac{1}{\delta})e^{-\lambda T}E|g^{0}|^{2}\Big\},
\end{array}
\end{equation}
where $\bar{\lambda}_{1}^{\delta}=\bar{\lambda}_{1}-(1+K^{2})\delta$ and $\bar{\lambda}_{2}^{\delta}=\bar{\lambda}_{2}-\delta$.
Now we define
\[\mu^{\delta}(\alpha,T):=K[C_{1}+K(1+\delta)]B(\bar{\lambda}_{2}^{\delta},T)
+\frac{A(\bar{\lambda}_{2}^{\delta},T)}{\alpha}[KC_{2}+k_{1}^{2}(1+\delta)].
\]
Then plugging (\ref{Estimate 7}) and (\ref{Estimate 8}) into (\ref{Estimate 5}) yields
\begin{equation}\label{auxiliary inequality}
\begin{array}
[c]{l}%
\qquad(1-\mu^{\delta}(\alpha,T)k_{2}^{2}(1+\delta))e^{-\lambda T}E|X_{T}^{\varepsilon,t,\xi}|^{2}+(\bar{\lambda}_{1}^{\delta}
-\mu^{\delta}(\alpha,T)KC_{3})\|X^{\varepsilon,t,\xi}\|^{2}_{M_{\lambda}[t,T]}\medskip\\
\leq e^{-\lambda t}E|\xi|^{2}
+\frac{1}{\delta}\|b^{0}\|_{M_{\lambda}[t,T]}+(1+\frac{1}{\delta})\|\sigma^{0}\|_{M_{\lambda}[t,T]}\medskip\\
\qquad\qquad\qquad\qquad\qquad\qquad
+\mu^{\delta}(\alpha,T)\Big\{\frac{1}{\delta}\|f^{0}\|_{M_{\lambda}[t,T]}+(1+\frac{1}{\delta})e^{-\lambda T}E|g^{0}|^{2}\Big\}.
\end{array}
\end{equation}
Observe that $\bar{\lambda}_{1}^{\delta}\rightarrow\bar{\lambda}_{1}$ and $\mu^{\delta}(\alpha,T)\rightarrow\mu(\alpha,T)$ as $\delta\rightarrow 0$ (Recall the definition of $\mu(\alpha,T)$ in (\ref{notations})). On the other hand, due to the assumption $(C2)$, there exists $\alpha\in(0,1)$  s.t. $\mu(\alpha,T)KC_{3}<\bar{\lambda}_{1}$. Consequently, we can choose a sufficiently small $\delta>0$ which is independent of $\varepsilon$ and $t$, such that $\bar{\lambda}_{1}^{\delta}-\mu^{\delta}(\alpha,T)KC_{3}>0$. Then we obtain for our fixed $\delta>0$ from (\ref{auxiliary inequality}),
\[\|X^{\varepsilon,t,\xi}\|^{2}_{M_{\lambda}[t,T]}\leq C\Gamma_{1},
\]
where
\[
\begin{array}
[c]{l}
C=
\frac{(1+\frac{1}{\delta})(1+\mu^{\delta}(\alpha,T))}{\bar{\lambda}_{1}^{\delta}-\mu^{\delta}(\alpha,T)KC_{3}},
\end{array}
\]
and
\[\Gamma_{1}=e^{-\lambda t}E|\xi|^{2}+\|b^{0}\|^{2}_{M_{\lambda}[t,T]}+\|\sigma^{0}\|^{2}_{M_{\lambda}[t,T]}+\|f^{0}\|^{2}_{M_{\lambda}[t,T]}+E|g^{0}|^{2}.
\]

Similarly, under the assumptions $(C1)$ and $(C3)$,  we know that there exists $\alpha\in(k^{2}_{1}k^{2}_{2},1)$, such that $\mu(\alpha,T)k_{2}^{2}<1 \text{ and } \bar{\lambda}_{1}\ge\frac{KC_{3}}{k_{2}^{2}}$. Hence, we can choose $\delta>0$ such that $1-\mu^{\delta}(\alpha,T)k_{2}^{2}(1+\delta)>0$ and $\bar{\lambda}_{1}^{\delta}-\mu^{\delta}(\alpha,T)KC_{3}\ge\bar{\lambda}_{1}^{\delta}-\bar{\lambda}_{1}\mu^{\delta}(\alpha,T)k_{2}^{2}>0$.
Then  from
(\ref{auxiliary inequality}) it follows
\[e^{-\lambda T}E|X_{T}^{\varepsilon,t,\xi}|+\|X^{\varepsilon,t,\xi}\|^{2}_{M_{\lambda}[t,T]}\leq C\Gamma_{1},
\]
where
\[
\begin{array}
[c]{l}
C=\max\Bigg\{\frac{(1+\frac{1}{\delta})(1+\mu^{\delta}(\alpha,T))}{1-\mu^{\delta}(\alpha,T)k_{2}^{2}(1+\delta)},
\frac{(1+\frac{1}{\delta})(1+\mu^{\delta}(\alpha,T))}{\bar{\lambda}_{1}^{\delta}-\mu^{\delta}(\alpha,T)KC_{3}}
\Bigg\}.
\end{array}
\]
Using (\ref{Estimate 7}) and (\ref{Estimate 8}), the above estimates for $X^{\varepsilon,t,\xi}$ allows to deduce that
\[\|X^{\varepsilon,t,\xi}\|^{2}_{M_{\lambda}[t,T]}
+\|Y^{\varepsilon,t,\xi}\|^{2}_{M_{\lambda}[t,T]}
+\|Z^{\varepsilon,t,\xi}\|^{2}_{M_{\lambda}[t,T]}\leq C\Gamma_{1}.
\]
Moreover, applying It\^{o}'s formula to $|X_{r}^{\varepsilon,t,\xi}|^{2}$ and $|Y_{r}^{\varepsilon,t,\xi}|^{2}$, and using (\ref{Property of Yosida approximation}-d)), the BDG and the H\"{o}lder inequality,  we obtain
$\|X^{\varepsilon,t,\xi}\|^{2}_{S_{\lambda}[t,T]}+\|Y^{\varepsilon,t,\xi}\|^{2}_{S_{\lambda}[t,T]}\leq C\Gamma_{1}$.

Finally, if only $(C1)$ holds, similar to the argument of \cite{MC-01}, if $k_{2}=0$, for fixed $C_{1},C_{2},C_{3}$ and $\alpha\in(0,1)$, we have $\mu(\alpha,0)KC_{3}=\frac{(KC_{2}+k_{1}^{2})KC_{3}}{\alpha}\leq \bar{\lambda}_{1}$, for $\lambda$ big enough. Then by the continuity of $\mu(\alpha,\cdot)$, we can find a $T_{0}>0$ small enough and only depending on $K,\gamma,k_{1},k_{2}$ such that $\mu(\alpha,T-t)KC_{3}\leq\bar{\lambda}_{1}$ for $T-t\leq T_{0}$. If $k_{2}>0$, we choose $\alpha\in(k^{2}_{1}k^{2}_{2},1)$,
 $C_{2}=\frac{\alpha-k^{2}_{1}k^{2}_{2}}{4Kk_{2}^{2}}$ and $C_{4}=\frac{1-\alpha}{K}$. Then we have
 $\mu(\alpha,0)=\frac{KC_{2}+k_{1}^{2}}{\alpha}<\frac{1}{k_{2}^{2}}$ and $\bar{\lambda}_{1}\ge\frac{KC_{3}}{k^{2}_{2}}$, for $\lambda$ big enough.
 Consequently, there exists a $T_{0}>0$ small enough (only depending on $K,\gamma,k_{1},k_{2}$) such that $\mu(\alpha,T-t)<\frac{1}{k_{2}^{2}}$ and $\bar{\lambda}_{1}\ge\frac{KC_{3}}{k^{2}_{2}}$, for $T-t\leq T_{0}$. Reproducing the above argument, we obtain the same estimates when $(C1)$ holds and $T-t\leq T_{0}$, for some $T_{0}$ small enough.
\hfill
\end{proof}

Now we give the solvability result for FBSDE (\ref{Penalized FBSDE with two operators in appendix}).
\begin{theorem}\label{Solvability of Penalized FBSDE in appendix}
Let the assumptions $(H_{1})$-$(H_{5})$ be satisfied. We also assume $(C1)$ and either $(C2)$ or $(C3)$ hold for some $\lambda,\alpha,C_{1},C_{2},C_{3}$ and $C_{4}=\frac{1-\alpha}{K}$, Then the penalized FBSDE (\ref{Penalized FBSDE with two operators in appendix}) has a unique adapted solution $(X^{\varepsilon},Y^{\varepsilon},Z^{\varepsilon})\in S^{2}_{n}\times S^{2}_{m}\times M^{2}_{m\times d}$. Moreover, if among the compatibility conditions, only $(C1)$ holds, then there exists a constant $T_{0}>0$ small enough (depending on $K,\gamma,k_{1},k_{2}$) such that for all $t\in[0,T]$ with $T-t\leq T_{0}$, the penalized  FBSDE (\ref{Penalized FBSDE with two operators in appendix}) has a unique adapted solution on $[t,T]$.
\end{theorem}
\begin{proof}
We introduce a mapping $\Lambda:H\rightarrow H$ (Recalling the definition of $H=H[0,T]$) such that $\bar{X}_{s}^{\varepsilon}:=\Lambda( X^{\varepsilon})_{s}$ ($X^{\varepsilon}\in H$) is the unique solution of the SDE:
\[
\bar{X}_{s}^{\varepsilon}+{\displaystyle\int_{t}^{s}}\nabla\psi_{\varepsilon}(\bar{X}_{r}^{\varepsilon})dr
=\xi+{\displaystyle\int_{t}^{s}}b(r,\bar{X}_{r}^{\varepsilon},Y_{r}^{\varepsilon},Z_{r}^{\varepsilon})dr
+{\displaystyle\int_{t}^{s}}\sigma(r,\bar{X}_{r}^{\varepsilon},Y_{r}^{\varepsilon},Z_{r}^{\varepsilon})dB_{r},
\]
where $(Y^{\varepsilon}, Z^{\varepsilon})$ is the unique solution of the following BSDE
\[
Y_{s}^{\varepsilon}+{\displaystyle\int_{s}^{T}}\nabla\varphi_{\varepsilon}(Y_{r}^{\varepsilon})dr=g(X_{T}^{\varepsilon})+{\displaystyle\int
_{s}^{T}}f(r,X_{r}^{\varepsilon},Y_{r}^{\varepsilon},Z_{r}^{\varepsilon})dr-{\displaystyle\int_{s}^{T}}Z_{r}^{\varepsilon}dB_{r}.
\]

Given arbitrary $X^{\varepsilon}, \tilde{X}^{\varepsilon}\in H$, we put $\Delta X^{\varepsilon}=X^{\varepsilon}-\tilde{X}^{\varepsilon}$ and
$\Delta \bar{X}^{\varepsilon}=\Lambda(X^{\varepsilon})-\Lambda(\tilde{X}^{\varepsilon})$.
The estimates (\ref{Estimate 1})-(\ref{Estimate 4}) yield
\begin{equation}\label{contracting mapping}
\begin{array}
[c]{ll}%
\qquad e^{-\lambda T}E|\Delta \bar{X}_{T}^{\varepsilon}|^{2}+\bar{\lambda}_{1}\|\Delta \bar{X}^{\varepsilon}\|^{2}_{M_{\lambda}[t,T]}\medskip\\
\leq
\mu(\alpha,T-t)\left\{k_{2}^{2}e^{-\lambda T}E|\Delta X_{T}^{\varepsilon}|^{2}+ KC_{3}\|\Delta X^{\varepsilon}\|^{2}_{M_{\lambda}[t,T]}\right\}.
\end{array}
\end{equation}
With a similar discussion as in Theorem 3.1  \cite{MC-01} we can prove $\Lambda$ is a contracting mapping on $(M_{n}^{2}[t,T],\|\cdot\|_{M_{\lambda}[t,T]})$, if $(C1)$ and $(C2)$ hold, and $\Lambda$ is a contracting mapping on $(\bar{H}[t,T],\|\cdot\|_{\lambda,\beta,[t,T]})$, if $(C1)$ and $(C3)$ hold. Moreover, if only $(C1)$ holds and $k_{2}=0$ (resp. $k_{2}>0$), $\Lambda$ is a contracting mapping on $(M_{n}^{2}[t,T],\|\cdot\|_{M_{\lambda}[t,T]})$ (resp. $(\bar{H}[t,T],\|\cdot\|_{\lambda,\beta,[t,T]})$), for all $t\in[0,T]$ with $T-t\leq T_{0}$, $T_{0}>0$ small enough.\hfill
\end{proof}

Similar to the proof Theorem A.5 \cite{D-02}, using $\langle x_{1}-x_{2},\nabla\psi_{\varepsilon}(x_{1})-\nabla\psi_{\varepsilon}(x_{2})\rangle\ge0$ and $\langle y_{1}-y_{2},\nabla\varphi_{\varepsilon}(y_{1})-\nabla\varphi_{\varepsilon}(y_{2})\rangle\ge0$, we have the following proposition:
\begin{proposition}\label{Delarue}
Suppose $(X^{\varepsilon},Y^{\varepsilon},Z^{\varepsilon})$ (resp. $(\tilde{X}^{\varepsilon},\tilde{Y}^{\varepsilon},\tilde{Z}^{\varepsilon})$) is a solution of FBSDE (\ref{Penalized FBSDE with two operators in appendix}) with parameters $(\xi,b,\sigma,f,g)$ (resp. $(\tilde{\xi},\tilde{b},\tilde{\sigma},\tilde{f},\tilde{g})$) which satisfy $(H_{1})$-$(H_{5})$ and $(H^{\prime}_{2}),(H^{\prime}_{3})$ and $\xi,\bar{\xi}\in L^{2}(\Omega,\mathcal{F}_{t},\mathbb{P};\mathbb{R}^{n})$.
Let $1\leq p\leq \frac{3+\rho_{0}}{2}$, such that
\[
E\left\{\sup\limits_{t\leq s\leq T}|X_{s}^{\varepsilon}|^{2p}+\sup\limits_{t\leq t\leq T}|Y_{s}^{\varepsilon}|^{2p}+\left({\displaystyle\int_{t}^{T}}|Z_{s}^{\varepsilon}|^{2}ds\right)^{p}\right\}<\infty,
\]
and
\[
E\left\{\sup\limits_{t\leq s\leq T}|\tilde{X}_{s}^{\varepsilon}|^{2p}+\sup\limits_{t\leq s\leq T}|\tilde{Y}_{s}^{\varepsilon}|^{2p}+\left({\displaystyle\int_{t}^{T}}|\tilde{Z}_{s}^{\varepsilon}|^{2}ds\right)^{p}\right\}<\infty.
\]
Then there exist constants $\delta_{K,k_{2},\gamma,p}>0$ and $C_{K,k_{2},\gamma,p}$ depending only on $K,k_{2},\gamma,p$, such that, for all $t\in[0,T]$ with $T-t\leq \delta_{K,k_{2},\gamma,p}$,
\begin{equation}\label{difference in small interval}
\begin{array}
[c]{l}%
\qquad E\left\{\sup\limits_{t\leq s\leq T}|\tilde{X}_{s}^{\varepsilon}-X_{s}^{\varepsilon}|^{2p}
+\sup\limits_{t\leq s\leq T}|\tilde{Y}_{s}^{\varepsilon}-Y_{s}^{\varepsilon}|^{2p}
+\left({\displaystyle\int_{t}^{T}}|\tilde{Z}_{s}^{\varepsilon}-Z_{s}^{\varepsilon}|^{2}ds\right)^{p}\right\}\medskip\\
\leq C_{K,k_{2},\gamma,p}E\Bigg\{|\tilde{\xi}-\xi|^{2p}+|\tilde{g}-g|^{2p}(X_{T}^{\varepsilon})
+\left({\displaystyle\int_{t}^{T}}|\tilde{\sigma}-\sigma|^{2}(s,X_{s}^{\varepsilon},Y_{s}^{\varepsilon})ds\right)^{p}\medskip\\
\qquad\qquad\qquad\qquad+\left({\displaystyle\int_{t}^{T}}(|\tilde{b}-b|+|\tilde{f}-f|)
(s,X_{s}^{\varepsilon},Y_{s}^{\varepsilon},Z_{s}^{\varepsilon})ds\right)^{2p}\Bigg\}.
\end{array}
\end{equation}
\end{proposition}

\subsection{Auxiliary Propositions}
In this section, for the reader's convenience, we give two auxiliary propositions (see \cite{PR-11}).
\begin{proposition}[Remark 2.34 \cite{PR-11}]
Let $X_{t}$ be an $n$-dimensional It\^{o} process given by
\[X_{t}=X_{0}+{\displaystyle\int_{0}^{t}}b_{s}ds+{\displaystyle\int_{0}^{t}}\sigma_{s}dB_{s},~t\in[0,T],
\]
where $b,\sigma$ satisfies $E\left({\displaystyle\int_{0}^{T}}|b_{s}|ds+{\displaystyle\int_{0}^{T}}|\sigma_{s}|^{2}ds\right)<\infty$. Assume $h\in C^{1}(\mathbb{R}^{n};\mathbb{R})$ and that there exists a constant $M$ such that
\begin{equation}\label{equation in appendix 4}
|\nabla_{x}h(x+y)-\nabla_{x}h(x)|\leq M|y|, ~\text{for all } x,y\in\mathbb{R}^{n}.
\end{equation}
Then, for all $0\leq s\leq t$, $\mathbb{P}-a.s.$
\[h(X_{t})\leq h(X_{s})+{\displaystyle\int_{s}^{t}}\langle\nabla_{x}h(X_{r}),dX_{r}\rangle+\frac{1}{2}{\displaystyle\int_{s}^{t}}M|\sigma_{r}|^{2}dr.
\]
\end{proposition}

\begin{remark}\label{remark for extend ito formula}
Similarly, for $r>0$ and $h\in C^{1}(\mathbb{R}^{n};\mathbb{R}_{+})$ satisfying (\ref{equation in appendix 4}), we can prove that
\begin{equation}\label{extend ito formula}
\begin{array}
[c]{ll}
h^{1+2r}(X_{t})\leq h^{1+2r}(X_{s})+(1+2r){\displaystyle\int_{s}^{t}}h^{2r}(X_{r})\langle\nabla_{x}h(X_{r}),dX_{r}\rangle\medskip\\
\quad+(1+2r)r{\displaystyle\int_{s}^{t}}h^{2r-1}(X_{r})
|\nabla_{x}h(X_{r})|^{2}|\sigma_{r}|^{2}dr+\frac{1+2r}{2}{\displaystyle\int_{s}^{t}}h^{2r}(X_{r})M|\sigma_{r}|^{2}dr, \text{ for all }0\leq s \leq t.
\end{array}
\end{equation}
\end{remark}
\begin{proposition}[Proposition 1.25, \cite{PR-11}]\label{proposition of pardoux and rascanu 2011}
Let $(X,K)$, $(X^{n},K^{n})$, $n\ge1$, be a sequence of couples of $C([0,T];\mathbb{R}^{d})$-valued random variables such that
\[
\begin{array}
[c]{l}
(i) \quad \text{There is } ~p>0, \text{ with }\sup\limits_{n\ge1}E\updownarrow K^{n}\updownarrow^{p}_{[0,T]}<\infty,\medskip\\
(ii)\quad \sup\limits_{0\leq t\leq T}|X^{n}_{t}-X_{t}|+\sup\limits_{0\leq t\leq T}|K^{n}_{t}-K_{t}|\xrightarrow[]{\;\mathbb{P}\;}0, ~as~n\rightarrow\infty.
\end{array}
\]
Then, for all $0\leq s\leq t\leq T$,
\[
{\displaystyle\int_{s}^{t}}\langle X_{r}^{n},dK_{r}^{n}\rangle
\xrightarrow[]{\;\mathbb{P}\;}{\displaystyle\int_{s}^{t}}\langle X_{r},dK_{r}\rangle, ~as~n\rightarrow\infty,
\]
and, moreover,
\[
E\updownarrow K\updownarrow^{p}_{[0,T]}\leq \liminf\limits_{n\rightarrow\infty} E\updownarrow K^{n}\updownarrow^{p}_{[0,T]}~.
\]
\end{proposition}

\textbf{Acknowledgement}\quad The author wish to express his thanks
to Rainer Buckdahn and Aurel R\u{a}\c{s}canu for their useful
suggestions and discussions.

\end{document}